%% file: ILRFPP.tex
\documentclass[letterpaper,11pt,reqno]{amsart}
\usepackage[margin=1in]{geometry}
\usepackage{graphicx}
\usepackage[T1]{fontenc}
\usepackage{tikz}
\usetikzlibrary{arrows.meta,calc,positioning,decorations.pathreplacing,decorations.pathmorphing}

\input{style}
\usepackage{booktabs,microtype}
\newcommand{\dist}{\mathrm{dist}}
\newcommand{\LLPS}{\mathrm{LLPS}}
\newcommand{\br}{\mathsf{br}}
\newcommand{\loc}{\mathsf{loc}}
\newcommand{\Hub}{\mathsf{Hub}}

\hypersetup{
  pdftitle={Inhomogeneous Long-Range First-Passage Percolation in a Random Vertex Environment},
  pdfauthor={Shirshendu Chatterjee, Partha S. Dey, and Daecheol Kim},
  pdfsubject={Long-range first-passage percolation in a random vertex environment},
  pdfkeywords={long-range first-passage percolation, scale-free percolation, phase transition, chemical distance, explosion}
}

\definecolor{trajectoryblue}{HTML}{0075FF}
\definecolor{networkgreen}{HTML}{007B60}
\definecolor{highlightpink}{HTML}{EE65DC}
\definecolor{warmdot}{HTML}{D58A00}
\definecolor{colIH}{rgb}{0.969,0.863,0.792}
\definecolor{colII}{rgb}{0.890,0.808,0.596}
\definecolor{colIVH}{rgb}{0.914,0.659,0.533}
\definecolor{colVH}{rgb}{0.780,0.404,0.325}
\definecolor{colIE}{rgb}{0.835,0.894,0.937}
\definecolor{colIII}{rgb}{0.749,0.831,0.737}
\definecolor{colIVE}{rgb}{0.616,0.769,0.859}
\definecolor{colVE}{rgb}{0.345,0.569,0.749}

\tikzset{
  disk/.style={draw=black, line width=0.42pt, fill=black!11},
  red path/.style={draw=red!90!black, line width=0.72pt,
                   line cap=round, line join=round},
  blue path/.style={draw=trajectoryblue, line width=0.85pt,
                    line cap=round, line join=round},
  network edge/.style={draw=networkgreen, line width=0.78pt,
                       -{Stealth[length=3.2pt,width=3.0pt]},
                       line cap=round, line join=round},
  network node/.style={draw=networkgreen, fill=networkgreen!18,
                       line width=0.75pt, circle, inner sep=1.25pt},
  blue node/.style={draw=trajectoryblue, fill=trajectoryblue!16,
                    line width=0.9pt, circle, inner sep=2.0pt},
}

\begin{document}
\title[ILRFPP]{Inhomogeneous Long-Range First-Passage Percolation in a Random Vertex Environment}
\author[Chatterjee]{Shirshendu Chatterjee$^\dagger$}
\author[Dey]{Partha S.~Dey$^\star$}
\author[Kim]{Daecheol Kim$^\ddagger$}

\address{$^\dagger$Department of Mathematics, The City College of New York, New York, New York 10031}
\address{$^{\star,\ddagger}$Department of Mathematics, University of Illinois Urbana-Champaign, Urbana, Illinois 61801}
\email{$^\dagger$shirshendu@ccny.cuny.edu, \{$^\star$psdey,$^\ddagger$dk43\}@illinois.edu}
\date{}
\begin{abstract}
	We study an inhomogeneous long-range first-passage percolation model on $\dZ^d$ in which every pair of vertices is joined by an edge whose passage time is
	$\norm{\mvx-\mvy}^{\ga}\go_{\mvx\mvy}/(V_{\mvx}V_{\mvy})$,
	where $(V_{\mvx})_{\mvx\in\dZ^d}$ are i.i.d.~positive vertex weights with polynomial upper-tail exponent $\gc$, $(\go_{\mvx\mvy})$ are i.i.d.~nonnegative edge noises with polynomial lower-tail exponent $\theta$ at zero, and $\alpha>0$ is a fixed constant. The model interpolates between long-range first-passage percolation and scale-free percolation. The two sources of randomness compete with each other. An atypically heavy vertex acts as a \emph{hub} that discounts \emph{all} of its incident edges simultaneously, whereas an atypically small edge noise produces an isolated \emph{bridge} between a specific pair of vertices.
	We conjecture that this competition produces a phase diagram in the $(\gc,\ga)$--plane with eight regimes, organized into five growth phases separated by the thresholds
	$q_{\rm hub}:=d/\gc$ and $q_{\rm edge}:=d/\theta$. Writing $T_n:=T(\mvzero,\lceil n\mvx\rceil)$ for fixed $\mvx\in\dR^d$, we prove upper bounds of the conjectured order in every regime, together with matching lower bounds in phases I and II. In the instantaneous phase, $T_n=0$ almost surely when $\ga<q_{\rm hub}\vee q_{\rm edge}$; in the tight phase, $T_n=\gTh_{\pr}(1)$ when $q_{\rm hub}\vee q_{\rm edge}<\ga<2q_{\rm hub}$. In the edge-dominated intermediate regime, $T_n=O_{\pr}((\log n)^{\Delta_{\rm III}+\eps})$ for every $\eps>0$, where $\Delta_{\rm III}=\log2/\log(2q_{\rm edge}/\ga)$. The two power-law regimes satisfy the sublinear upper bounds $T_n=O_{\pr}(n^{\ga-2q_{\rm hub}+\eps})$ and $T_n=O_{\pr}(n^{\ga-2q_{\rm edge}+\eps})$, respectively, and the conjectured linear regime satisfies $T_n=O_{\pr}(n)$.
	All upper bounds are constructive. Moreover, we isolate two explicit path architectures---a \emph{hub-chain ansatz}, which forms a non-branching chain by retaining several high-weight vertices at each scale and selecting connecting edges by their noises, and a \emph{binary edge-bridge multiscale ansatz}, which bridges a segment by a small-noise edge between two shells and then recurses on the two remaining gaps.
\end{abstract}

\maketitle

\setcounter{tocdepth}{1}\tableofcontents

\section{Introduction and Main Results}\label{sec:intro}
\subsection{The model}\label{ssec:model}
Let $G$ be a simple connected graph with vertex set $\cV$, and let $\dist(\mvu,\mvv)$ denote the graph distance. The \emph{Inhomogeneous Long-Range First-Passage Percolation} (ILRFPP) model is built from three ingredients.

\begin{enumeratei}
	\item \textbf{Vertex weights.} Every vertex $\mvu\in\cV$ carries a positive random weight $V_{\mvu}$, and the family of random variables $(V_{\mvu})_{\mvu\in\cV}$ is i.i.d.~with common law $V$. One should think of $V_{\mvu}$ as quantifying the \emph{importance}, activity, or degree-propensity of $\mvu$: an airport's traffic volume, a farm's livestock density, an individual's contact rate.

	\item \textbf{Edge noises.} For every unordered pair $\{\mvu,\mvv\}$ of distinct vertices there is a nonnegative random variable $\go_{\mvu\mvv}=\go_{\mvv\mvu}$, and the family of random variables $(\go_{\mvu\mvv})$ is i.i.d.~with common law $\go$, independent of $(V_{\mvu})$. One may view $\go_{\mvu\mvv}$ as the idiosyncratic resistance of the direct connection---for example, weather, congestion, or another pair-specific obstruction. An atypically small $\go_{\mvu\mvv}$ creates a \emph{shortcut}, or \emph{bridge}, which circumvents the spatial cost of the pair $\{\mvu,\mvv\}$ but of no other pair. Note that $\mvu$ and $\mvv$ need not be adjacent in $\cG$: all pairs are available.

	\item \textbf{Cost function.} A nonnegative function $\cost:\dN\to[0,\infty]$ encodes the penalty for a direct jump between vertices at graph distance $r$.
\end{enumeratei}

For $\mvu,\mvv\in\cV$ the cost, or passage time, of the direct jump from $\mvu$ to $\mvv$ is
\begin{align}\label{eq:def-W}
	W_{\mvu\mvv}
	:=\cost\big(\dist(\mvu,\mvv)\big)\cdot\frac{1}{V_{\mvu}V_{\mvv}}\cdot\go_{\mvu\mvv},
	\qquad \mvu\ne\mvv.
\end{align}
A finite \emph{path} $\mvpi=\la\mvv_0\mvv_1\cdots\mvv_m\ra$ is a sequence of distinct vertices, and its passage time is $W_{\mvpi}:=\sum_{e\in\mvpi}W_e$. The \emph{first-passage time} between $\mvx$ and $\mvy$ is
\begin{align}\label{eq:def-T}
	T(\mvx,\mvy):=\inf_{\mvpi\in\cP_{\mvx,\mvy}}W_{\mvpi},
\end{align}
where $\cP_{\mvx,\mvy}$ is the set of finite self-avoiding paths from $\mvx$ to $\mvy$. We write $\abs{\mvpi}$ for the number of edges of $\mvpi$, the \emph{hop count}. The \emph{growth set} and its diameter, respectively, are
\begin{align}\label{eq:def-ball}
	\vB_t(\mvx):=\{\mvy\in\cV:T(\mvx,\mvy)\le t\},
	\text{ and }
	\vD_t(\mvx):=\sup\{\dist(\mvy,\mvz):\mvy,\mvz\in\vB_t(\mvx)\}\in[0,\infty].
\end{align}

The multiplicative structure in~\eqref{eq:def-W} is what makes the model interesting. The factor $1/(V_{\mvu}V_{\mvv})$ is a \emph{shared} discount: a single exceptional vertex simultaneously cheapens every one of the infinitely many edges incident to it, so exceptional vertices act as \emph{routers} or \emph{influencers}. The factor $\go_{\mvu\mvv}$ is a \emph{private} discount attached to one pair only.
Taking $V\equiv1$ gives long-range first-passage percolation, whereas taking $\go\equiv1$ gives a purely vertex-weighted first-passage model.
For exponential edge noises and power-law costs, thresholding the passage times gives the exponential connection rule of scale-free percolation; see Remark~\ref{rem:LLPS-comparison}.
The interest of the general model lies in the competition between the two mechanisms, and in the fact that the resulting phase diagram is genuinely two-dimensional: neither mechanism dominates throughout.

A motivating example is a spatial-interaction or individual-level epidemic model whose pairwise transmission rates separate endpoint activity from spatial friction
\cite{Wilson1971SpatialInteraction,Deardon2010IndividualLevel}:
\begin{align}\label{eq:applied-rate}
	\lambda_{\mvu\mvv}
	=
	A_{\mvu}B_{\mvv}
	K\bigl(\dist(\mvu,\mvv)\bigr),
	\qquad
	K(r)\asymp r^{-\ga}.
\end{align}
Here, $A_{\mvu}$ captures source infectivity, $B_{\mvv}$ target susceptibility, and $K$ the spatial friction. If a transmission with rate $\lambda_{\mvu\mvv}$ is driven by an independent exponential clock $E_{\mvu\mvv}$, its waiting time is
\begin{align*}
	\frac{E_{\mvu\mvv}}{\lambda_{\mvu\mvv}}
	\asymp
	\dist(\mvu,\mvv)^{\ga}
	\frac{E_{\mvu\mvv}}{A_{\mvu}B_{\mvv}}.
\end{align*}

The ILRFPP model is a symmetric, random-metric abstraction of this dynamic. It sets $A_{\mvu}=B_{\mvu}=V_{\mvu}$ and replaces the exponential clock with a general pair-specific noise $\go_{\mvu\mvv}$.
This abstraction naturally captures settings where spatial separation, endpoint heterogeneity, and random link quality all compete. Farm-level epidemics and superspreading rely on heterogeneous activity marks~\cite{Keeling2001FMD,LloydSmith2005Superspreading}. Human-mediated dispersal distinguishes source and destination effects~\cite{TakahashiPark2020SourceDestination}. Similarly, transportation-driven cascades motivate long-range, random arrival-time metrics~\cite{HufnagelBrockmannGeisel2004GlobalEpidemics,Balcan2009MultiscaleMobility,BrockmannHelbing2013EffectiveDistance,GomezRodriguezBalduzziScholkopf2011NetRate}. We do not claim that these applied models share an identical phase diagram. Rather, ILRFPP isolates the fundamental structural question of whether fast transport is governed by reusable endpoint hubs or by rare pair-specific shortcuts.

\subsection{Assumptions}\label{ssec:assumptions}
We work under the following assumptions throughout this article.

\begin{enumerate}[label=\textup{(A\arabic*)},ref=A\arabic*,leftmargin=*,itemsep=1ex]
	\item \label{A1}\textbf{Base graph.} We take $\cV=\dZ^d$ with $d\ge1$, and we let
	      \begin{align*}
		      \sE_{\mathrm{com}}:=\big\{\la\mvx\mvy\ra:\mvx,\mvy\in\dZ^d, \mvx\ne\mvy\big\}
	      \end{align*}
	      be the complete edge set. The lattice supplies polynomial volume growth of exponent $d$: there is $c\ge1$ with
	      \begin{align*}
		      c^{-1}\le r^{-d} \big|\{\mvy\in\cV: \dist(\mvx,\mvy)\le r\}\big|\le c
		      \text{ for all }\mvx\in\cV, r\ge1.
	      \end{align*}
	      We use the $\ell_1$ distance
	      $\norm{\mvx-\mvy}:=\sum_{i\le d}\abs{x_i-y_i}$ in the constructions and the
	      $\ell_\infty$ distance $\norm{\cdot}_\infty$ when convenient. The two are
	      equivalent on $\dZ^d$ and every statement below is insensitive to this choice.
	      Several estimates use only polynomial volume growth, but the path constructions
	      also use lattice geometry. Extensions to other vertex-transitive graphs of
	      polynomial growth therefore require a suitable replacement for the Euclidean
	      placement of balls and corridors; see Section~\ref{sec:discussion}.

	\item \label{A2} \textbf{Cost function.} Fix $\ga>0$ and take $\cost(r)=r^{\ga}$, so that
	      \begin{align}\label{eq:W-explicit}
		      W_{\mvx\mvy}=\frac{\norm{\mvx-\mvy}^{\ga} \go_{\mvx\mvy}}{V_{\mvx}V_{\mvy}}.
	      \end{align}
	      The proofs below are written for the exact power law $\cost(r)=r^\ga$. Many exponent-level estimates remain valid for $\cost(r)=r^{\ga+o(1)}$ away from the phase boundaries, with a correspondingly small loss in the stated exponent. At a boundary, however, a slowly varying factor in $\cost(r)=r^\ga L(r)$ can affect the answer and must be retained.

	\item \label{A3}\textbf{Vertex weights.} Fix $\gc>0$. The law of $V$ satisfies
	      \begin{align}\label{def:V_x}
		      \pr(0<V<\infty)=1
		      \text{ and }
		      \lim_{x\to\infty}\frac{\log\pr(V\ge x)}{\log x}=-\gc.
	      \end{align}
	      Equivalently, $\pr(V\ge x)=L_V(x)x^{-\gc}$ for a function $L_V$ with $\log L_V(x)/\log x\to0$.

	\item \label{A4}\textbf{Edge noises.} Fix $\theta>0$ and $\gz>0$. The law of $\go$ satisfies
	      \begin{align}\label{def:go}
		      \lim_{x\downarrow0}\frac{\log\pr(\go\le x)}{\log x}=\theta
		      \text{ and }
		      \E\go^{\gz}<\infty.
	      \end{align}
	      Equivalently, $\pr(\go\le x)=L_\go(1/x)x^{\theta}$ near $0$ for a function $L_\go$ with $\log L_\go(y)/\log y\to0$ as $y\to\infty$, and $\go$ has a finite moment of some positive order. As in \textup{\eqref{A3}}, this is weaker than regular variation.
\end{enumerate}
The main theorems below are proved under \eqref{A1}--\eqref{A4}. See Section~\ref{ssec:disc-assumptions} for a discussion of these assumptions and Remark~\ref{rem:IVH-lossless} for a refinement under a stronger vertex-tail bound.

We will use the following conventions: $c,C,c_1,C_1,\dots$ denote finite positive constants that may change from line to line and depend only on $d,\ga,\gc,\theta,\gz$ and the laws of $V,\go$ unless indicated. Constants that depend on additional parameters carry them as arguments.

\subsection{Conjectured phase diagram and summary of the main results}\label{ssec:summary}

Set
\begin{align}\label{eq:def-q}
	q_{\rm hub}:=\frac{d}{\gc}
	\text{ and }
	q_{\rm edge}:=\frac{d}{\theta}.
\end{align}
These are the two \emph{competing exponents}: $q_{\rm hub}$ is the exponent governing how cheaply a path can reach a distant \emph{hub}, whereas $q_{\rm edge}$ governs how cheaply it can find a distant \emph{bridge}. Section~\ref{ssec:heuristics} derives both. The five conjectured growth phases are
\begin{align}\label{eq:regime-list}
	\begin{aligned}
		\textbf{I}   & :  \ga<q_{\rm hub}\vee q_{\rm edge},
		             & \qquad
		\textbf{II}  & :  q_{\rm hub}\vee q_{\rm edge}<\ga<2q_{\rm hub},                              \\
		\textbf{III} & :  2q_{\rm hub}\vee q_{\rm edge}<\ga<2q_{\rm edge},
		             & \qquad
		\textbf{IV}  & :  2(q_{\rm hub}\vee q_{\rm edge})<\ga<1+2(q_{\rm hub}\vee q_{\rm edge}),      \\
		\textbf{V}   & :  \ga>1+2(q_{\rm hub}\vee q_{\rm edge}).                                 &  &
	\end{aligned}
\end{align}
Phases I, IV and V split further according to which mechanism is responsible, and we append a subscript ${\rm H}$ (hub-dominated, $\gc<\theta$) or ${\rm E}$ (edge-dominated, $\gc>\theta$). This gives the eight regimes
\begin{align*}
	\rm I_H, I_E, II, III, IV_H, IV_E, V_H, V_E,
\end{align*}
drawn in Figure~\ref{fig:phase}. Throughout, we fix a direction $\mvx\in B(\mvzero,1)\setminus\{\mvzero\}$ and write
\begin{align}\label{eq:def-Tn}
	T_n:=T\big(\mvzero,\lceil n\mvx\rceil\big).
\end{align}

\begin{figure}[htbp]
	\centering
	\begin{tikzpicture}[
			x=108bp, y=45bp,scale=.75,
			line join=round, line cap=round,
			region/.style={draw=black, line width=0.8bp},
			curve/.style={samples=120, smooth},
			axis/.style={draw=black, line width=1.2bp,
						-{Triangle[length=7bp,width=4.7bp]}},
			tick/.style={draw=black, line width=0.4bp},
		]

		\def\xmax{3.45}   
		\def\ymax{4}      

		\filldraw[region, fill=colIH]
		(0,0) -- (1,0) -- (1,1)
		-- plot[curve, domain=1:0.25] (\x,{1/\x})
		-- (0,\ymax) -- cycle;

		\filldraw[region, fill=colII]
		(0.5,\ymax)
		-- plot[curve, domain=0.5:2] (\x,{2/\x})
		-- (1,1)
		-- plot[curve, domain=1:0.25] (\x,{1/\x})
		-- cycle;

		\filldraw[region, fill=colIVH]
		(0.5,\ymax)
		-- plot[curve, domain=0.5:1] (\x,{2/\x})
		-- (1,3)
		-- plot[curve, domain=1:0.6666667] (\x,{1+2/\x})
		-- cycle;

		\filldraw[region, fill=colVH]
		(0.6666667,\ymax)
		-- plot[curve, domain=0.6666667:1] (\x,{1+2/\x})
		-- (1,\ymax) -- cycle;

		\filldraw[region, fill=colIE] (1,0) rectangle (\xmax,1);

		\filldraw[region, fill=colIII]
		(1,2) -- (\xmax,2) -- (\xmax,1) -- (2,1)
		-- plot[curve, domain=2:1] (\x,{2/\x}) -- cycle;

		\filldraw[region, fill=colIVE] (1,2) rectangle (\xmax,3);

		\filldraw[region, fill=colVE] (1,3) rectangle (\xmax,\ymax);

		\draw[axis] (0,0) -- (3.5241,0);
		\draw[axis] (0,0) -- (0,4.1778);

		\foreach \tx in {1,2,3}{\draw[tick] ([yshift=-3bp]\tx,0) -- ([yshift=3bp]\tx,0);}
		\foreach \ty in {1,2,3}{\draw[tick] ([xshift=-3bp]0,\ty) -- ([xshift=3bp]0,\ty);}

		\node[font=\Large] at (0.55,0.70) {$\mathrm{I}_{\mathrm{H}}$};
		\node[font=\Large] at (2.40,0.55) {$\mathrm{I}_{\mathrm{E}}$};
		\node[font=\Large] at (0.72,2.05) {$\mathrm{II}$};
		\node[font=\Large] at (2.40,1.55) {$\mathrm{III}$};
		\node[font=\Large] at (0.86,2.93) {$\mathrm{IV}_{\mathrm{H}}$};
		\node[font=\Large] at (2.40,2.55) {$\mathrm{IV}_{\mathrm{E}}$};
		\node[font=\Large] at (0.88,3.74) {$\mathrm{V}_{\mathrm{H}}$};
		\node[font=\Large] at (2.40,3.55) {$\mathrm{V}_{\mathrm{E}}$};

		\node[font=\huge, anchor=east]  at (-0.1111,3.9111) {$\alpha$};
		\node[font=\huge, anchor=north] at (3.4685,-0.1333) {$\gamma$};

		\node[font=\large, anchor=north east] at (-0.0556,-0.1333) {$0$};
		\node[font=\large, anchor=north]      at (1,-0.1556) {$\theta$};
		\node[font=\large, anchor=north]      at (2,-0.1556) {$2\theta$};
		\node[font=\large, anchor=north]      at (3,-0.1556) {$3\theta$};

		\node[font=\large, anchor=east] at (-0.1296,1) {$\displaystyle\frac{d}{\theta}$};
		\node[font=\large, anchor=east] at (-0.1296,2) {$\displaystyle\frac{2d}{\theta}$};
		\node[font=\large, anchor=east] at (-0.1296,3) {$\displaystyle 1+\frac{2d}{\theta}$};
	\end{tikzpicture}
	\caption{The conjectured phase diagram in the $(\gc,\ga)$-plane. The horizontal lines $\ga=d/\theta, 2d/\theta, 1+2d/\theta$ are the edge-mechanism thresholds. The hyperbolas $\ga=d/\gc, 2d/\gc, 1+2d/\gc$ are the hub-mechanism thresholds. The vertical line $\gc=\theta$ separates the hub-dominated half plane ($\gc<\theta$, subscript ${\rm H}$) from the edge-dominated one ($\gc>\theta$, subscript ${\rm E}$). Note that regime II protrudes into $\{\gc>\theta\}$ as far as $\gc=2\theta$: for $\theta<\gc<2\theta$ and $d/\theta<\ga<2d/\gc$ the hub mechanism still wins, even though the vertex weights have a lighter tail than the edge noises.}
	\label{fig:phase}
\end{figure}
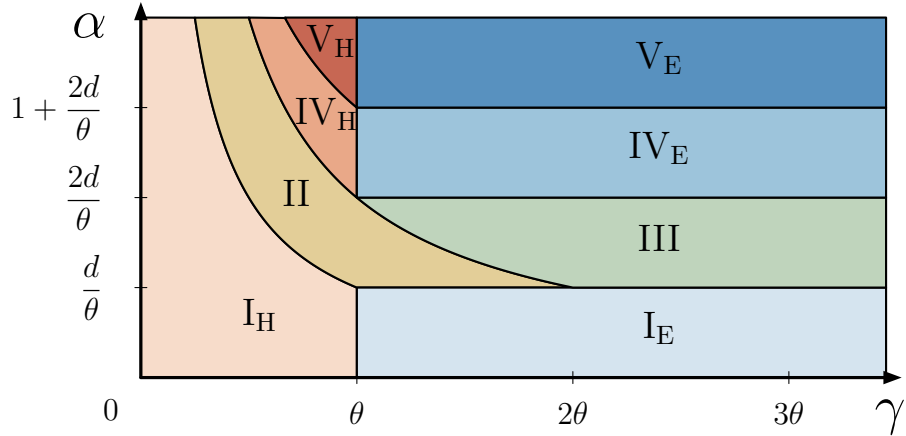

The results of this article are upper bounds in each of the eight regimes, of the order predicted by the heuristics of Section~\ref{ssec:heuristics}. Matching lower bounds outside phases I and II are deferred to a companion manuscript~\cite{CDK26} in preparation. Here we include the short lower-bound arguments for regimes $\rm I$ and $\rm II$. Table~\ref{table:summary} summarizes the results proved in this article. Precise statements are presented in Sections~\ref{sec:mainresv} and~\ref{sec:mainrese}.
Here
\begin{align}\label{eq:def-DeltaIII}
	\Delta_{\rm III}
	:=\frac{\log2}{\log\big(2d/(\ga\theta)\big)}
	=\frac{\log 2}{\log(2q_{\rm edge}/\ga)}
	\in(1,\infty)
\end{align}
in Regime III, $H_{\mvx}$ is the hop-count profile of~\eqref{eq:def-H}, and $\tau_\infty:=\inf\{t\ge0:\abs{\vB_t(\mvzero)}=\infty\}$ is the explosion time.

\begin{table}[htbp]
	\begin{center}
		\renewcommand{\arraystretch}{2}
		\begin{tabular}{@{}llll@{}}
			\toprule
			Region        & Parameter range                                                           & Result proved in this article                                        & Reference                            \\
			\midrule
			$\rm I_H$     & $\ga<\frac{d}{\gc}$                                                       & $T(\mvzero,\mvx)=0$ a.s.; $H_{\mvx}\equiv2$ on $(0,W_{\mvzero\mvx})$ & Thms.~\ref{thm:IH},~\ref{thm:IH-hop} \\
			$\rm I_E$     & $\ga<\frac{d}{\theta}$                                                    & $T(\mvzero,\mvx)=0$ a.s.                                             & Thm.~\ref{thm:IE}                    \\
			$\rm II$      & $\frac{d}{\gc}\vee\frac{d}{\theta}<\ga<\frac{2d}{\gc}$
			              & $T_n=\gTh_{\pr}(1)$; $\tau_\infty<\infty$; $\abs{\mvpi}\gtrsim\log\log n$
			              & Thms.~\ref{thm:II-ub},~\ref{thm:II-lb},~\ref{thm:II-hop}                                                                                                                                \\
			$\rm III$     & $\frac{2d}{\gc}\vee\frac{d}{\theta}<\ga<\frac{2d}{\theta}$
			              & $T_n=O_{\pr}\big((\log n)^{\Delta_{\rm III}+\eps}\big)$                   & Thm.~\ref{thm:III-ub}                                                                                       \\
			$\rm IV_H$    & $\frac{2d}{\gc}<\ga<1+\frac{2d}{\gc}$, $\gc<\theta$
			              & $T_n=O_{\pr}\big(n^{\ga-2d/\gc+\eps}\big)$                                & Thm.~\ref{thm:IVH-ub}                                                                                       \\
			$\rm IV_E$    & $\frac{2d}{\theta}<\ga<1+\frac{2d}{\theta}$, $\gc>\theta$
			              & $T_n=O_{\pr}\big(n^{\ga-2d/\theta+\eps}\big)$                             & Thm.~\ref{thm:IVE-ub}                                                                                       \\
			$\rm V_H,V_E$ & $\ga>1+\frac{2d}{\gc\wedge\theta}$
			              & $T_n=O_{\pr}(n)$                                                          & Thm.~\ref{thm:linear-ub}                                                                                    \\[2ex]
		\end{tabular}
		\caption{The rows labeled $\rm I_H$ and $\rm I_E$ refer to the corresponding sides
			$\gc<\theta$ and $\gc>\theta$ of the diagram. The theorems themselves extend
			into the overlap, as Remarks~\ref{rem:IH} and~\ref{rem:IE} explain.}
		\label{table:summary}
	\end{center}
\end{table}

\begin{rem}[Removing the exponent loss in $\rm IV_H$]\label{rem:IVH-lossless}
	Suppose in addition that
	\begin{align}\label{def:V-strong}
		\pr(V\ge x)\ge c_Vx^{-\gc}
		\text{ for all }x\ge x_V,
	\end{align}
	for some $c_V>0$ and $x_V\ge1$. Then $T_n=O_{\pr}(n^{\ga-2d/\gc})$ under the assumptions of Theorem~\ref{thm:IVH-ub}. The exponent loss comes from using $\tilde\gc>\gc$ in the hub normalization. This bound allows $\tilde\gc=\gc$, and the same summable-chain estimates apply. See the proof of Theorem~\ref{thm:IVH-ub} in Section~\ref{sec:proof-IVH}.
\end{rem}

Two features of the table deserve emphasis. The first is the \emph{asymmetry} between II and III. The two parameter windows $\ga<2q_{\rm hub}$ and $\ga<2q_{\rm edge}$ look symmetric, and the one-bridge heuristics of Section~\ref{ssec:heuristics} produce them by identical computations. Yet the hub construction yields a tight family $T_n=\gTh_{\pr}(1)$ in Regime II, while the edge construction yields the polylogarithmic upper bound of Theorem~\ref{thm:III-ub}. Section~\ref{ssec:dichotomy} explains the difference: a hub discounts all of its edges at once and therefore supports a non-branching \emph{chain}, whereas a bridge discounts a single pair and therefore forces the construction to \emph{bifurcate}. The entropy of the resulting binary tree is precisely the $\log2$ in~\eqref{eq:def-DeltaIII}.
The second is the \emph{protrusion} of regime II past $\gc=\theta$ visible in Figure~\ref{fig:phase}. The label ``hub-dominated'' does not simply mean $\gc<\theta$. For $\theta<\gc<2\theta$ and $d/\theta<\ga<2d/\gc$ the vertex weights have the lighter tail, and yet the hub-chain construction still beats the bridge construction, because a chain is cheaper than a tree even when its individual steps are more expensive. The correct dividing line between the two mechanisms is $\ga=2q_{\rm hub}$ inside $\{\ga>q_{\rm edge}\}$, not $\gc=\theta$.

\subsection{Heuristics behind the phase boundaries}\label{ssec:heuristics}

All six thresholds in~\eqref{eq:regime-list} can be located heuristically by an extreme-value computation performed separately for the two mechanisms. We now present this heuristic.

\subsubsection*{The edge mechanism ($\gc>\theta$; freeze $V\asymp1$)}
First, we present the edge computation by freezing $V_{\mvx}\asymp1$. The hub computation is the mirror image with $\theta$ replaced by $\gc$.

\begin{enumerate}[label=\textup{(E\arabic*)},ref=E\arabic*,leftmargin=*,itemsep=1ex]
	\item \label{E1} \textbf{The threshold $q_{\rm edge}=d/\theta$ and point-to-bulk escape.}
	      To jump from a fixed vertex to \emph{somewhere} in a target block of volume $\asymp\ell^d$ at distance $\ell$, there are $\asymp\ell^d$ candidate edges. The logarithmic tail assumption gives $\pr(\go\le t)=t^{\theta+o(1)}$ as $t\downarrow0$, so the smallest of $\ell^d$ i.i.d.~copies of $\go$ has logarithmic order $\ell^{-d/\theta+o(1)}$. Suppressing sub-polynomial factors in this heuristic, we obtain
	      \begin{align}\label{eq:heur-escape}
		      W_{\rm escape}(\ell)\asymp\ell^{\ga}\min_{1\le j\le\ell^d}\go_j\asymp\ell^{\ga-d/\theta}.
	      \end{align}
	      If $\ga<d/\theta$ this cost vanishes as $\ell\to\infty$: one can escape arbitrarily far at arbitrarily small cost. The second-moment argument in Section~\ref{sec:proof-IE} turns this heuristic into $T\equiv0$. This is Regime $\rm I_E$.

	\item \textbf{The threshold $2q_{\rm edge}=2d/\theta$ and bulk-to-bulk bridging.}
	      Suppose $\ga>d/\theta$, so that a single vertex can no longer escape for free, and consider instead bridging two disjoint blocks of volume $\asymp\ell^d$ at distance $\asymp\ell$. There are now $N(\ell)\asymp\ell^{2d}$ candidate edges, and if $M_\ell$ denotes the minimum of their noises,
	      \begin{align*}
		      \pr(M_\ell>t)\approx\big(1-c t^{\theta}\big)^{\ell^{2d}}
		      \approx\exp\big\{-c \ell^{2d}t^{\theta}\big\},
	      \end{align*}
	      so $M_\ell\asymp\ell^{-2d/\theta}$ and the cheapest available bridge at scale $\ell$ costs
	      \begin{align}\label{eq:heur-bridge}
		      W_{\rm bridge}(\ell)\asymp\ell^{\ga}M_\ell\asymp\ell^{\ga-2d/\theta}.
	      \end{align}
	      The exponent changes sign at $\ga=2d/\theta$. If $\ga<2d/\theta$, longer bridges are individually \emph{cheaper}. A naive chain of $n/\ell$ equal hops would then cost $n\ell^{\ga-2d/\theta-1}$, which decreases in $\ell$ and suggests using a single macroscopic bridge. But after taking that bridge one is left with the \emph{same} problem at two smaller scales, and the recursion -- not the chain -- is what actually governs the cost. This is the origin of the multi-scale binary tree of Section~\ref{ssec:multi-scale-ansatz} and of Regime III.

	\item \textbf{The threshold $1+2q_{\rm edge}$ and path shattering.}
	      Once $\ga>2d/\theta$, the heuristic bridge cost increases with length and suggests a chain of hops $e_1,e_2,\dots$ with $\sum_i\norm{e_i}\ge n$ and effective cost $\sum_i\norm{e_i}^{\Delta}$, where $\Delta:=\ga-2d/\theta>0$. The resulting optimization is governed by the convexity of $x\mapsto x^{\Delta}$:
	      \begin{itemize}
		      \item If $0<\Delta<1$, \ie\ $2d/\theta<\ga<1+2d/\theta$, the map is concave and subadditive, so $\sum_i\norm{e_i}^{\Delta}\ge(\sum_i\norm{e_i})^{\Delta}\ge n^{\Delta}$, with equality for a single macroscopic leap. Fragmentation is penalized, suggesting $T_n$ of order $n^{\ga-2d/\theta}$. This is Regime $\rm IV_E$.
		      \item If $\Delta>1$, \ie\ $\ga>1+2d/\theta$, the map is convex and
		            $\sum_i\norm{e_i}^{\Delta}=\sum_i\norm{e_i}\cdot\norm{e_i}^{\Delta-1}\ge n\min_i\norm{e_i}^{\Delta-1}$,
		            which is minimized by shattering every jump into microscopic steps, $\norm{e_i}\asymp1$. This \emph{path shattering} predicts classical linear growth for $T_n$ in Regime $\rm V_E$.
	      \end{itemize}
\end{enumerate}

\subsubsection*{The hub mechanism ($\gc<\theta$; freeze $\go\asymp1$)}
The same three computations with $V$ in place of $\go$ give the hyperbolic boundaries. Since $\pr(V\ge x)=x^{-\gc+o(1)}$, the largest of $\ell^d$ weights in a block has logarithmic order $\ell^{d/\gc+o(1)}$. Again suppressing sub-polynomial factors, the cost of jumping from a typical vertex to the best hub of a block at distance $\ell$ is $\ell^{\ga-d/\gc}$, giving the threshold $q_{\rm hub}=d/\gc$ and Regime $\rm I_H$. The cost of a hop between the best hubs of two blocks at distance $\ell$ is
\begin{align}\label{eq:heur-hub-hop}
	\frac{\ell^{\ga}}{\ell^{d/\gc}\cdot\ell^{d/\gc}}=\ell^{\ga-2d/\gc},
\end{align}
giving the thresholds $2q_{\rm hub}$ and, by the same convexity dichotomy, $1+2q_{\rm hub}$.

\medskip
The point of the phase diagram is that these two families of thresholds are \emph{not} interchangeable, even though the computations that produce them are. That is the content of the next subsection.

\subsection{The geometric dichotomy: chains versus binary trees}\label{ssec:dichotomy}

To contrast Regimes II and III, note that in Regime II we have $\ga<2q_{\rm hub}$, so a hop between the best hubs at scale $\ell$ costs $\ell^{\ga-2d/\gc}\to0$ by~\eqref{eq:heur-hub-hop}. In Regime III we have $\ga<2q_{\rm edge}$, so a bridge at scale $\ell$ costs $\ell^{\ga-2d/\theta}\to0$ by~\eqref{eq:heur-bridge}. Although individual step costs decay in both cases as the scale grows, the constructions lead to different proved upper bounds: $T_n=\gTh_{\pr}(1)$ in Regime II and $T_n=O_{\pr}((\log n)^{\Delta_{\rm III}+\eps})$ in Regime III. This difference stems from the distinct topological structures of the two constructed path families.

\begin{itemize}[leftmargin=*, itemsep=1ex]
	\item \textbf{Two-sided hub discounts and support chains.}
	      The discount $1/V_{\mvx}$ is attached to the \emph{vertex} $\mvx$, meaning an exceptionally heavy hub simultaneously discounts both its incoming and outgoing edges. The proof exploits this via a non-branching chain
	      \begin{align*}
		      \mvzero\to\mvx_{-k}\to\cdots\to\mvx_{-1}\to\mvx_0\to\mvx_1\to\cdots\to\mvx_k\to\lceil n\mvx\rceil,
	      \end{align*}
	      visiting one selected hub per scale $n_i$.
	      To control large edge noises, we do not simply route through the heaviest vertex in each ball. Instead, we retain the $\ell$ heaviest candidates and select the hub minimizing the connecting edge noise. Since $\E\go^{\gz}<\infty$, fixing an integer $\ell>2/\gz$ ensures this minimum $\go_{[\ell]}$ has a finite second moment (Lemma~\ref{lem:min-moments}), while the selected hub weight remains asymptotically as large as the maximum. With $O(1)$ hops per scale, the scale map and hub truncation can be chosen so that the deterministic chain-cost coefficients are summable uniformly in $n$ whenever $\ga<2d/\gc$. This is the Hub-Chain Ansatz in Section~\ref{ssec:Ansatz_II_IVH}. See Figure~\ref{fig:hub-ansatz}.

	\item \textbf{Pair-specific bridge discounts and binary recursion.}
	      The noise $\go_{\mvu\mvv}$ is attached to a \emph{pair} $\{\mvu,\mvv\}$. To utilize a cheap macroscopic bridge, the path must resolve two remaining routing problems: reaching $\mvu$, and continuing from $\mvv$. Thus, bridging one gap creates two smaller gaps, requiring $2^i$ distinct bridges at depth $i$.

	      This binary structure governs the multiscale search, while the actual path remains a single non-branching sequence. Balancing the geometric decay of individual bridge costs against this $2^i$ entropy penalty truncates the recursion at depth $k\asymp\log\log n$. The total cost is dominated by the $2^k\asymp(\log n)^{\log 2/\log(1/\phi)}$ terminal segments. Optimizing the multiscale ratio $\phi$ over the admissible range gives the exponent $\Delta_{\rm III}$ in~\eqref{eq:def-DeltaIII}, with the factor $\log 2$ originating directly from the binary recursion. See Figure~\ref{fig:binary-multiscale-ansatz}.
\end{itemize}

This structural dichotomy explains the extension of Regime II into the $\{\gc>\theta\}$ regime noted after the summary table. A non-branching chain of slightly more expensive edges can yield a tighter bound than a binary tree of cheaper ones by avoiding the exponential entropy penalty. Thus, the hub construction remains optimal in the parameter wedge $\theta<\gc<2\theta$ and $d/\theta<\ga<2d/\gc$.

\subsection{Related work}\label{ssec:related}

The model sits at the confluence of three models of percolation, namely, long-range percolation, long-range first-passage percolation, and scale-free percolation. Next, we briefly discuss relevant results in these three topics.

\medskip
\noindent\textbf{Long-range percolation.}
Setting $V\equiv1$ and thresholding the passage times recovers long-range percolation on $\dZ^d$, where an edge $\la\mvx\mvy\ra$ is present with probability $\asymp\norm{\mvx-\mvy}^{-s}$ for $s=\ga\theta$.
Benjamini and Berger~\cite{BenjaminiBerger2001Cycle} and Coppersmith, Gamarnik, and Sviridenko~\cite{CoppersmithGamarnikSviridenko2002Diameter} established diameter bounds on finite cycles and boxes.
For $d<s<2d$, Biskup~\cite{Biskup2004Chemical} identified the exponent $\Delta=\log2/\log(2d/s)$ for typical chemical distances on the infinite cluster, subsequently establishing it for box diameters~\cite{Biskup2011Diameter}. Biskup and Lin~\cite{BiskupLin2017Sharp} obtained two-sided constant-factor bounds on the lattice and sharper asymptotics in the continuum. Biskup and Krieger~\cite{BiskupKrieger2021Oscillations} identified a log-log-periodic normalization of typical distances outside a countable exceptional set of intensities, and proved that the periodic factor is nonconstant for sufficiently large intensities.
For $s>2d$, Berger~\cite{Berger2004LowerBound} proved a linear lower bound.

At the metric threshold $s=2d$, polynomial distance growth was established by Ding and Sly~\cite{DingSly2013Critical} in $d=1$ and B\"aumler~\cite{Baeumler2022CriticalAllD} in all dimensions. Ding, Fan, and Huang proved convergence to a unique scaling limit~\cite{DingFanHuang2023Uniqueness}, subsequently establishing geodesic uniqueness and determining their Euclidean Hausdorff dimension~\cite{DingFanHuang2025Geodesic}. Hutchcroft~\cite{Hutchcroft2025CriticalLRPI,Hutchcroft2025CriticalLRPII,Hutchcroft2025CriticalLRPIII} studied critical cluster geometry in high, low, and upper critical effective dimensions; here, criticality refers to the percolation phase transition, not the distance-growth threshold $s=2d$.

The exponent $\Delta_{\rm III}$ in~\eqref{eq:def-DeltaIII} precisely matches Biskup's exponent at $s=\ga\theta$. For fixed positive endpoint weights, an edge of length $\ell$ is usable at cost $O(1)$ with probability $\ell^{-\ga\theta+o(1)}$, and both upper-bound arguments rely on a binary multiscale construction. We do not address the boundary cases $\ga\theta\in\{d,2d\}$ or $\ga\gc\in\{d,2d\}$ here.

\medskip
\noindent\textbf{Long-range first-passage percolation.}
Setting $V\equiv 1$ reduces the model to homogeneous long-range FPP, initially studied by Mollison~\cite{Mol72} in the context of spatial epidemics. Under exponential edge noises, Chatterjee and Dey~\cite{ChatterjeeDey2013LRFPP} established growth transitions at $\ga\in\{d,2d,2d+1\}$ and a linear-regime shape theorem, with extensions to polynomial powers of exponential noises discussed therein.
For $0\le\ga<d$, van der Hofstad and Lodewijks~\cite{vdHL24} analyzed typical distances, flooding times, and diameters on finite tori.
Crucially, these prior works rely on homogeneous vertex weights; the random vertex environment in our model introduces an additional source of spatial dependence.

Spread-out FPP restricts such homogeneous models to a finite range by setting $V\equiv1$, $\cost(r)=1$ for $1\le r\le\ell$, and $\cost(r)=\infty$ for $r>\ell$.
On fixed-range line graphs, Dey and Kim~\cite{DK26a} identified Gaussian and stable fluctuation regimes.
For growing-range cycle graphs with Weibull edge noises, they established a transition between spatial and mean-field asymptotics~\cite{DK26b}.

\medskip
\noindent\textbf{Scale-free percolation and related models.}
Scale-free percolation, introduced in~\cite{DeijfenHofstadHooghiemstra2011SFP}, is a spatial random graph whose edges are independent conditionally on the vertex weights, with connection probabilities determined by the endpoint weights and distance. For $\cost(r)=r^\ga$ and $\go\sim\mathrm{Exp}(1)$, thresholding the ILRFPP passage times recovers its exponential connection rule. Its supercritical structure and weighted explosion behavior are analyzed in~\cite{HHJ17,HofstadKomjathy2017ExplosionSFP}.

Related spatial models include geometric inhomogeneous random graphs (GIRGs)~\cite{BringmannKeuschLengler2015SamplingGIRG,BringmannKeuschLengler2016AverageDistance} and hyperbolic random graphs (HRGs)~\cite{KrioukovEtAl2010Hyperbolic}. Diameter bounds for HRGs are established in~\cite{KiwiMitsche2014HRGdiameter,FriedrichEtAl2015HRGdiameter}, and limiting weighted distances in the explosive regimes of these models are obtained in~\cite{KomjathyLodewijks2018ExplosionGIRGHRG}. Jacob et al.~\cite{Jacob2025Crossing} study annulus-crossing probabilities in continuum GIRGs. The present hub-chain construction serves as a first-passage analogue of the greedy hub-routing arguments developed in these literatures.

Degree-dependent FPP on sparse spatial graphs is studied in~\cite{KomjathyLapinskasLengler2021Penalising,KomjathyEtAl2024Polynomial,KomjathyLapinskasLenglerSchaller2023}.
A present edge has passage time $L_{\mvx\mvy}(w_{\mvx}w_{\mvy})^\mu$, where $\mu\ge0$ and the variables $(L_{\mvx\mvy})$ are i.i.d.\ and independent of the graph and vertex weights. These models exhibit explosive, polylogarithmic, sublinear, and linear regimes. For $\mu>0$, larger vertex weights increase passage times on present edges. In contrast, our model considers all complete-graph edges and discounts passage times by $V_{\mvx}V_{\mvy}$, requiring a structurally different analysis.

Non-spatial FPP has been studied on complete graphs~\cite{Janson1999OneTwoThree,BhamidiHofstad2012WeakDisorder,EckhoffEtAl2012ShortPaths,EckhoffEtAl2015LongPathsI,EckhoffEtAl2015LongPathsII}, configuration models, and inhomogeneous random graphs~\cite{BhamidiHofstadHooghiemstra2010FiniteMean,BhamidiHofstadHooghiemstra2012Universality,KolossvaryKomjathy2012IRGFPP,BaroniHofstadKomjathy2015InfiniteVar,AdriaansKomjathy2017WeightedCM}. These works establish passage-time and hop-count asymptotics under various structural and distributional assumptions; however, they do not directly treat the distance-dependent passage times considered here. The monographs~\cite{vanDerHofstadRGCN1,vanDerHofstadRGCN2} survey random graphs and their metrics.

\begin{rem}[Comparison with SFP and CFFP]
	\label{rem:LLPS-comparison}
	Lakis, Lengler, Petrova, and Schiller~\cite{LLPS2024} analyze graph distances and first-passage percolation on sparse scale-free percolation (SFP) networks. To establish lower bounds, they introduce Complete Scale Free First Passage Percolation (CFFP) as an auxiliary comparison model, where edge costs are independent exponentials with rate $w_{\mvx}^{\ga_{\LLPS}} w_{\mvy}^{\ga_{\LLPS}} \norm{\mvx-\mvy}^{-d\ga_{\LLPS}}$. This is exactly the exponential--Pareto specialization of our model with $\theta=1$. Setting $\go\sim\mathrm{Exp}(1)$ and $\pr(V>v)=v^{-\gc}$ for $v\ge1$, the parameter correspondence is
	\begin{align}\label{eq:LLPS-param}
		\ga_{\LLPS}={\ga\theta}/{d}, \qquad w_{\mvx}=V_{\mvx}^{d/\ga}, \qquad \tau-1={\ga\gc}/{d}.
	\end{align}

	A monotone coupling in~\cite{LLPS2024} bounds sparse FPP distances below by CFFP cost-distances. Their argument analyzes CFFP under the moment condition $2\ga_{\LLPS}<\tau-1$ (equivalently $\gc>2$) for its role in those sparse-graph lower bounds and leaves the full phase diagram for the complete-graph metric undeveloped.
	This article started as an undergraduate project in the Fall of 2023 (see the acknowledgment section). Consequently, the exponential--Pareto specialization with $\theta=1$ independently and simultaneously appeared in the literature as CFFP. The contribution of the present paper is instead to take the complete-graph metric as the primary object, allow general logarithmic tails and arbitrary $\theta>0$, and develop constructive upper bounds across the conjectured phase diagram. The present assumptions also allow vertex weights that are not bounded away from zero and do not impose inverse-moment conditions.
\end{rem}

\subsection{Our contributions}
\label{ssec:contributions}

\begin{enumeratei}
	\item We formulate a full two-parameter phase diagram for the ILRFPP metric, see Figure~\ref{fig:phase}, and prove upper bounds of the conjectured order in each of its eight regimes. In particular, the construction locates the proposed boundary between the hub and edge mechanisms at $\ga=2q_{\rm hub}$ inside $\{\ga>q_{\rm edge}\}$, rather than at the naive line $\gc=\theta$.
	\item We isolate two reusable path architectures -- the \emph{Hub-Chain Ansatz} (Section~\ref{ssec:Ansatz_II_IVH}) and the \emph{Binary Edge-Bridge multi-scale Ansatz} (Section~\ref{ssec:multi-scale-ansatz}). For each architecture, we prove a master estimate (Lemma~\ref{lem:hub-ub} and Proposition~\ref{prop:Ansatz-bound}) that reduces the regime-specific upper bounds to an optimization over the scale map $f_\phi$, the depth $k(n)$, and a sequence of truncation thresholds.
	\item We demonstrate that the two mechanisms yield structurally distinct optimal paths, and we quantify the difference: the hub-chain cost forms a convergent series and yields tightness, whereas the binary tree structure introduces an entropy factor of $2^k$, leading to divergence. Thus, the exponent $\gD_{\rm III}$ emerges as a branching number.
	\item We provide matching lower bounds for Regime II (Theorems~\ref{thm:II-lb} and~\ref{thm:II-hop}), showing that $T_n=\gTh_{\pr}(1)$ there and that, despite tightness, the hop count of any bounded-cost path diverges at rate at least $\log\log n$: the metric collapses but the combinatorial distance does not.
	\item In Regime $\rm I_H$ we determine the hop-count profile exactly (Theorem~\ref{thm:IH-hop}): the infimum $T(\mvzero,\mvx)=0$ is attained by no finite path, but every positive budget below $W_{\mvzero\mvx}$ is met by infinitely many $2$-hop paths. This is a sharp form of ``instantaneous percolation through hubs''.
\end{enumeratei}

\smallskip\noindent\textbf{Universality and tail assumptions.}
The exponents in the conjectured phase diagram and in the upper bounds proved here depend only on $d,\ga,\gc$, and $\theta$. The implicit constants may depend on the specific weight laws. The proofs require only the logarithmic upper tail of $V$, the lower tail of $\go$, and a finite positive moment for $\go$; no density, exact tail matching, or regular variation is assumed. As detailed in Section~\ref{ssec:disc-assumptions}, the upper-bound arguments use one-sided lower estimates on the relevant tails. In Regimes III--V, universality of the actual asymptotic growth exponent remains contingent on matching lower bounds.

\subsection{Simulations}\label{ssec:simulations}
Figures~\ref{fig:sim-II-III}--\ref{fig:sim-V} show finite-volume simulations of the growth model on an $800\times800$ box in $d=2$. The vertex weights are sampled as i.i.d.~variables with $\pr(V>v)=v^{-\gc}$ for $v\geq 1$, and the edge noises are i.i.d.~$\mathrm{Exp}(1)$, giving $\theta=1$. To avoid artificial edge truncation, we simulated the exact continuous-time Markov growth process, stopping when $15\%$ of the vertices ($96{,}000$) were reached. The plots are therefore qualitative illustrations of the geometric dichotomy in Section~\ref{ssec:dichotomy}, not numerical tests of the asymptotic theorems. The color scales vary between panels and should be compared only after reading their labels.
\begin{figure}[thbp]
	\centering
	\begin{minipage}[t]{0.45\textwidth}
		\includegraphics[height=.85\linewidth]{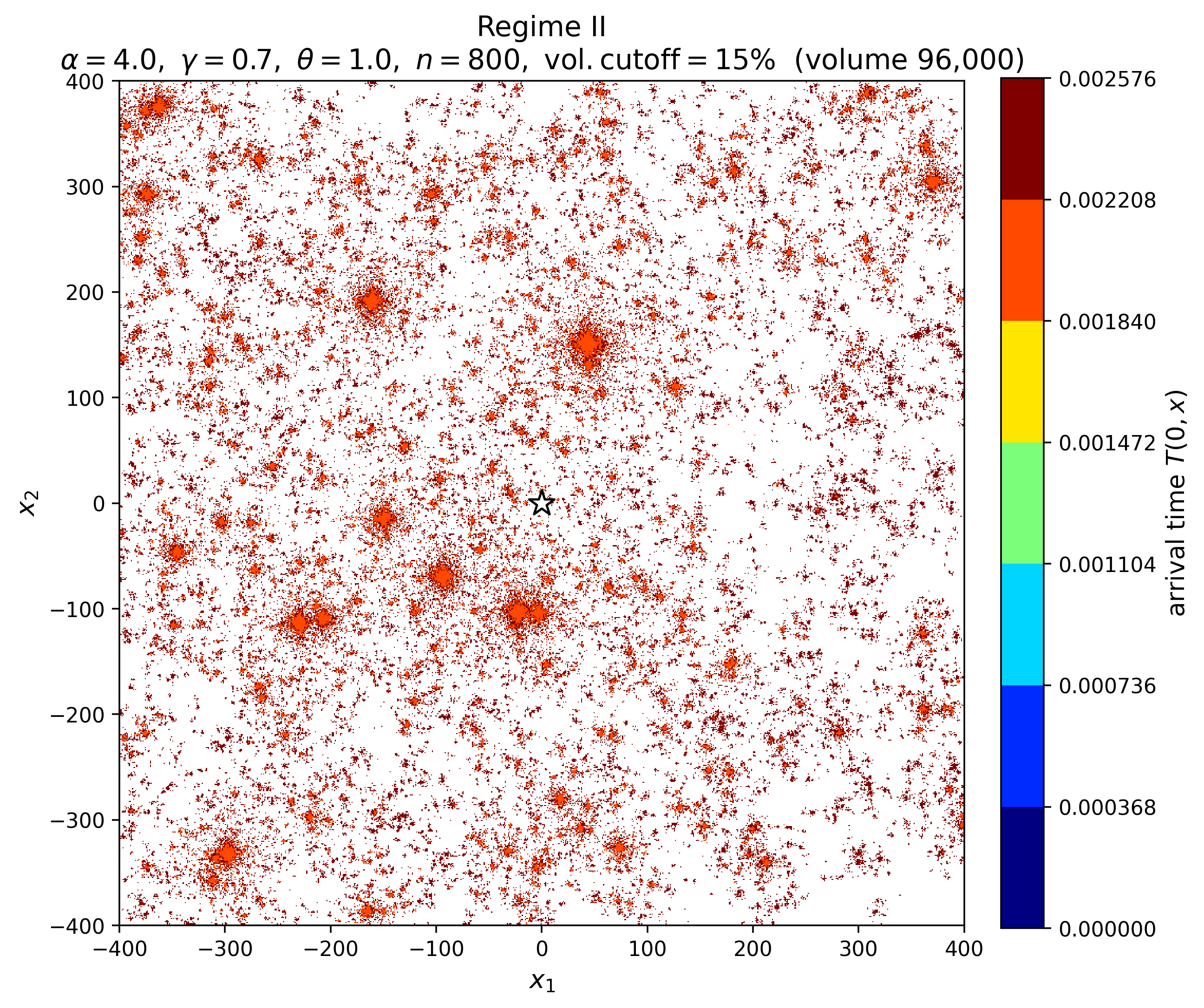}
	\end{minipage}
	\begin{minipage}[t]{0.45\textwidth}
		\includegraphics[height=.85\linewidth]{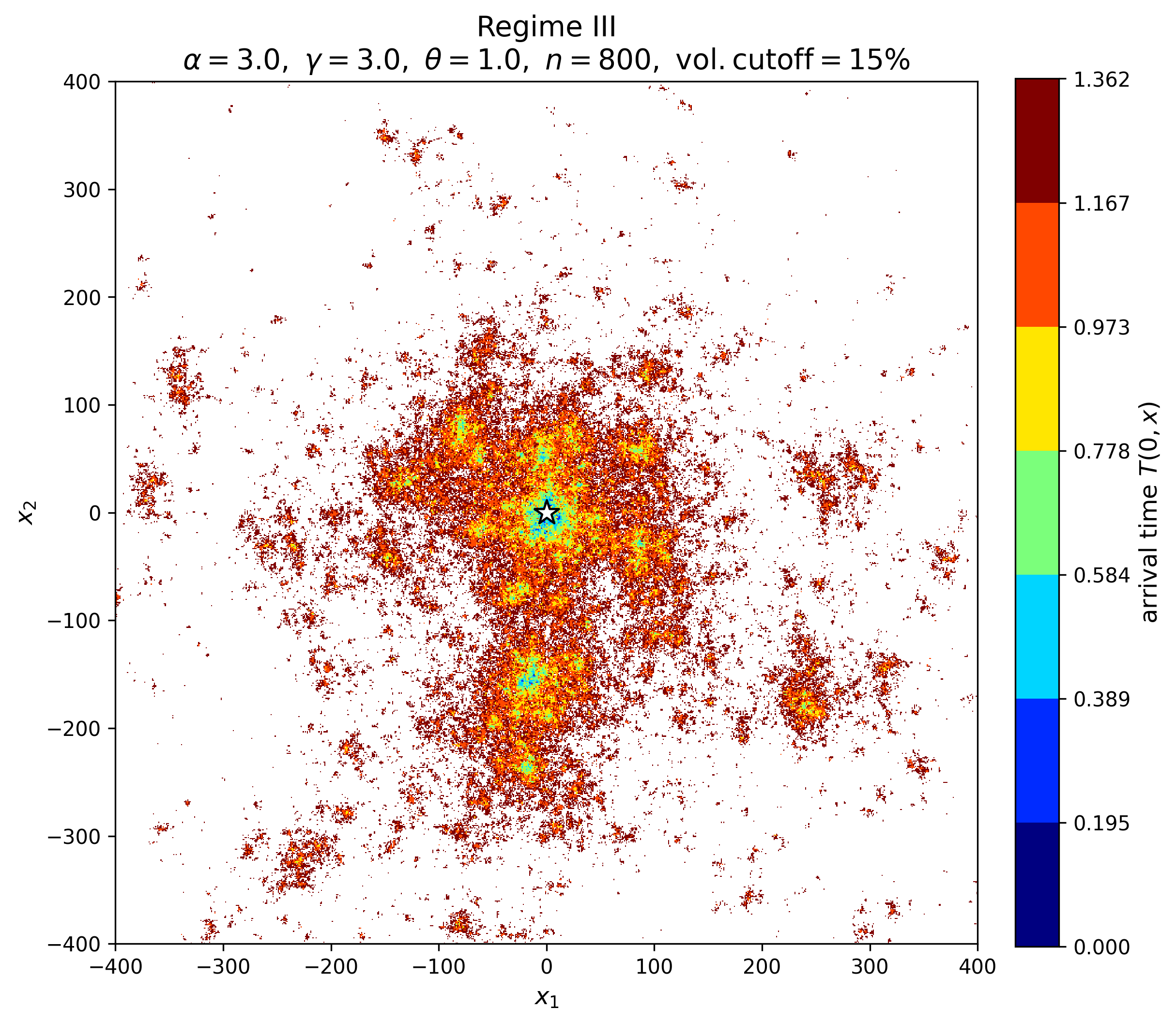}
	\end{minipage}
	\caption{Left: Regime II ($\ga=4$, $\gc=0.7$). The box is reached by time approximately $2.6\times10^{-3}$, with no visible radial structure. Instead, one sees scattered hub neighborhoods. This is qualitatively consistent with tight passage times and finite-time explosion. Right: Regime III ($\ga=3$, $\gc=3$). A radial gradient is visible, while the reached set remains irregular. Passage times are several hundred times larger in this finite simulation. Because the color scales differ, this comparison is only qualitative.}
	\label{fig:sim-II-III}
\end{figure}

\begin{figure}[thbp]
	\centering
	\begin{minipage}[t]{0.45\textwidth}
		\includegraphics[height=\linewidth]{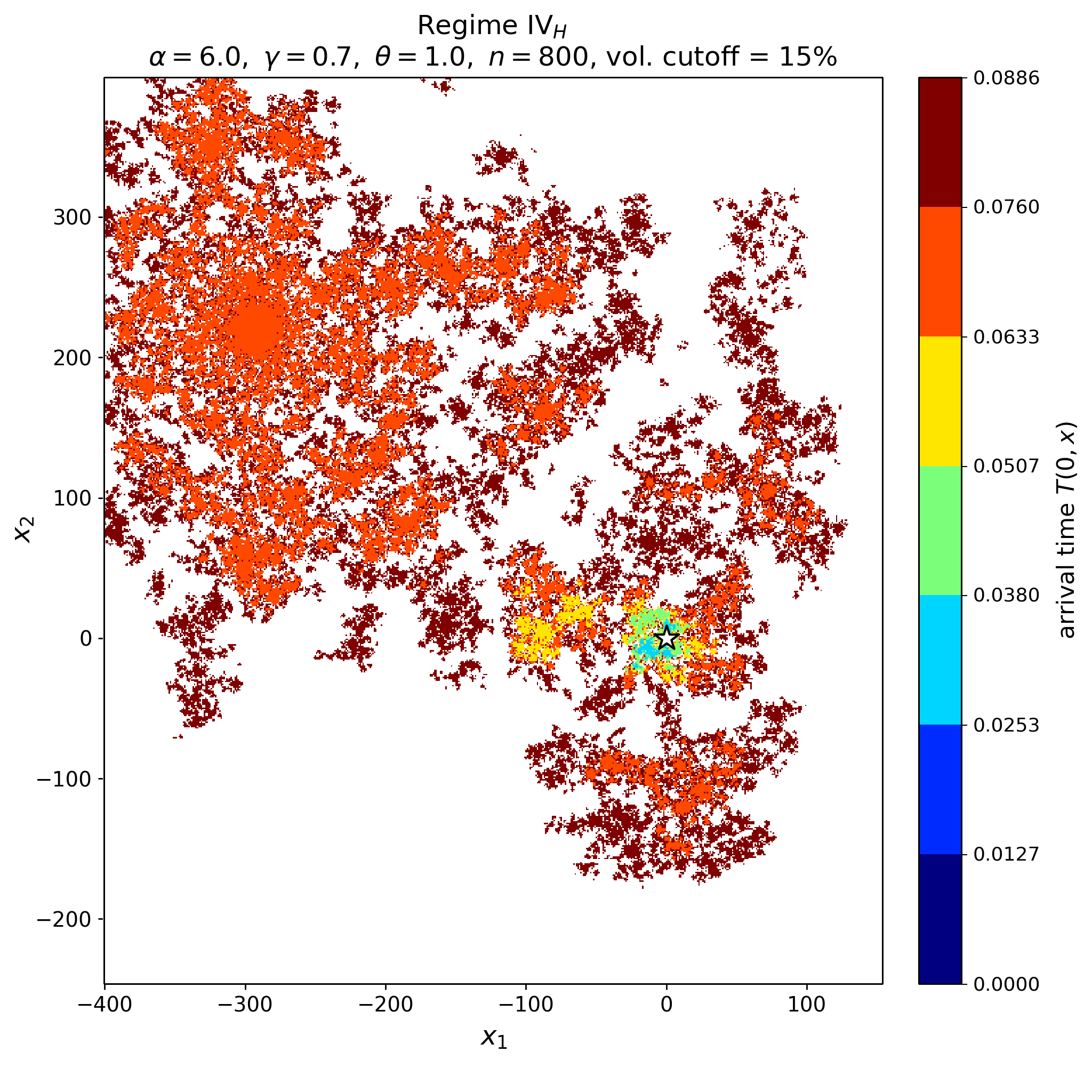}
	\end{minipage}
	\begin{minipage}[t]{0.45\textwidth}
		\includegraphics[height=\linewidth]{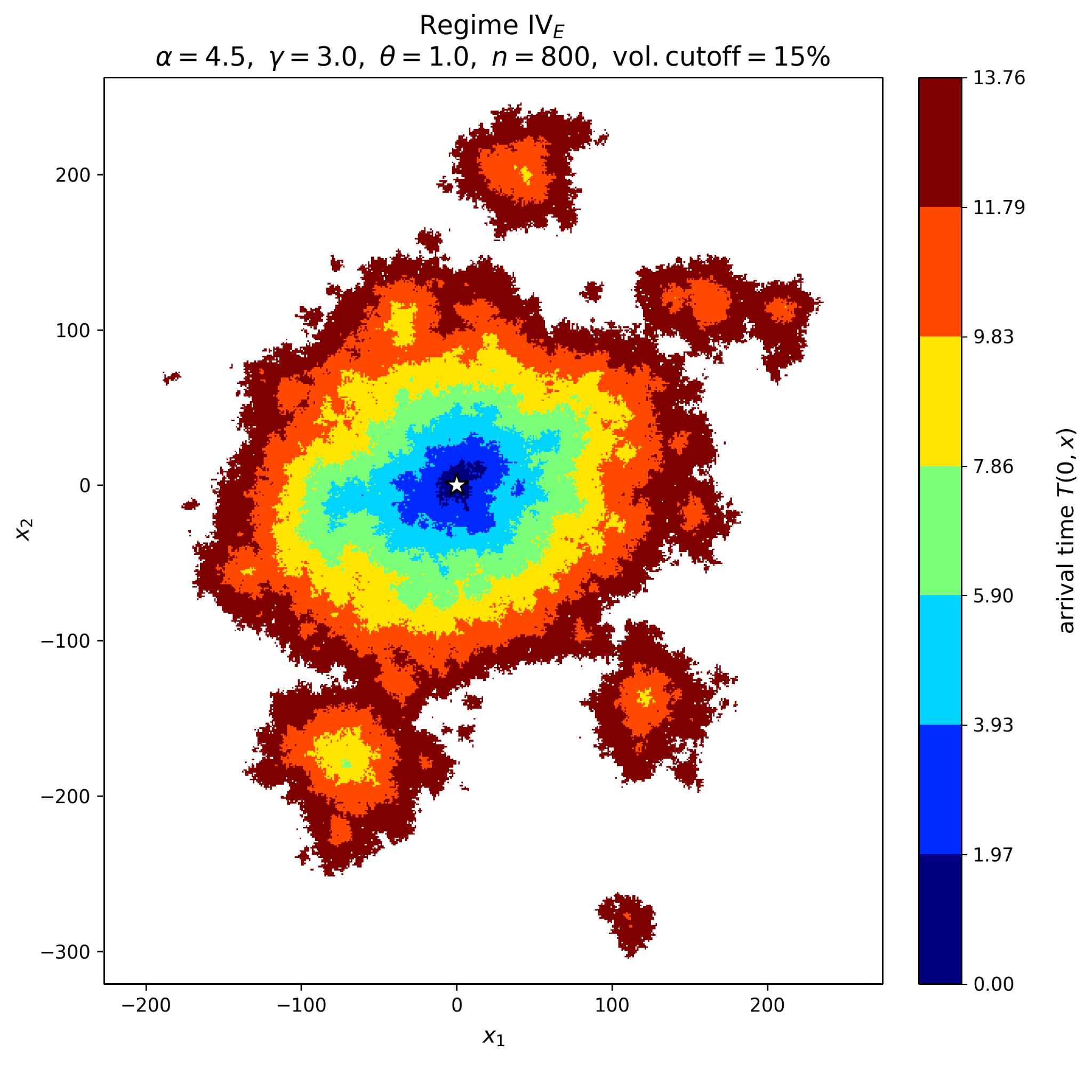}
	\end{minipage}
	\caption{The two power-law regimes, which have sublinear upper bounds on passage times and hence conjecturally superlinear spatial growth. Left: $\rm IV_H$ ($\ga=6$, $\gc=0.7$). A few large hubs dominate the reached set, which is strongly anisotropic and not centered at the origin. Right: $\rm IV_E$ ($\ga=4.5$, $\gc=3$). A roughly radial cluster coexists with detached islands, consistent with isolated cheap long bridges.}
	\label{fig:sim-IV}
\end{figure}

\begin{figure}[thbp]
	\centering
	\begin{minipage}[t]{0.45\textwidth}
		\includegraphics[height=.85\linewidth]{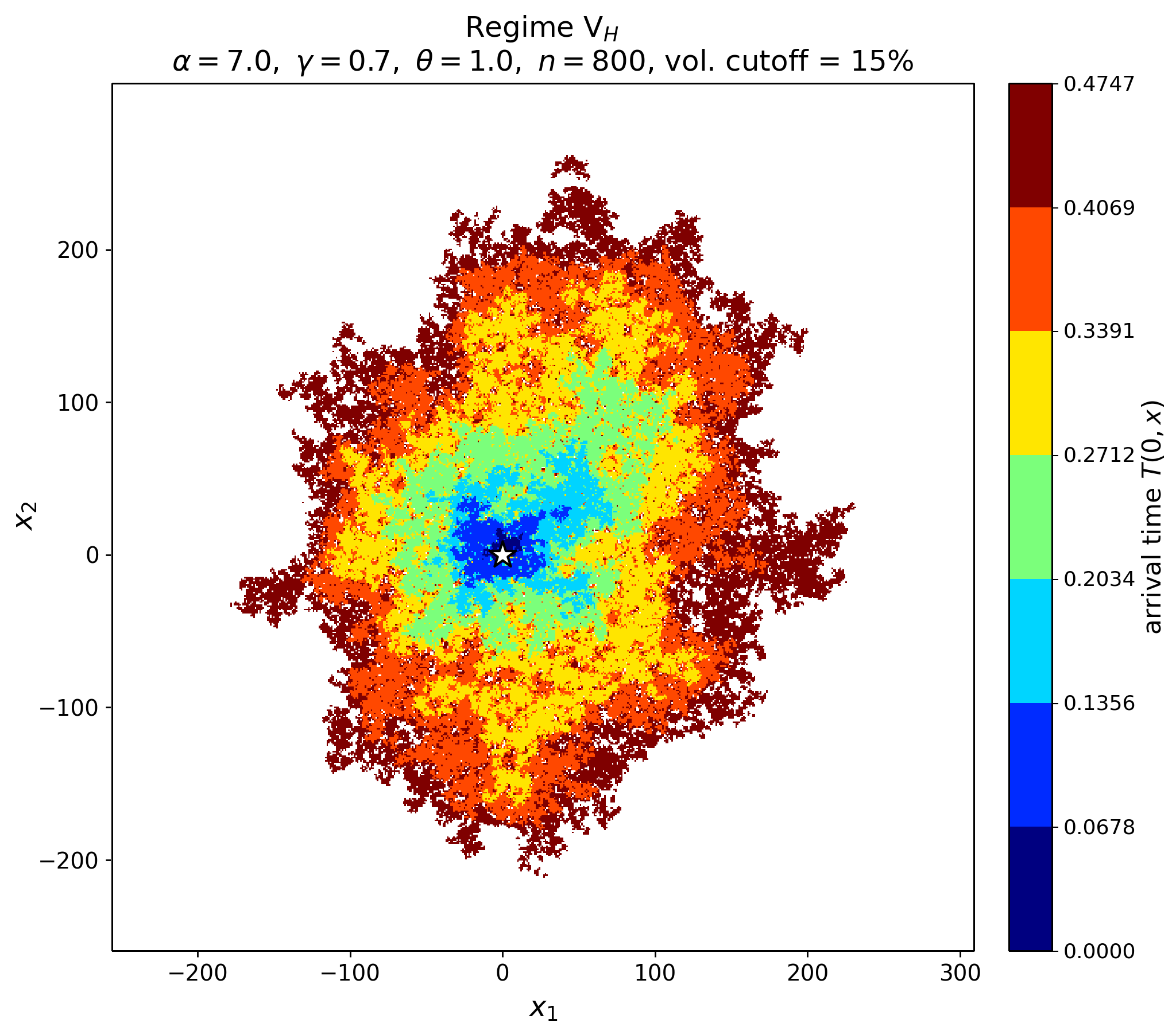}
	\end{minipage}
	\begin{minipage}[t]{0.45\textwidth}
		\includegraphics[height=.85\linewidth]{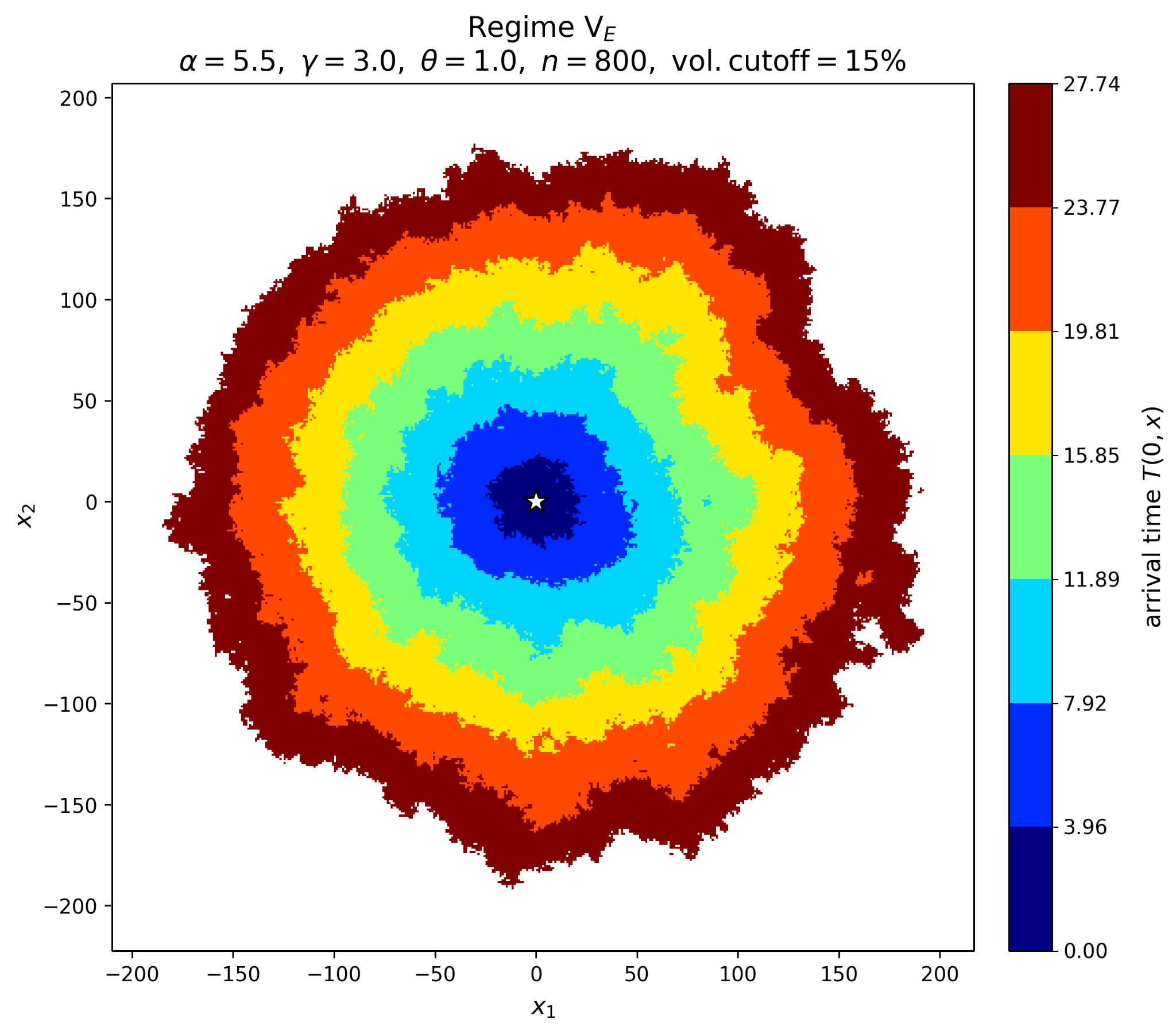}
	\end{minipage}
	\caption{The conjectured linear regimes. Left: $\rm V_H$ ($\ga=7$, $\gc=0.7$). Heavy vertex weights produce visibly rougher, anisotropic level sets. Right: $\rm V_E$ ($\ga=5.5$, $\gc=3$). The finite-volume level sets are approximately round, consistent with the predominance of short hops. A shape theorem is not proved here, so these plots should not be interpreted as evidence for a particular limiting shape.}
	\label{fig:sim-V}
\end{figure}

\subsection{Discussion on the assumptions}\label{ssec:disc-assumptions}
Here we briefly discuss some of our standing assumptions and their possible relaxations.

\begin{enumeratea}
	\item \textbf{Logarithmic tails and one-sided estimates.}\label{rem:tail-loss}
	The limits in~\eqref{def:V_x} and~\eqref{def:go} permit sub-polynomial corrections without assuming regular variation. Thus, for any $\eta>0$, the tails can be bounded by power laws with perturbed exponents $\gc\pm\eta$ and $\theta\pm\eta$. Crucially, our proofs decouple these requirements and rely strictly on one-sided tail estimates:
	\begin{itemize}
		\item \emph{Upper bounds:} Passage-time upper bounds require only the lower estimates $\pr(V\ge v)\ge c_{V,\eta}v^{-\gc-\eta}$ and $\pr(\go\le t)\ge c_{\go,\eta}t^{\theta+\eta}$ to guarantee a sufficient supply of heavy hubs and cheap edges. Away from phase boundaries, these arbitrarily small $\eta$-perturbations are absorbed into the stated $\eps$-losses of the theorems. (Remark~\ref{rem:IE-one-sided} bypasses the upper-tail requirement in Regime $\rm I_E$ via a monotone coupling.)
		\item \emph{Lower bounds:} Conversely, passage-time lower bounds in the companion article~\cite{CDK26} require only the upper estimates $\pr(V\ge v)\le C_{V,\eta}v^{-\gc+\eta}$ and $\pr(\go\le t)\le C_{\go,\eta}t^{\theta-\eta}$ to limit the availability of exceptionally fast paths.
	\end{itemize}

	\item \textbf{On the moment assumption $\E\go^{\gz}<\infty$.}\label{rem:zeta}
	Only a \emph{positive} moment is required, and $\gz$ may be arbitrarily small.
	Specifically, for the $\ell$-fold minimum $\go_{[\ell]}$, we have $\pr(\go_{[\ell]}>t) \le (\E\go^\gz)^\ell t^{-\ell\gz}$.
	Choosing $\ell$ sufficiently large guarantees the necessary integrability. This assumption also governs the endpoint attachments, which are optimized over multiple candidate edges. Therefore, the exponent $\gz(1-\gb)$ in Proposition~\ref{prop:lin-ub-input} is a sufficient operational bound derived from our construction, rather than a fundamental restriction. Determining the weakest upper-tail conditions that preserve these passage-time bounds remains an open question.

	\item \textbf{Independence and the choice of $\cost$.}\label{rem:indep}
	The independence of $(V_{\mvx})$ from $(\go_e)$ is used throughout: our estimates condition on $\cF_V:=\gs(V_{\mvz}:\mvz\in\dZ^d)$ and then treat the edge noises as fresh. Extensions to finite-range dependent vertex weights should be possible when suitable concentration estimates replace the binomial bounds, but the random-center arguments would also have to be revisited. Likewise, a cost of logarithmic order $\cost(r)=r^{\ga+o(1)}$ can be handled away from the boundaries when an arbitrarily small loss in the exponent is allowed. The lossless bound in Regime $\rm IV_H$ and all boundary cases require separate control of the sub-polynomial correction.
\end{enumeratea}

\subsection{Notation}\label{ssec:notation}
We collect the notation used throughout.
\begin{itemize}[leftmargin=2em]
	\item $\dZ^d$ is the lattice, $\sE_{\mathrm{com}}$ the complete edge set, $\la\mvx\mvy\ra$ an edge, and $\mvpi=\la\mvx_0\mvx_1\cdots\mvx_m\ra$ a finite self-avoiding path with hop count $\abs{\mvpi}=m$.
	\item $\norm{\cdot}$ is the $\ell_1$ norm, $\norm{\cdot}_\infty$ the sup norm, and $B(\mvy,r)$ the $\ell_\infty$ ball of radius $r$ about $\mvy$. We write $\wt B(\mvy,r):=B(\mvy,r)\setminus B(\mvy,r/2)$ for the corresponding shell.
	\item $W_e$, $W_{\mvpi}$, $T(\mvx,\mvy)$, $T_n$, $\vB_t$, $\vD_t$ are as in~\eqref{eq:def-W}--\eqref{eq:def-ball} and~\eqref{eq:def-Tn}.
	\item $\ga$ is the distance exponent, $\gc$ the vertex-tail exponent, $\theta$ the edge lower-tail exponent, and $\gz$ the moment exponent of $\go$; $q_{\rm hub}=d/\gc$ and $q_{\rm edge}=d/\theta$.
	\item $\cF_V:=\gs(V_{\mvz}:\mvz\in\dZ^d)$.
	\item For $a>0$, $\cV^a:=\{\mvz\in\dZ^d:V_{\mvz}>a\}$ is the favorable vertex set and $p_a:=\pr(V>a)$.
	\item $\go_{[\ell]}:=\min\{\go_1,\dots,\go_\ell\}$ denotes the minimum of $\ell$ i.i.d.~copies of $\go$.
	\item $f_\phi$ is the regime-dependent scale map, with the globally normalized scale sequence
	      \begin{align}\label{eq:scale-seq}
		      n_0:=n,\qquad n_i:=f_\phi^{(i)}(n)=f_\phi\big(n_{i-1}\big),\qquad i\ge1,
	      \end{align}
	      and $k=k(n)$ the terminal depth. We use $\cA_k,\cB_k,\cE_k$ and
	      $\cG_n,\cL_k$ for good events, plain $A$ for annuli, and $L$ for slowly varying
	      functions.
	\item $a\vee b:=\max\{a,b\}$, $a\wedge b:=\min\{a,b\}$, $[k]:=\{1,\dots,k\}$, and $\ind_A$ is an indicator.
	\item $\asymp$ means equality up to positive constants, $\approx$ heuristic order, $\preceq$ stochastic domination, and $O_{\pr},\gO_{\pr},\gTh_{\pr}$ are the probabilistic order symbols.
\end{itemize}

\subsection{Roadmap}\label{ssec:roadmap}
Sections~\ref{sec:mainresv} and~\ref{sec:mainrese} state the main results, split by mechanism: hub-dominated regimes ($\rm I_H$, II, $\rm IV_H$, $\rm V_H$) and edge-dominated ones ($\rm I_E$, III, $\rm IV_E$, $\rm V_E$), respectively. Section~\ref{sec:aux} collects the three probabilistic estimates used everywhere: a two-sided small-ball bound for weighted sums of edge noises, a moment bound for $1$-dependent sums, and a moment bound for $\ell$-fold minima. Section~\ref{sec:two-path-architectures} constructs the two path architectures and proves the corresponding master estimates; this is the technical heart of the article. Sections~\ref{sec:proof-IE}--\ref{sec:proof-V} then carry out the regime-specific scale optimizations. Section~\ref{sec:discussion} discusses open problems.

\section{Main Results I: Hub-Dominated Regimes}\label{sec:mainresv}

Throughout this section the operative mechanism is the vertex environment. We recall that $q_{\rm hub}=d/\gc$.

\subsection{Regime $\rm I_H$: the instantaneous regime}\label{ssec:res-IH}

\begin{thm}[Instantaneous regime $\rm I_H$]\label{thm:IH}
	Assume \textup{\eqref{A1}--\eqref{A4}} and $\ga<d/\gc$. Fix $\mvx\in\dZ^d\setminus\{\mvzero\}$ and set $n:=\norm{\mvx}$. Define the dyadic annuli
	\begin{align}\label{eq:def-Ahub}
		A_m^{\rm hub}:=\{\mvz\in\dZ^d:m\le\norm{\mvz}<2m\},\qquad m\ge1.
	\end{align}
	Then there exists a random sequence of vertices $\mvy_k\in A_{2^kn}^{\rm hub}$, $k\ge1$, such that
	\begin{align*}
		W_{\mvzero\mvy_k}+W_{\mvy_k\mvx}\longrightarrow0 \text{ a.s.}
	\end{align*}
	In particular $T(\mvzero,\mvx)=0$ almost surely.
\end{thm}

For $t\ge0$, we define the \emph{hop-count profile}
\begin{align}\label{eq:def-H}
	H_{\mvx}(t)
	:=\inf\Big\{\abs{\mvpi}:\mvpi\in\cP_{\mvzero,\mvx}, W_{\mvpi}\le t\Big\},
	\qquad\inf\varnothing:=\infty.
\end{align}

\begin{thm}[Hop-count profile in $\rm I_H$]\label{thm:IH-hop}
	Assume \textup{\eqref{A1}--\eqref{A4}} and $\ga<d/\gc$. Then almost surely, for every $\mvx\in\dZ^d\setminus\{\mvzero\}$,
	\begin{align*}
		H_{\mvx}(0)=\infty,
		\qquad
		H_{\mvx}(t)=2  \text{ for }0<t<W_{\mvzero\mvx},
		\qquad
		H_{\mvx}(t)=1  \text{ for }t\ge W_{\mvzero\mvx}.
	\end{align*}
	In particular,
	\begin{enumeratei}
		\item the infimum $T(\mvzero,\mvx)=0$ is not attained by any finite path,
		\item for every $t\in(0,W_{\mvzero\mvx})$ there are infinitely many distinct $2$-hop paths from $\mvzero$ to $\mvx$ of total weight at most $t$.
	\end{enumeratei}
\end{thm}

\begin{rem}\label{rem:IH}
	\begin{enumeratei}
		\item The ``large hub in an annulus'' step of the proof uses only the lower half of the tail assumption~\eqref{def:V_x} on $V$. No assumption on $\ga$ or $\theta$ enters there.
		\item The essential condition is $\ga<d/\gc$, which makes the exponent $\ga-d/\gc+\eps$ negative for small $\eps$ and forces the two-hop passage times to vanish.
		\item The inequality $\gc<\theta$ is never used. It serves only to locate $\rm I_H$ as the hub-dominated part of the instantaneous phase, and it explains the overlap with $\rm I_E$, where instantaneity is also available through the edge mechanism.
		\item The proof uses the moment assumption $\E\go^\gz<\infty$ from~\eqref{def:go}, but only to run a conditional Borel--Cantelli argument, which is what upgrades convergence in probability to almost sure convergence.
	\end{enumeratei}
\end{rem}

\subsection{Regime II: the tight regime}\label{ssec:res-II}
\begin{thm}[Uniform tightness in Regime II]\label{thm:II-ub}
	Assume \textup{\eqref{A1}--\eqref{A4}}, and suppose that
	\begin{align}\label{eq:II-range}
		\frac{d}{\gc}<\ga<\frac{2d}{\gc}.
	\end{align}
	Then for every $\eps\in(0,1)$ there is a constant $C(\eps)<\infty$, depending only on $\eps$ and on $d,\ga,\gc,\theta,\gz$ and the laws of $V,\go$, such that
	\begin{align}\label{eq:II-uniform}
		\sup_{\mvu\in\dZ^d}\pr\big(T(\mvzero,\mvu)>C(\eps)\big)<\eps.
	\end{align}
	In particular, for every direction $\mvx$ the family $\{T_n\}_{n\ge1}$ is tight, and $T_n=O_{\pr}(1)$.
\end{thm}

The uniformity over the target $\mvu$ in~\eqref{eq:II-uniform}, rather than tightness along a fixed ray, is a genuinely stronger marginal statement. With probability at least $1-\eps$ the origin is within passage time $C(\eps)$ of any prescribed vertex, however far away. It comes for free from the construction, whose geometry depends on $\mvu$ only through $\norm{\mvu}$.

\begin{prop}[Finite explosion time in Regime II]\label{prop:II-ball}
	Under the hypotheses of Theorem~\ref{thm:II-ub}:
	\begin{enumeratei}
		\item for every $\eps\in(0,1)$ there is $t_\eps<\infty$ with
		\begin{align*}
			\pr\big(\abs{\vB_{t_\eps}(\mvzero)}=\infty\big)
			=\pr\big(\vD_{t_\eps}(\mvzero)=\infty\big)\ge1-\eps,
		\end{align*}
		\item consequently the explosion time $\tau_\infty:=\inf\{t\ge0:\abs{\vB_t(\mvzero)}=\infty\}$ is finite almost surely.
	\end{enumeratei}
\end{prop}

The matching lower bound is elementary and holds throughout the non-instantaneous phase.

\begin{thm}[$\gO_{\pr}(1)$ lower bound]\label{thm:II-lb}
	Assume \textup{\eqref{A1}--\eqref{A4}} and
	\begin{align*}
		\ga>\max\Big\{\frac{d}{\theta},\frac{d}{\gc}\Big\}.
	\end{align*}
	Then there exist $\eta>0$ and $C<\infty$ such that
	\begin{align*}
		\sup_{\mvu\in\dZ^d\setminus\{\mvzero\}}\pr\big(T(\mvzero,\mvu)\le t\big)\le C t^{\eta}
		\text{ for all }t\in[0,1].
	\end{align*}
	In particular $T_n=\gO_{\pr}(1)$, and combined with Theorem~\ref{thm:II-ub}, $T_n=\gTh_{\pr}(1)$ in Regime II.
\end{thm}

Tightness of the metric does not mean that bounded-cost connecting paths are short. The following shows that their hop count must diverge, albeit very slowly, matching the $\log\log n$ depth of the hub-chain construction.

\begin{thm}[Hop-count lower bound in Regime II]\label{thm:II-hop}
	Assume \textup{\eqref{A1}--\eqref{A4}}, $\ga\gc>d$ and $\ga\theta>d$. Fix $K>0$ and let $\cP_n(K)$ be the set of finite self-avoiding paths from $\mvzero$ to $\lceil n\mvx\rceil$ of total weight at most $K$. Then there is a constant $a_0>1$ such that for every $\eps>0$ there is $C_{K,\eps}<\infty$ with
	\begin{align*}
		\limsup_{n\to\infty}\pr\Big(\exists \mvpi\in\cP_n(K)\text{ with }\abs{\mvpi}\le\frac{\log\log n}{\log a_0}-C_{K,\eps}\Big)\le\eps.
	\end{align*}
	In particular, for every $c_K<1/\log a_0$,
	$\lim_{n\to\infty}\pr\big(\exists \mvpi\in\cP_n(K)\text{ with }\abs{\mvpi}\le c_K\log\log n\big)=0$.
\end{thm}

\begin{rem}\label{rem:II-hop-range}
	The theorem is stated for Regime II, but its hypotheses only ask that $\ga$ lie above both escape exponents $d/\gc$ and $d/\theta$. The upper constraint $\ga<2d/\gc$ of Regime II is not used. Above Regime II the bound is of course far from sharp, since the constructions of Sections~\ref{sec:proof-IVH} and~\ref{sec:proof-V} use $\asymp\log n$ and $\asymp n$ hops respectively.
\end{rem}

\subsection{Regime $\rm IV_H$: the power-law regime}\label{ssec:res-IVH}

\begin{thm}[Power-law upper bound in $\rm IV_H$]\label{thm:IVH-ub}
	Assume \textup{\eqref{A1}--\eqref{A4}}, and suppose that
	\begin{align}\label{eq:IVH-range}
		\frac{2d}{\gc}<\ga<\frac{2d}{\gc}+1,
	\end{align}
	and set $\Delta_{\rm IV_H}:=\ga-2d/\gc\in(0,1)$. Then, for every $\eps>0$, there exists $C<\infty$ such that
	\begin{align*}
		\pr\Big(T_n>Cn^{\Delta_{\rm IV_H}+\eps}\Big)
		\xrightarrow[n\to\infty]{}0.
	\end{align*}
	In particular, $T_n=O_{\pr}\big(n^{\ga-2d/\gc+\eps}\big)$ for every $\eps>0$.
\end{thm}

For the corresponding bound without an exponent loss under a stronger vertex-tail assumption, see Remark~\ref{rem:IVH-lossless}.

\subsection{Regime $\rm V_H$: the linear regime}\label{ssec:res-VH}

The linear upper bound is proved once and used in both linear sectors; see Section~\ref{ssec:res-V} below.

\section{Main Results II: Edge-Dominated Regimes}\label{sec:mainrese}
Here the operative mechanism is the edge noise, and $q_{\rm edge}=d/\theta$.

\subsection{Regime $\rm I_E$: the instantaneous regime}\label{ssec:res-IE}

\begin{thm}[Instantaneous regime $\rm I_E$]\label{thm:IE}
	Assume \textup{\eqref{A1}--\eqref{A4}} and $\ga<d/\theta$. Then $T(\mvzero,\mvx)=0$ almost surely for every $\mvx\in\dZ^d$.
\end{thm}

\begin{rem}\label{rem:IE}
	\begin{enumeratei}
		\item The multi-scale path family and the ``good vertex'' event in the proof use only that $V>0$ almost surely and that there is an interval $[a_0,A_0]$ with $\pr(a_0\le V\le A_0)>0$. No assumption on the tail of $V$ at infinity is needed. In particular, Theorem~\ref{thm:IE} holds for bounded $V$.
		\item The condition $\ga<d/\theta$ permits a choice of $k$ and $\eps>0$ for which $d(k-1)-k\ga(\theta+\eps)>0$. The proof combines this first-moment growth with an overlap estimate to obtain the second-moment bound needed for the existence of cheap paths.
		\item No comparison between $\gc$ and $\theta$ is used. As in Remark~\ref{rem:IH}, that inequality only locates $\rm I_E$ in the diagram and explains the overlap with $\rm I_H$.
	\end{enumeratei}
\end{rem}

\subsection{Regime III: the polylogarithmic regime}\label{ssec:res-III}

\begin{thm}[Polylogarithmic upper bound in Regime III]\label{thm:III-ub}
	Assume \textup{\eqref{A1}--\eqref{A4}} and
	\begin{align}\label{eq:III-range}
		\frac{d}{\theta}<\ga<\frac{2d}{\theta}
		\text{ and }
		\ga>\frac{2d}{\gc},
	\end{align}
	and recall $\Delta_{\rm III}=\log2/\log(2d/(\ga\theta))\in(1,\infty)$ from~\eqref{eq:def-DeltaIII}. Then for every $\eps>0$ there is $C<\infty$ with
	\begin{align*}
		\pr\Big(T_n>C(\log n)^{\Delta_{\rm III}+\eps}\Big)\xrightarrow[n\to\infty]{}0.
	\end{align*}
\end{thm}

\begin{rem}\label{rem:III-gc}
	The second condition in~\eqref{eq:III-range} is not used anywhere in the proof: the binary edge-bridge construction only requires that $\pr(V>a)>0$ for some $a>0$, which holds by~\eqref{def:V_x}. Thus the polylogarithmic bound of Theorem~\ref{thm:III-ub} is valid whenever $d/\theta<\ga<2d/\theta$, for every $\gc>0$. The condition $\ga>2d/\gc$ only records that this is the regime where no better bound is available from the hub mechanism: for $\ga<2d/\gc$ Theorem~\ref{thm:II-ub} supersedes it. The same remark applies verbatim to the hypothesis $\gc>\theta$ in Theorem~\ref{thm:IVE-ub}.
\end{rem}

\subsection{Regime $\rm IV_E$: the power-law regime}\label{ssec:res-IVE}

\begin{thm}[Power-law upper bound in $\rm IV_E$]\label{thm:IVE-ub}
	Assume \textup{\eqref{A1}--\eqref{A4}}, $\gc>\theta$ and
	\begin{align}\label{eq:IVE-range}
		\frac{2d}{\theta}<\ga<\frac{2d}{\theta}+1,
	\end{align}
	and set $\Delta_{\rm IV_E}:=\ga-2d/\theta\in(0,1)$. Then for every $\eps>0$ there is $C<\infty$ with
	\begin{align*}
		\pr\Big(T_n>C n^{\Delta_{\rm IV_E}+\eps}\Big)\xrightarrow[n\to\infty]{}0.
	\end{align*}
\end{thm}

\subsection{Regimes $\rm V_E$ and $\rm V_H$: the linear regime}\label{ssec:res-V}

Both linear sectors are covered by a single fixed-scale chaining estimate, which does not compare $\gc$ with $\theta$ and in fact does not restrict $\ga$ at all.

\begin{thm}[Linear upper bound]\label{thm:linear-ub}
	Assume \textup{\eqref{A1}--\eqref{A4}}. Then for every direction $\mvx\in B(\mvzero,1)$ there is a constant $C=C(\mvx)<\infty$ with
	\begin{align*}
		\pr\big(T_n>Cn\big)\xrightarrow[n\to\infty]{}0,
	\end{align*}
	that is, $T_n=O_{\pr}(n)$. In particular this holds in Regimes $\rm V_E$ and $\rm V_H$, where it is of the correct order.
\end{thm}

That Theorem~\ref{thm:linear-ub} needs no restriction on $\ga$ is not an oversight. A chain of $\asymp n$ short hops between favorable boxes always costs $O_{\pr}(n)$, whatever $\ga$ is, since the hop lengths have geometric tails and hence all moments. The condition $\ga>1+2(q_{\rm hub}\vee q_{\rm edge})$ defines the regime in which the conjectured phase diagram predicts that this baseline upper bound is sharp. The proof rests on the following estimate, which is also the input for the local segments of the multi-scale ansatz and is therefore stated in the form needed there.

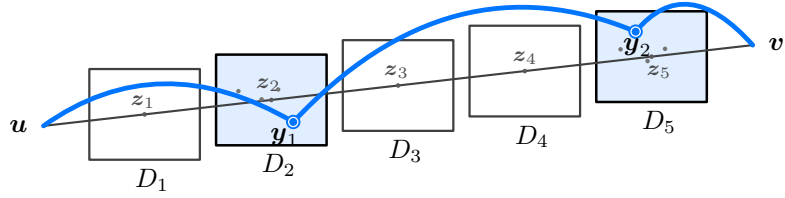
\begin{figure}[htbp]
	\centering
	\begin{tikzpicture}[x=1cm,y=1cm,scale=.75,line cap=round,line join=round,font=\small]

		\def\W{1.95}
		\def\HalfH{0.80}
		\def\Gap{0.30}
		\def\Left{0.85}
		\def\AxisExtra{0.82}
		\def\Tilt{6.5}

		\pgfmathsetmacro{\Step}{\W+\Gap}
		\pgfmathsetmacro{\RightEdge}{\Left + 5*\W + 4*\Gap}

		\pgfmathsetmacro{\xU}{\Left-\AxisExtra}
		\pgfmathsetmacro{\xV}{\RightEdge+\AxisExtra}

		\pgfmathsetmacro{\xZOne}{\Left + 0.5*\W}
		\pgfmathsetmacro{\xZTwo}{\Left + 1*\Step + 0.5*\W}
		\pgfmathsetmacro{\xZThree}{\Left + 2*\Step + 0.5*\W}
		\pgfmathsetmacro{\xZFour}{\Left + 3*\Step + 0.5*\W}
		\pgfmathsetmacro{\xZFive}{\Left + 4*\Step + 0.5*\W}

		\pgfmathsetmacro{\xYOne}{\xZTwo + 0.34}
		\pgfmathsetmacro{\yYOne}{-0.43}
		\pgfmathsetmacro{\xYTwo}{\xZFive - 0.23}
		\pgfmathsetmacro{\yYTwo}{ 0.47}

		\begin{scope}[rotate around={\Tilt:(0,0)}]

			\coordinate (u)  at (\xU,0);
			\coordinate (v)  at (\xV,0);
			\coordinate (z1) at (\xZOne,0);
			\coordinate (z2) at (\xZTwo,0);
			\coordinate (z3) at (\xZThree,0);
			\coordinate (z4) at (\xZFour,0);
			\coordinate (z5) at (\xZFive,0);

			\foreach \i/\fav/\xc in {
					1/0/\xZOne,
					2/1/\xZTwo,
					3/0/\xZThree,
					4/0/\xZFour,
					5/1/\xZFive
				}{%
					\begin{scope}[shift={(\xc,0)}, rotate=-\Tilt]
						\ifnum\fav=1
							\draw[line width=0.95pt, draw=black, fill=trajectoryblue!12]
							(-0.5*\W,-\HalfH) rectangle (0.5*\W,\HalfH);
						\else
							\draw[line width=0.85pt, draw=black!75]
							(-0.5*\W,-\HalfH) rectangle (0.5*\W,\HalfH);
						\fi
					\end{scope}
					\node at (\xc,-1.16) {$D_{\i}$};
				}

			\draw[black!75, line width=0.8pt] (u) -- (v);
			\node[anchor=east] at ($(u)+(-0.10,0.01)$) {$\mvu$};
			\node[anchor=west] at ($(v)+(0.10,0.01)$) {$\mvv$};

			\foreach \P in {z1,z2,z3,z4,z5}{\fill[black!65] (\P) circle (1.15pt);}
			\node[font=\scriptsize, text=black!70] at ($(z1)+( 0.00, 0.23)$) {$\mvz_1$};
			\node[font=\scriptsize, text=black!70] at ($(z2)+(-0.02, 0.28)$) {$\mvz_2$};
			\node[font=\scriptsize, text=black!70] at ($(z3)+( 0.00, 0.24)$) {$\mvz_3$};
			\node[font=\scriptsize, text=black!70] at ($(z4)+( 0.03, 0.23)$) {$\mvz_4$};
			\node[font=\scriptsize, text=black!70] at ($(z5)+( 0.10,-0.28)$) {$\mvz_5$};

			\coordinate (d21) at ($(z2)+(-0.56, 0.22)$);
			\coordinate (d22) at ($(z2)+(-0.17, 0.03)$);
			\coordinate (d23) at ($(z2)+( 0.14, 0.17)$);
			\coordinate (y1)  at (\xYOne,\yYOne);

			\coordinate (d51) at ($(z5)+(-0.53, 0.19)$);
			\coordinate (d52) at ($(z5)+(-0.08,-0.07)$);
			\coordinate (d53) at ($(z5)+( 0.25, 0.11)$);
			\coordinate (y2)  at (\xYTwo,\yYTwo);

			\foreach \P in {d21,d22,d23,d51,d52,d53}{%
					\fill[black!55] (\P) circle (1.12pt);
				}

			\draw[trajectoryblue, line width=1.80pt]
			(u)
			.. controls ($(u)!0.333!(y1)+(0,0.95)$)
			and ($(u)!0.667!(y1)+(0,0.95)$).. (y1);

			\draw[trajectoryblue, line width=1.80pt]
			(y1)
			.. controls ($(y1)!0.333!(y2)+(0,1.45)$)
			and ($(y1)!0.667!(y2)+(0,1.45)$).. (y2);

			\draw[trajectoryblue, line width=1.80pt]
			(y2)
			.. controls ($(y2)!0.333!(v)+(0,0.80)$)
			and ($(y2)!0.667!(v)+(0,0.80)$).. (v);

			\fill[white] (y1) circle (3.05pt);
			\fill[white] (y2) circle (3.05pt);
			\fill[trajectoryblue] (y1) circle (2.0pt);
			\fill[trajectoryblue] (y2) circle (2.0pt);
			\draw[trajectoryblue, line width=0.72pt] (y1) circle (0.11);
			\draw[trajectoryblue, line width=0.72pt] (y2) circle (0.11);

			\node[fill opacity=.92, text opacity=1, inner sep=0.9pt]
			at ($(y1)+(-0.18,-0.25)$) {$\mvy_1$};
			\node[fill opacity=.92, text opacity=1, inner sep=0.9pt]
			at ($(y2)+(0.00,-0.27)$) {$\mvy_2$};

		\end{scope}
	\end{tikzpicture}
	\caption{The shaded boxes are favorable boxes. The selected points $\mvy_1$ and $\mvy_2$ are chosen from the candidate vertices inside the favorable boxes, and the segment $\mvu\mvv$ passes through the deterministic centers $\mvz_1,\dots,\mvz_5$.}
	\label{fig:linear-ansatz}
\end{figure}

\begin{prop}[Fixed-scale chaining]\label{prop:lin-ub-input}
	Assume \textup{\eqref{A1}--\eqref{A4}} and fix $a>0$ with $p_a=\pr(V>a)>0$.
	\begin{enumeratei}
		\item \textup{(Interior version.)} Fix $\eps_1\in(0,1)$ and a desired tail exponent $q_0>0$. There are constants $\Lambda_0<\infty$, $c_{\rm int}<\infty$ and $R_{\rm int}<\infty$, which may depend on $q_0$, with the following properties. For $\mvu,\mvv\in\dZ^d$ put $m:=\norm{\mvu-\mvv}$ and let $\cC(\mvu,\mvv)$ be the $\cF_V$-measurable \emph{corridor event}
		\begin{align}\label{eq:corridor-event}
			\begin{split}
				\cC(\mvu,\mvv) & :=\{V_{\mvu}\wedge V_{\mvv}>a\}\cap\{N\le\Lambda_0m\}         \\
				               & \qquad\cap\Big\{\sum_{j=1}^{N+1}Y_j^{2\ga}\le\Lambda_0m\Big\}
				\cap\Big\{\max_{1\le j\le N+1}Y_j\le\eps_1m/C_0\Big\},
			\end{split}
		\end{align}
		where $N$ and $(Y_j)$ are the number of good boxes and the gaps between consecutive good boxes along a lattice ray from $\mvu$ through $\mvv$, and $C_0<\infty$ is a geometric constant, all defined in the proof. For every $q>0$ there is $C_q<\infty$ such that, for all $m\ge R_{\rm int}$,
		\begin{align}\label{eq:lin-int-a}
			\pr\Big(\cC(\mvu,\mvv)^{\rm c}, V_{\mvu}\wedge V_{\mvv}>a\Big)\le C_q m^{-q},
		\end{align}
		and, almost surely on $\cC(\mvu,\mvv)$, for every $\gs$-field $\cH$ with
		$\cF_V\subseteq\cH\subseteq\gs\big(\cF_V,\{\go_e:\norm{e}_\infty>\eps_1m\}\big)$,
		\begin{align}\label{eq:lin-int}
			\pr\Big(T^{\eps_1}(\mvu,\mvv)\ge c_{\rm int} m\ \big|\ \cH\Big)\le C_{q_0} m^{-q_0},
		\end{align}
		where $T^{\eps_1}(\mvu,\mvv)$ is the first-passage time computed over paths all of whose edges have $\ell_\infty$ length at most $\eps_1m$.
		The same conditional estimate remains valid if an auxiliary $\gs$-field independent of the entire edge-noise family is adjoined to both endpoint $\gs$-fields. This observation will be used for the tie-breaking marks in the binary construction.
		\item \textup{(Boundary version.)} Fix $p\ge2$ and $\gb\in(0,1)$ and set $\gd:=\min\{p/2,\gz(1-\gb)\}$. There are constants $b_0,c_\ast,C<\infty$ such that, with
		\begin{align*}
			\cE_{\mvu\mvv}:=\Big\{V_{\mvu}\wedge V_{\mvv}\ge c_\ast\norm{\mvu-\mvv}^{-\gb}\Big\},
		\end{align*}
		we have
		\begin{align}\label{eq:lin-bdry}
			\pr\Big(T(\mvu,\mvv)\ge b_0\norm{\mvu-\mvv}, \cE_{\mvu\mvv}\Big)
			\le C\norm{\mvu-\mvv}^{-\gd}
		\end{align}
		for all $\norm{\mvu-\mvv}$ sufficiently large.
	\end{enumeratei}
\end{prop}

The two versions differ in what they assume about the endpoints and in what they deliver. In (i) both endpoints are favorable, and the construction may optimize the edge noise at \emph{both} ends over $\ell$ candidates. Choosing $\ell$ after $q_0$ yields the requested power $q_0$, which is what permits a union bound over the $2^k$ terminal segments of the binary ansatz. The constants and the corridor construction therefore depend on $q_0$. In (ii) the endpoints are prescribed and their weights may be small. The endpoint estimate used below contributes the term $\gz(1-\gb)$; no optimality of this exponent is claimed.

Two features of (i) are dictated by the way it is used inside the multi-scale ansatz, where the endpoints $\mvu,\mvv$ are not deterministic but are produced by the construction itself. First, the estimate is split into an $\cF_V$-measurable \emph{corridor} condition~\eqref{eq:lin-int-a}, which concerns the vertex environment only, and a conditional bound~\eqref{eq:lin-int} which holds \emph{pointwise} on that condition and involves only the short edge noises. Second, the restriction to short edges keeps the chaining edges disjoint from every candidate bridge edge, so that~\eqref{eq:lin-int} may be applied conditionally on the whole bridge history. Together these let us verify the corridor condition once, as part of the good event of the construction (Lemma~\ref{lem:Agd-prob}), and then apply~\eqref{eq:lin-int} conditionally, which is what makes the argument insensitive to the fact that the endpoints are random.

\section{Auxiliary Lemmas}\label{sec:aux}

We isolate three estimates used repeatedly. The first is a two-sided small-ball bound for a weighted sum of edge noises, with explicit dependence on the weights and on the number of summands. The polynomial-in-$t$ form with an arbitrarily small loss $\eps$ in the exponent is what makes the second-moment computation of Section~\ref{sec:proof-IE} possible.

\begin{lem}[Small-ball bounds for weighted noise sums]\label{lem:go-tail-bound}
	Assume~\eqref{def:go}. Let $\go_1,\go_2,\dots$ be i.i.d.~copies of $\go$ and let $\gl_1,\gl_2,\dots>0$. Fix $\eps\in(0,\theta)$. There exist $a=a(\eps)>0$ and constants $0<c_\eps<C_\eps<\infty$ such that for every $k\ge1$ and every $t>0$ with $t\gl_i/k\le a$ for all $1\le i\le k$,
	\begin{align}\label{eq:go-tail-two-sided}
		c_\eps^k t^{k(\theta+\eps)}k^{-k(\theta+\eps)}\prod_{i=1}^k\gl_i^{\theta+\eps}
		\le \pr\Big(\sum_{i=1}^k\frac{\go_i}{\gl_i}\le t\Big)
		\le C_\eps^k t^{k(\theta-\eps)}k^{-k(\theta-\eps)}\prod_{i=1}^k\gl_i^{\theta-\eps}.
	\end{align}
	The upper bound holds for every $t>0$, without the constraint $t\gl_i/k\le a$.
\end{lem}

\begin{proof}
	The logarithmic limit in~\eqref{def:go} implies that, for the given $\eps\in(0,\theta)$, there are $a=a(\eps)>0$ and $0<\tilde c_\eps<\tilde C_\eps<\infty$ with
	\begin{align}\label{eq:go-eps-bd}
		\tilde c_\eps u^{\theta+\eps}\le\pr(\go\le u)\le\tilde C_\eps u^{\theta-\eps}
		\text{ for all }u\in(0,a].
	\end{align}

	\smallskip
	\noindent\emph{Lower bound.} The event $\{\go_i\le t\gl_i/k\text{ for all }i\}$ is contained in $\{\sum_i\go_i/\gl_i\le t\}$. Since $t\gl_i/k\le a$ for each $i$, independence and~\eqref{eq:go-eps-bd} give
	\begin{align*}
		\pr\Big(\sum_{i=1}^k\frac{\go_i}{\gl_i}\le t\Big)
		\ge\prod_{i=1}^k\pr\Big(\go\le \frac{t\gl_i}{k}\Big)
		\ge\tilde c_\eps^{ k} t^{k(\theta+\eps)}k^{-k(\theta+\eps)}\prod_{i=1}^k\gl_i^{\theta+\eps}.
	\end{align*}

	\smallskip
	\noindent\emph{Upper bound.} For any $y>0$, Markov's inequality applied to $e^{-y\sum_i\go_i/\gl_i}$ gives
	\begin{align}\label{eq:expmarkov}
		\pr\Big(\sum_{i=1}^k\frac{\go_i}{\gl_i}\le t\Big)
		\le e^{yt}\prod_{i=1}^k\E e^{-y\go/\gl_i}.
	\end{align}
	Write $F_\go(u):=\pr(\go\le u)$. Integrating by parts, $\E e^{-s\go}=\int_0^\infty e^{-u}F_\go(u/s) du$ for every $s>0$. Splitting at $u=as$ and using~\eqref{eq:go-eps-bd} on $(0,as]$ and $F_\go\le1$ beyond,
	\begin{align*}
		\E e^{-s\go}
		 & \le\tilde C_\eps\int_0^{as}e^{-u}\Big(\frac{u}{s}\Big)^{\theta-\eps}du+e^{-as} \\
		 & \le\Big(\tilde C_\eps\gC(\theta-\eps+1)
		+a^{-(\theta-\eps)}\sup_{r>0}r^{\theta-\eps}e^{-r}\Big)s^{-(\theta-\eps)}
		=:D_\eps s^{-(\theta-\eps)},
	\end{align*}
	valid for all $s>0$. Substituting into~\eqref{eq:expmarkov} with $s=y/\gl_i$,
	\begin{align*}
		\pr\Big(\sum_{i=1}^k\frac{\go_i}{\gl_i}\le t\Big)
		\le e^{yt}D_\eps^{ k} y^{-k(\theta-\eps)}\prod_{i=1}^k\gl_i^{\theta-\eps}.
	\end{align*}
	Choosing $y:=k(\theta-\eps)/t$ and absorbing $e^{k(\theta-\eps)}(\theta-\eps)^{-k(\theta-\eps)}$ into $C_\eps$ gives the upper bound.
\end{proof}

\begin{lem}[Finite balls and geodesics]\label{lem:finite-geodesics}
	Assume \textup{\eqref{A1}--\eqref{A4}}, $\ga\theta>d$, and
	$\ga\gc>2d$. For $\mvx\in\dZ^d$ and $t\ge0$, let
	\begin{align*}
		N_t(\mvx):=\#\{\mvpi:\mvpi\text{ is a finite self-avoiding path
			                          starting at }\mvx,\ \abs{\mvpi}\ge1,\ W_{\mvpi}\le t\}.
	\end{align*}
	Then $\E N_t(\mvx)<\infty$. Almost surely, simultaneously for all $\mvx,\mvy\in\dZ^d$ and all finite $t\ge0$, the set $\vB_t(\mvx)$ is finite and $T(\mvx,\mvy)$ is attained by a finite self-avoiding path. If the law of $\go$ is atomless, this path is almost surely unique for every pair of distinct vertices.
\end{lem}

\begin{proof}
	Choose $\eta$ such that $d/\ga<\eta<\min\{\theta,\gc/2\}$. Then $\E V^{2\eta}<\infty$ and $\sum_{\mvz\ne\mvzero}\norm{\mvz}^{-\ga\eta}<\infty$.

	For a fixed self-avoiding path $\mvpi=\la\mvx_0\cdots\mvx_k\ra$, condition on its vertex weights and apply the upper bound of Lemma~\ref{lem:go-tail-bound} with $\eps=\theta-\eta$ and $\gl_i=V_{\mvx_{i-1}}V_{\mvx_i}\norm{\mvx_i-\mvx_{i-1}}^{-\ga}$. Integrating the distinct vertex weights gives, for $t>0$,
	\begin{align*}
		\pr(W_{\mvpi}\le t)
		\le C^k(t/k)^{k\eta}(\E V^\eta)^2 (\E V^{2\eta})^{k-1}
		\prod_{i=1}^k\norm{\mvx_i-\mvx_{i-1}}^{-\ga\eta}.
	\end{align*}
	Since $(\E V^\eta)^2\le \E V^{2\eta}$, relaxing the self-avoidance constraint and bounding the sum over paths from $\mvx$ by independent sums over nonzero increments gives
	\begin{align*}
		\E N_t(\mvx)
		\le\sum_{k\ge1}\left(\frac{At^\eta}{k^\eta}\right)^k<\infty,
		\text{ where } A:=C \cdot(\E V^{2\eta})\sum_{\mvz\in\dZ^d\setminus\{\mvzero\}}
		\norm{\mvz}^{-\ga\eta}.
	\end{align*}

	Intersecting over $\mvx\in\dZ^d$ and $t\in\dN$, $N_t(\mvx)<\infty$ globally almost surely. Thus, $\vB_t(\mvx)$ is finite since every $\mvy\in\vB_t(\mvx)$ is reached by a path of weight $<t+1$. For $\mvx\ne\mvy$, the infimum $T(\mvx,\mvy)$ is attained over the finite, nonempty set of paths from $\mvx$ weighing at most $W_{\mvx\mvy}$. Uniqueness for atomless $\go$ follows since distinct finite paths have equal weights with probability zero.
\end{proof}

\begin{lem}[$p$-th moment bound for $1$-dependent sums]\label{lem:p-conS_k}
	Fix $p\ge2$ and let $\{Y_i\}_{i=1}^k$ be a $1$-dependent family: whenever $I,J\subset[k]$ satisfy $\min\{\abs{i-j}:i\in I,j\in J\}>1$, the $\gs$-fields generated by $\{Y_i:i\in I\}$ and $\{Y_j:j\in J\}$ are independent. Write $\overline Y_i:=Y_i-\E Y_i$ and assume
	\begin{align*}
		M_p:=\sup_{1\le i\le k}\norm{\overline Y_i}_p<\infty.
	\end{align*}
	Then there is a constant $C_p<\infty$ depending only on $p$ such that for all deterministic $(a_i)_{i=1}^k$,
	\begin{align*}
		\Big\|\sum_{i=1}^k a_i\overline Y_i\Big\|_p\le C_p M_p\Big(\sum_{i=1}^k a_i^2\Big)^{1/2},
	\end{align*}
	and consequently, for every $t>0$,
	\begin{align}\label{eq:1dep-tail}
		\pr\Big(\Big|\sum_{i=1}^k a_i\overline Y_i\Big|>t\Big)
		\le C_p^{ p}\left(\frac{M_p\big(\sum_{i=1}^ka_i^2\big)^{1/2}}{t}\right)^{\!p}.
	\end{align}
\end{lem}

\begin{proof}
	By $1$-dependence, each of the subfamilies $\{\overline Y_i:i\text{ even}\}$ and $\{\overline Y_i:i\text{ odd}\}$ consists of independent random variables. Put
	$S_{\rm ev}:=\sum_{i\text{ even}}a_i\overline Y_i$ and $S_{\rm od}:=\sum_{i\text{ odd}}a_i\overline Y_i$, so that by the triangle inequality it suffices to bound each. Let $I$ be the set of even indices and $Z_i:=a_i\overline Y_i$. Rosenthal's inequality provides $A_p<\infty$ with
	\begin{align*}
		\E\Big|\sum_{i\in I}Z_i\Big|^p
		\le A_p\left(\sum_{i\in I}\E\abs{Z_i}^p+\Big(\sum_{i\in I}\E Z_i^2\Big)^{p/2}\right).
	\end{align*}
	Since $\E\abs{Z_i}^p\le\abs{a_i}^pM_p^p$ and $\E Z_i^2\le a_i^2M_p^2$, and since $\big(\sum_i\abs{a_i}^p\big)^{1/p}\le\big(\sum_ia_i^2\big)^{1/2}$ for $p\ge2$, we get
	\begin{align*}
		\E\Big|\sum_{i\in I}Z_i\Big|^p
		\le 2A_p M_p^p\Big(\sum_{i\in I}a_i^2\Big)^{p/2},
	\end{align*}
	whence $\norm{S_{\rm ev}}_p\le(2A_p)^{1/p}M_p(\sum_ia_i^2)^{1/2}$, and the same for $S_{\rm od}$. Taking $C_p:=2(2A_p)^{1/p}$ gives the moment bound, and~\eqref{eq:1dep-tail} follows from Markov's inequality applied to the $p$-th power.
\end{proof}

\begin{lem}[Moments of $\ell$-fold minima]\label{lem:min-moments}
	Assume~\eqref{def:go} and let $\go_{[\ell]}=\min\{\go_1,\dots,\go_\ell\}$ with $\go_1,\dots,\go_\ell$ i.i.d.~copies of $\go$. Then
	\begin{align*}
		\pr(\go_{[\ell]}>t)\le\big(\E\go^{\gz}\big)^{\ell}t^{-\ell\gz}
		\text{ for all }t>0,
	\end{align*}
	and hence $\E\go_{[\ell]}^{ m}<\infty$ for every $m<\ell\gz$. In particular, for any prescribed $m$ one may choose $\ell>m/\gz$ and obtain a finite $m$-th moment. The same holds for $\min\{\go_1+\go_1',\dots,\go_\ell+\go_\ell'\}$ where $(\go_j,\go_j')_{j\le\ell}$ are $2\ell$ i.i.d.~copies of $\go$, with $\E\go^\gz$ replaced by $2^{1+\gz}\E\go^{\gz}$.
\end{lem}

\begin{proof}
	By independence and Markov's inequality, $\pr(\go_{[\ell]}>t)=\pr(\go>t)^{\ell}\le(\E\go^\gz t^{-\gz})^{\ell}$. Integrating the tail, $\E\go_{[\ell]}^m=\int_0^\infty mt^{m-1}\pr(\go_{[\ell]}>t) dt$ converges at infinity when $m<\ell\gz$ and at $0$ trivially. For the second statement, $\pr(\go+\go'>t)\le2\pr(\go>t/2)\le2^{1+\gz}\E\go^{\gz}t^{-\gz}$, and the same computation applies.
\end{proof}

\section{Two Path Architectures}\label{sec:two-path-architectures}

The intermediate-regime upper bounds use the two path constructions developed in this section. The instantaneous and linear regimes are treated separately. In the hub-dominated regimes we select a sequence of exceptionally heavy vertices, one per scale, and join them by a non-branching chain. In the edge-dominated regimes we recursively split the segment from $\mvzero$ to $\lceil n\mvx\rceil$ into a bridge and two smaller segments. In both cases the regime-specific input consists only of
\begin{enumeratei}
	\item the scale map $f_\phi$ generating $n_i$ as in~\eqref{eq:scale-seq},
	\item the terminal depth $k=k(n)$, and
	\item a truncation sequence,
\end{enumeratei}
and the geometry is identical across regimes. Fix throughout a direction $\mvx\in B(\mvzero,1)\setminus\{\mvzero\}$, write $\mvm:=\lceil n\mvx\rceil$ and $T_n=T(\mvzero,\mvm)$, and recall $B(\mvy,r)$ is the $\ell_\infty$ ball.

\subsection{The Hub-Chain Ansatz for the hub-dominated regimes}\label{ssec:Ansatz_II_IVH}

This is the construction used in Regimes II and $\rm IV_H$. Throughout this subsection we assume that the scale sequence~\eqref{eq:scale-seq} satisfies
\begin{align}\label{eq:hub-scale-separation}
	n_{i+1}\le\frac{n_i}{8},\qquad 0\le i\le k,
	\text{ and }
	n_1\le\frac{1}{16}\norm{\mvm}_\infty ,
\end{align}
which holds for both of the scale maps used below once $n$ is large. (The second condition is a \emph{root separation}: it keeps the two halves of the chain apart. For the scale map $n_i=n^{\phi^i}$ it is automatic for large $n$. For the geometric map $n_i=n/\phi^i$ it is a lower bound on $\phi$ in terms of $\norm{\mvx}_\infty$.) Throughout this subsection distances between centers are measured in $\ell_\infty$, in which the balls $B(\cdot,r)$ are defined. The constant $c_0$ below converts to the $\ell_1$ lengths entering $W_e$.

\subsubsection*{Geometry}

Let $\hat{\mvz}$ be a lattice point nearest to $\mvm/2$, so that $\norm{\hat{\mvz}}_\infty\asymp\norm{\mvm-\hat{\mvz}}_\infty\asymp\norm{\mvm}_\infty\asymp n$. For $1\le i\le k$ choose lattice points $\mvz_i^+$ and $\mvz_i^-$ with
\begin{align}\label{eq:hub-centers}
	n_i\le\norm{\mvz_i^+-\mvm}_\infty\le2n_i,
	\qquad
	n_i\le\norm{\mvz_i^--\mvzero}_\infty\le2n_i.
\end{align}
Such points exist whenever $n_i\ge1$. For example, move $\lceil n_i\rceil$
steps from the endpoint $\mvm$ or $\mvzero$ in a fixed coordinate direction.
No collinearity with $\hat{\mvz}$ is required: all estimates below use only the
radial bounds in~\eqref{eq:hub-centers}. Define the balls
\begin{align}\label{eq:hub-balls}
	B_0:=B\big(\hat{\mvz},n_1\big),
	\qquad
	B_{\pm i}:=B\big(\mvz_i^{\pm},n_{i+1}\big),\qquad 1\le i\le k.
\end{align}
Thus the ball of index $i$ always has radius $n_{\abs{i}+1}$ and sits at distance of order $n_{\abs{i}}$ from the endpoint it is approaching. We record the two geometric facts we shall need.

\begin{lem}[Geometry of the hub chain]\label{lem:hub-geom}
	Assume~\eqref{eq:hub-scale-separation}. Then the $2k+1$ balls in~\eqref{eq:hub-balls} are pairwise disjoint, and there is a constant $c_0=c_0(d)<\infty$ such that for all choices of $\mvy_i\in B_i$,
	\begin{align}\label{eq:hub-hoplength}
		\norm{\mvy_i-\mvy_{i+1}}\le c_0 n_{\abs{i}}\quad(0\le i\le k-1),
		\qquad
		\norm{\mvy_i-\mvy_{i+1}}\ge\tfrac12 n_{\abs{i}}\quad(1\le i\le k-1),
	\end{align}
	and the same on the left half, while the terminal distances satisfy
	\begin{align}\label{eq:hub-terminal}
		\norm{\mvm-\mvy_k}\le c_0n_k,
		\qquad
		\norm{\mvzero-\mvy_{-k}}\le c_0n_k.
	\end{align}
\end{lem}

\begin{proof}
	All statements are elementary consequences of~\eqref{eq:hub-scale-separation} and~\eqref{eq:hub-centers}. We do the right half, the left being identical. We work in $\ell_\infty$ and convert at the end. For $1\le i<j\le k$, the centers satisfy
	$\norm{\mvz_i^+-\mvz_j^+}_\infty\ge\norm{\mvz_i^+-\mvm}_\infty-\norm{\mvz_j^+-\mvm}_\infty\ge n_i-2n_j\ge n_i-2n_{i+1}\ge\tfrac34n_i$,
	whereas the sum of the radii is
	$n_{i+1}+n_{j+1}\le2n_{i+1}\le\tfrac14 n_i$. Hence
	$B_i\cap B_j=\varnothing$. Likewise,
	\begin{align*}
		\norm{\hat{\mvz}-\mvz_i^+}_\infty
		 & \ge\norm{\hat{\mvz}-\mvm}_\infty-2n_i
		\ge\tfrac12\norm{\mvm}_\infty-1-2n_1\ge5n_1,
	\end{align*}
	for $n_1\ge1$, while the radii sum to at most $2n_1$. Thus $B_0$ is
	disjoint from every $B_i^+$. Finally, every $B_i^+$ lies within $\ell_\infty$
	distance $3n_1$ of $\mvm$ and every $B_j^-$ within $3n_1$ of $\mvzero$.
	Since $\norm{\mvm}_\infty\ge16n_1$, it follows that
	$B_i^+\cap B_j^-=\varnothing$.

	For~\eqref{eq:hub-hoplength} with $1\le i\le k-1$, and any $\mvy_i\in B_i$, $\mvy_{i+1}\in B_{i+1}$,
	\begin{align*}
		\norm{\mvy_i-\mvy_{i+1}}_\infty
		 & \le\norm{\mvz_i^+-\mvm}_\infty+\norm{\mvm-\mvz_{i+1}^+}_\infty+n_{i+1}+n_{i+2} \\
		 & \le2n_i+2n_{i+1}+n_{i+1}+n_{i+2}\le3n_i,                                       \\
		\text{ and }
		\norm{\mvy_i-\mvy_{i+1}}_\infty
		 & \ge\norm{\mvz_i^+-\mvm}_\infty-\norm{\mvz_{i+1}^+-\mvm}_\infty-n_{i+1}-n_{i+2} \\
		 & \ge n_i-2n_{i+1}-n_{i+1}-n_{i+2}
		\ge n_i-4n_{i+1}\ge\tfrac12 n_i.
	\end{align*}
	For $i=0$, $\norm{\mvy_0-\mvy_1}_\infty\le\norm{\hat{\mvz}-\mvm}_\infty+2n_1+n_1+n_2\le\tfrac12\norm{\mvm}_\infty+1+4n_1\le\norm{\mvm}_\infty+1\le n+2$, using $4n_1\le\tfrac14\norm{\mvm}_\infty$ and $\norm{\mvm}_\infty\le n\norm{\mvx}_\infty+1\le n+1$. Finally $\norm{\mvm-\mvy_k}_\infty\le\norm{\mvm-\mvz_k^+}_\infty+n_{k+1}\le2n_k+n_{k+1}\le3n_k$. Since $\norm{\cdot}_\infty\le\norm{\cdot}\le d\norm{\cdot}_\infty$, the $\ell_1$ statements~\eqref{eq:hub-hoplength}--\eqref{eq:hub-terminal} follow with $c_0:=3d$.
\end{proof}

\subsubsection*{Selection of the hubs}

Fix $\ell\in\dN$. For $-k\le i\le k$ and $1\le j\le\ell$ let
\begin{align}\label{eq:V_0x_0}
	V_i^{(1)}\ge V_i^{(2)}\ge\cdots\ge V_i^{(\ell)}
	\text{ and }
	\mvx_i^{(1)},\mvx_i^{(2)},\dots,\mvx_i^{(\ell)}\in B_i
\end{align}
denote the $\ell$ largest vertex weights in $B_i$, in decreasing order, and the locations attaining them, so that $V_i^{(j)}=V_{\mvx_i^{(j)}}$. Ties are broken by a fixed deterministic rule (this requires $\abs{B_i}\ge\ell$, which the regime-specific choices below guarantee). All of these are $\cF_V$-measurable. Set
\begin{align*}
	\mvx_0:=\mvx_0^{(1)},
\end{align*}
the heaviest vertex of the central ball, and define recursively, for $0\le i\le k-1$,
\begin{align}\label{eq:hub-selection}
	\mvx_{i+1}:=\mvx_{i+1}^{(J_{i+1})},
	\qquad
	J_{i+1}:=\argmin_{1\le j\le\ell}\go\big(\mvx_i,\mvx_{i+1}^{(j)}\big)
\end{align}
(smallest index in case of ties), and symmetrically on the left half, producing $\mvx_{-1},\dots,\mvx_{-k}$. The resulting path is
\begin{align}\label{eq:hub-path}
	\mvpi_k^{\mathsf{Hub}}
	:=\big\la\mvzero \mvx_{-k} \mvx_{-(k-1)} \cdots \mvx_0 \cdots \mvx_{k-1} \mvx_k \mvm\big\ra,
\end{align}
which is self-avoiding because the balls are disjoint. Consequently
\begin{align}\label{eq:Tn-Sk}
	T_n\le W_{\mvzero\mvx_{-k}}+S_k^-+S_k^++W_{\mvx_k\mvm},
	\qquad
	S_k^+:=\sum_{i=0}^{k-1}W_{\mvx_i\mvx_{i+1}},
	\quad
	S_k^-:=\sum_{i=0}^{k-1}W_{\mvx_{-i}\mvx_{-(i+1)}}.
\end{align}
We write $S_k:=S_k^+$. The same estimates will be applied to $S_k^-$ using the left-half construction.

The key structural observation is that the selection~\eqref{eq:hub-selection} produces an exact i.i.d.~sequence of $\ell$-fold minima.

\begin{lem}[Independence of the selected noises]\label{lem:hub-indep}
	Let $\cF_i^{\Hub}:=\gs\big(\cF_V, \go(\mvx_j,\mvx_{j+1}^{(l)}):0\le j<i, 1\le l\le\ell\big)$. Then $\mvx_i$ is $\cF_i^{\Hub}$-measurable and, conditionally on $\cF_i^{\Hub}$, the family $\{\go(\mvx_i,\mvx_{i+1}^{(l)})\}_{l\le\ell}$ is i.i.d.~with the law of $\go$. Consequently the sequence
	\begin{align*}
		\go_{i,[\ell]}:=\go(\mvx_i,\mvx_{i+1})=\min_{1\le l\le\ell}\go\big(\mvx_i,\mvx_{i+1}^{(l)}\big),
		\qquad 0\le i\le k-1,
	\end{align*}
	consists of i.i.d.~random variables with the law of $\go_{[\ell]}$, independent of $\cF_V$. The same holds on the left half, and the two families together are i.i.d.
\end{lem}

\begin{proof}
	Measurability of $\mvx_i$ is immediate by induction from~\eqref{eq:hub-selection}, since the points $\mvx_j^{(l)}$ are $\cF_V$-measurable. For the conditional law, note that the edges $\{\la\mvx_i\mvx_{i+1}^{(l)}\ra\}_{l\le\ell}$ join $B_i$ to $B_{i+1}$, whereas the edges generating $\cF_i^{\Hub}$ join $B_j$ to $B_{j+1}$ for $j<i$. Since the balls $B_0,\dots,B_k$ are pairwise disjoint by Lemma~\ref{lem:hub-geom}, the edge sets $B_j\times B_{j+1}$, $0\le j\le k-1$, are pairwise disjoint, so none of the $\ell$ edges in question has been examined. As $(\go_e)_{e\in\sE_{\mathrm{com}}}$ is i.i.d.~and independent of $\cF_V$, the conditional law of $\{\go(\mvx_i,\mvx_{i+1}^{(l)})\}_{l}$ given $\cF_i^{\Hub}$ is that of $\ell$ i.i.d.~copies of $\go$. Hence $\go_{i,[\ell]}$ has the law of $\go_{[\ell]}$ and is independent of $\cF_i^{\Hub}$. Iterating over $i$ gives the independence of the sequence. The two halves use edge sets on opposite sides, which are disjoint since $B_i^+\cap B_j^-=\varnothing$.
\end{proof}

\subsubsection*{The master estimate}
Fix $\tilde\gc>\gc$. By~\eqref{def:V_x}, there are constants $c_V>0$ and $x_V\ge1$, depending on $\tilde\gc$ and the law of $V$, such that
\begin{align}\label{eq:V-tilde-lower}
	\pr(V\ge x)\ge c_Vx^{-\tilde\gc}
	\text{ for all }x\ge x_V.
\end{align}
The exponent $\tilde\gc$ is fixed independently of $n$ and $k$; its dependence in the notation below is suppressed.
By Lemma~\ref{lem:hub-geom} the hop from $\mvx_i$ to $\mvx_{i+1}$ has length at most $c_0n_i$, so with the \emph{normalized hub weights}
\begin{align}\label{eq:hub-normalized}
	\hV_i:=\frac{V_{\mvx_i}}{n_{\abs{i}+1}^{d/\tilde\gc}},
	\qquad -k\le i\le k,
\end{align}
we obtain the deterministic bound
\begin{align}\label{eq:sum-gl}
	S_k
	=\sum_{i=0}^{k-1}\frac{\norm{\mvx_i-\mvx_{i+1}}^{\ga}\go(\mvx_i,\mvx_{i+1})}{V_{\mvx_i}V_{\mvx_{i+1}}}
	\le\sum_{i=0}^{k-1}\gl_i(n) \frac{\go(\mvx_i,\mvx_{i+1})}{\hV_i \hV_{i+1}}
\end{align}
with $\gl_i(n):=c_0^{\ga}n_i^{\ga}n_{i+1}^{-d/\tilde\gc}n_{i+2}^{-d/\tilde\gc}$.
The three exponents in $\gl_i(n)$ record the geometry exactly: the hop has length $n_i$, and it is discounted by the two hubs of the balls of radii $n_{i+1}$ and $n_{i+2}$.

\begin{lem}[Hub-chain master estimate]\label{lem:hub-ub}
	Assume \textup{\eqref{A3}} and fix $\tilde\gc>\gc$, with $c_V,x_V$ as in~\eqref{eq:V-tilde-lower}. Let $\vec v=\{v_i\}_{i=0}^k\subset(0,1)$ be a sequence of truncation thresholds and define the \emph{hub-good event}
	\begin{align}\label{eq:hub-good}
		\cG_n^{\Hub}(\vec v):=\bigcap_{i=-k}^{k}\Big\{V_i^{(\ell)}\ge v_{\abs{i}} n_{\abs{i}+1}^{d/\tilde\gc}\Big\}.
	\end{align}
	We have the following.
	\begin{enumeratei}
		\item \textbf{\textup{(Failure probability.)}} There are constants $C_1,c_2>0$, depending only on $d,\tilde\gc,\ell,c_V,x_V$, such that if
		\begin{align}\label{eq:hub-xV-cond}
			v_i n_{i+1}^{d/\tilde\gc}\ge x_V\text{ for all }0\le i\le k,
		\end{align}
		then
		\begin{align}\label{eq:hub-fail}
			\pr\Big(\big(\cG_n^{\Hub}(\vec v)\big)^{\rm c}\Big)
			\le2\sum_{i=0}^{k}C_1\exp\big(-c_2v_i^{-\tilde\gc}\big).
		\end{align}
		\item \textbf{\textup{(Deterministic cost bound.)}} With $\{\go_{i,[\ell]}\}_{i=0}^{k-1}$ the i.i.d.~sequence of Lemma~\ref{lem:hub-indep},
		\begin{align}\label{eq:Sk-det-cost}
			S_k^{+} \ind_{\cG_n^{\Hub}(\vec v)}
			\le\sum_{i=0}^{k-1}\frac{\gl_i(n)}{v_iv_{i+1}} \go_{i,[\ell]},
		\end{align}
		and the same bound holds for $S_k^-$ with an independent copy of the sequence.
	\end{enumeratei}
\end{lem}

\begin{proof}
	(i) Fix $0\le i\le k$ and let
	$N_i:=\big|\{\mvy\in B_i:V_{\mvy}\ge v_i n_{i+1}^{d/\tilde\gc}\}\big|$,
	so that $\{V_i^{(\ell)}<v_i n_{i+1}^{d/\tilde\gc}\}=\{N_i\le\ell-1\}$. Since the weights are i.i.d.~and $\abs{B_i}\asymp n_{i+1}^d$, we have $N_i\sim\mathrm{Bin}\big(\abs{B_i},\pr(V\ge v_i n_{i+1}^{d/\tilde\gc})\big)$, and by~\eqref{eq:hub-xV-cond} the lower bound~\eqref{eq:V-tilde-lower} applies, giving
	\begin{align}\label{eq:hub-ENi}
		\E N_i=\abs{B_i} \pr\big(V\ge v_i n_{i+1}^{d/\tilde\gc}\big)
		\ge c n_{i+1}^{d} c_V\big(v_i n_{i+1}^{d/\tilde\gc}\big)^{-\tilde\gc}
		=cc_V v_i^{-\tilde\gc}.
	\end{align}
	The crucial point is that the same fixed exponent $\tilde\gc$ is used in the threshold and in the tail lower bound. The right-hand side is therefore independent of the scale $n_{i+1}$, so a single sequence $\vec v$ can work simultaneously at all $k+1$ scales under~\eqref{def:V_x} alone. Set $c_2:=cc_V/8$. If $\E N_i\le2\ell$ then $v_i^{-\tilde\gc}\le2\ell/(cc_V)$, so $\exp(-c_2v_i^{-\tilde\gc})\ge\exp(-\ell/4)$ is bounded below and~\eqref{eq:hub-fail} is trivial after enlarging $C_1$. Otherwise $\ell-1\le\tfrac12\E N_i$ and the standard Chernoff bound for the binomial lower tail gives
	\begin{align*}
		\pr\big(V_i^{(\ell)}<v_i n_{i+1}^{d/\tilde\gc}\big)=\pr\big(N_i\le\ell-1\big)
		\le\exp\big(-\tfrac18\E N_i\big)
		\le C_1\exp\big(-c_2v_i^{-\tilde\gc}\big).
	\end{align*}
	The same estimate holds for $B_{-i}$, and a union bound over $i\in\{-k,\dots,k\}$ gives~\eqref{eq:hub-fail}.

	(ii) On $\cG_n^{\Hub}(\vec v)$ we have $V_{\mvx_i}\ge V_i^{(\ell)}\ge v_{\abs{i}}n_{\abs{i}+1}^{d/\tilde\gc}$, since $\mvx_i$ is one of the top $\ell$ vertices of $B_i$. That is, $\hV_i\ge v_{\abs{i}}$. Substituting into~\eqref{eq:sum-gl} and using $\go(\mvx_i,\mvx_{i+1})=\go_{i,[\ell]}$ yields~\eqref{eq:Sk-det-cost}.
\end{proof}

In the regime-specific proofs we first apply Lemma~\ref{lem:hub-ub}(i) to fix a truncation sequence with small total failure probability, then use~\eqref{eq:Sk-det-cost} to replace the random inverse hub weights by deterministic factors, and finally apply Lemma~\ref{lem:p-conS_k} to the resulting sum of i.i.d.~variables $\go_{i,[\ell]}$. The geometry is identical in Regimes II and $\rm IV_H$. Only $f_\phi$, $k(n)$ and $\vec v$ change. Figure~\ref{fig:hub-ansatz} shows the construction.

\begin{figure}[htbp]
	\centering
	\resizebox{\linewidth}{!}{%
		\begin{tikzpicture}[
				>=Latex,
				line cap=round,
				line join=round,
				font=\small,
				hubball/.style={draw=black, fill=black!0, line width=0.25pt},
				candvtx/.style={circle, fill=black!68, draw=black!68, inner sep=1pt},
				chosenvtx/.style={circle, fill=trajectoryblue, draw=trajectoryblue, inner sep=1.5pt},
				chosenring/.style={draw=trajectoryblue, line width=0.5pt},
				candedge/.style={draw=networkgreen, line width=.75pt},
				hubpath/.style={draw=trajectoryblue, line width=1.5pt},
				annot/.style={align=left, fill=white, fill opacity=.96, text opacity=1,inner sep=2.0pt, rounded corners=1pt}
			]

			\def\ang{8}
			\def\gap{0.17}
			\def\rO{2.85}
			\def\rA{1.33}
			\def\rB{0.84}
			\def\rC{0.52}

			\pgfmathsetmacro{\dA}{\rO+\rA+\gap}
			\pgfmathsetmacro{\dB}{\rA+\rB+\gap}
			\pgfmathsetmacro{\dC}{\rB+\rC+\gap}

			\pgfmathsetmacro{\laboffO}{\rO+0.28}
			\pgfmathsetmacro{\laboffA}{\rA+0.24}
			\pgfmathsetmacro{\laboffB}{\rB+0.21}
			\pgfmathsetmacro{\laboffC}{\rC+0.19}

			\coordinate (z0)  at (0,0);
			\coordinate (z1)  at ($(z0)+(\ang:\dA)$);
			\coordinate (z2)  at ($(z1)+(\ang:\dB)$);
			\coordinate (z3)  at ($(z2)+(\ang:\dC)$);
			\coordinate (zm1) at ($(z0)+(180+\ang:\dA)$);
			\coordinate (zm2) at ($(zm1)+(180+\ang:\dB)$);
			\coordinate (zm3) at ($(zm2)+(180+\ang:\dC)$);

			\coordinate (L) at ($(zm3)+(180+\ang:1.12)$);
			\coordinate (R) at ($(z3)+(\ang:1.90)$);

			\node[below left=1pt and 1pt] at (L) {$\mvzero$};
			\node[below right=1pt and 1pt] at (R) {$\lceil n\mvx\rceil$};

			\draw[hubball] (zm3) circle (\rC);
			\draw[hubball] (zm2) circle (\rB);
			\draw[hubball] (zm1) circle (\rA);
			\draw[hubball] (z0)  circle (\rO);
			\draw[hubball] (z1)  circle (\rA);
			\draw[hubball] (z2)  circle (\rB);
			\draw[hubball] (z3)  circle (\rC);

			\node at ($(zm3)+(90+\ang:\laboffC)$) {$B_{-3}$};
			\node at ($(zm2)+(90+\ang:\laboffB)$) {$B_{-2}$};
			\node at ($(zm1)+(90+\ang:\laboffA)$) {$B_{-1}$};
			\node at ($(z0)+(90+\ang:\laboffO)$) {$B_0$};
			\node at ($(z1)+(90+\ang:\laboffA)$) {$B_1$};
			\node at ($(z2)+(90+\ang:\laboffB)$) {$B_2$};
			\node at ($(z3)+(90+\ang:\laboffC)$) {$B_3$};

			\coordinate (xm3) at ($(zm3)+(58:0.14)$);
			\coordinate (xm2) at ($(zm2)+(104:0.22)$);
			\coordinate (xm1) at ($(zm1)+(92:0.44)$);
			\coordinate (x0)  at ($(z0)+(99:0.34)$);
			\coordinate (x1)  at ($(z1)+(111:0.38)$);
			\coordinate (x2)  at ($(z2)+(108:0.20)$);
			\coordinate (x3)  at ($(z3)+(126:0.14)$);

			\node[candvtx] (m31a) at ($(zm3)+(148:0.14)$) {};
			\node[candvtx] (m31b) at ($(zm3)+(232:0.15)$) {};

			\node[candvtx] (m21a) at ($(zm2)+(150:0.27)$) {};
			\node[candvtx] (m21b) at ($(zm2)+(246:0.27)$) {};
			\node[candvtx] (m21c) at ($(zm2)+(18:0.20)$) {};

			\node[candvtx] (m11a) at ($(zm1)+(154:0.61)$) {};
			\node[candvtx] (m11b) at ($(zm1)+(238:0.63)$) {};
			\node[candvtx] (m11c) at ($(zm1)+(10:0.56)$) {};

			\node[candvtx] at ($(z0)+(154:0.98)$) {};
			\node[candvtx] at ($(z0)+(240:0.84)$) {};
			\node[candvtx] at ($(z0)+(334:0.92)$) {};

			\node[candvtx] (p11a) at ($(z1)+(150:0.52)$) {};
			\node[candvtx] (p11b) at ($(z1)+(338:0.46)$) {};
			\node[candvtx] (p11c) at ($(z1)+(28:0.40)$) {};

			\node[candvtx] (p21a) at ($(z2)+(152:0.25)$) {};
			\node[candvtx] (p21b) at ($(z2)+(15:0.18)$) {};

			\node[candvtx] (p31a) at ($(z3)+(154:0.12)$) {};
			\node[candvtx] (p31b) at ($(z3)+(238:0.12)$) {};

			\node at ($(L)+(\ang:0.45)$) {$\cdots$};
			\node at ($(R)+(180+\ang:0.70)$) {$\cdots$};

			\draw[candedge] (x0) .. controls ($(x0)+(-0.55,1.05)$) and ($(m11a)+(0.35,0.25)$).. (m11a);
			\draw[candedge] (x0) .. controls ($(x0)+(-0.92,0.78)$) and ($(m11c)+(0.42,-0.02)$).. (m11c);
			\draw[candedge] (x0) .. controls ($(x0)+(-1.02,-0.10)$) and ($(m11b)+(0.42,0.10)$).. (m11b);

			\draw[candedge] (xm1).. controls ($(xm1)+(-0.65,0.55)$) and ($(m21a)+(0.28,0.12)$).. (m21a);
			\draw[candedge] (xm1).. controls ($(xm1)+(-0.82,0.12)$) and ($(m21c)+(0.28,0.02)$).. (m21c);
			\draw[candedge] (xm1).. controls ($(xm1)+(-0.82,-0.40)$) and ($(m21b)+(0.28,0.10)$).. (m21b);

			\draw[candedge] (xm2).. controls ($(xm2)+(-0.40,0.24)$) and ($(m31a)+(0.18,0.05)$).. (m31a);
			\draw[candedge] (xm2).. controls ($(xm2)+(-0.46,-0.18)$) and ($(m31b)+(0.18,0.04)$).. (m31b);

			\draw[candedge] (x0) .. controls ($(x0)+(0.60,1.00)$) and ($(p11c)+(-0.35,0.22)$).. (p11c);
			\draw[candedge] (x0) .. controls ($(x0)+(0.92,0.52)$) and ($(p11a)+(-0.35,0.08)$).. (p11a);
			\draw[candedge] (x0) .. controls ($(x0)+(0.95,-0.10)$) and ($(p11b)+(-0.35,0.06)$).. (p11b);

			\draw[candedge] (x1) .. controls ($(x1)+(0.48,0.28)$) and ($(p21b)+(-0.28,0.08)$).. (p21b);
			\draw[candedge] (x1) .. controls ($(x1)+(0.62,-0.12)$) and ($(p21a)+(-0.28,0.05)$).. (p21a);

			\draw[candedge] (x2) .. controls ($(x2)+(0.34,0.15)$) and ($(p31a)+(-0.18,0.04)$).. (p31a);
			\draw[candedge] (x2) .. controls ($(x2)+(0.34,-0.12)$) and ($(p31b)+(-0.18,0.04)$).. (p31b);

			\foreach \hubfrom/\hubto/\hubrise in {
					L/xm3/0.28,
					xm3/xm2/0.52,
					xm2/xm1/0.78,
					xm1/x0/1.35,
					x0/x1/1.35,
					x1/x2/0.78,
					x2/x3/0.52,
					x3/R/0.38
				}{%
					\draw[hubpath]
					(\hubfrom)
					.. controls ($(\hubfrom)!0.333!(\hubto)+(0,\hubrise)$)
					and ($(\hubfrom)!0.667!(\hubto)+(0,\hubrise)$)
					.. (\hubto);
				}

			\tikzset{
				hubxlabel/.style={
						text=trajectoryblue,
						fill=white,
						fill opacity=0.94,
						text opacity=1,
						inner sep=0.8pt
					}
			}

			\foreach \hubpoint/\hubindex/\hubradius/\hubdrop in {
					xm3/-3/0.075/0.82,
					xm2/-2/0.085/0.66,
					xm1/-1/0.095/0.34,
					x0/0/0.10/0.34,
					x1/1/0.095/0.34,
					x2/2/0.085/0.43,
					x3/3/0.075/0.42
				}{%
					\node[chosenvtx] at (\hubpoint) {};
					\draw[chosenring] (\hubpoint) circle (\hubradius);
					\node[hubxlabel] at ($(\hubpoint)+(0,-\hubdrop)$)
					{$\mvx_{\hubindex}$};
				}

			\node[annot, anchor=south west] (n0) at ($(zm1)+(-1.70,2.55)$)
			{$\mvx_0=\arg\max_{\mvy\in B_0}V_{\mvy}$};
			\draw[-{Latex[length=2mm]}] (n0.south east).. controls ($(n0.south east)+(0.55,-0.10)$) and ($(x0)+(-0.18,0.32)$).. (x0);

			\node[annot, anchor=west] (n1) at ($(z1)+(0.55,2.10)$)
			{$\mvx_i=\arg\min_{1\le j\le \ell}\go\bigl(\mvx_{i-1},\mvx_i^{(j)}\bigr)$};
			\draw[-{Latex[length=2mm]}] (n1.south west).. controls ($(n1.south west)+(-0.30,-0.45)$) and ($(p11c)+(-0.05,0.28)$).. ($(x0)!0.54!(p11c)$);

			\node[annot, anchor=north west] (n2) at ($(z1)+(-0.18,-1.55)$)
			{top $\ell$ vertices in $B_i$\\
				$\{\mvx_i^{(1)},\ldots,\mvx_i^{(\ell)}\}$};
			\draw[-{Latex[length=2mm]}] (n2.north east).. controls ($(n2.north east)+(0.22,0.58)$) and ($(p11b)+(0.00,-0.28)$).. (p11b);

			\node[annot, anchor=west] at ($(zm1)+(-0.30,-3.10)$)
			{$n_i=f^{(i)}(n), \qquad B_i=B(\mvz_i,n_{i+1}), \qquad P_n=\{\mvx_{-k},\ldots,\mvx_{-1},\mvx_0,\mvx_1,\ldots,\mvx_k\}$};

		\end{tikzpicture}%
	}
	\caption{The Hub-Chain Ansatz for Regimes II and $\rm IV_H$. One first chooses $\mvx_0$ as the maximizer of $V$ inside the central ball $B_0$. Then, in each subsequent ball $B_i$, one keeps the top $\ell$ vertices (ranked by $V$) and selects the next point by minimizing the edge weight $\go$ among those candidates. The thick blue curves indicate the selected direct-hop Ansatz path, while the thin green curves indicate the candidate edges compared at each step.}
	\label{fig:hub-ansatz}
\end{figure}
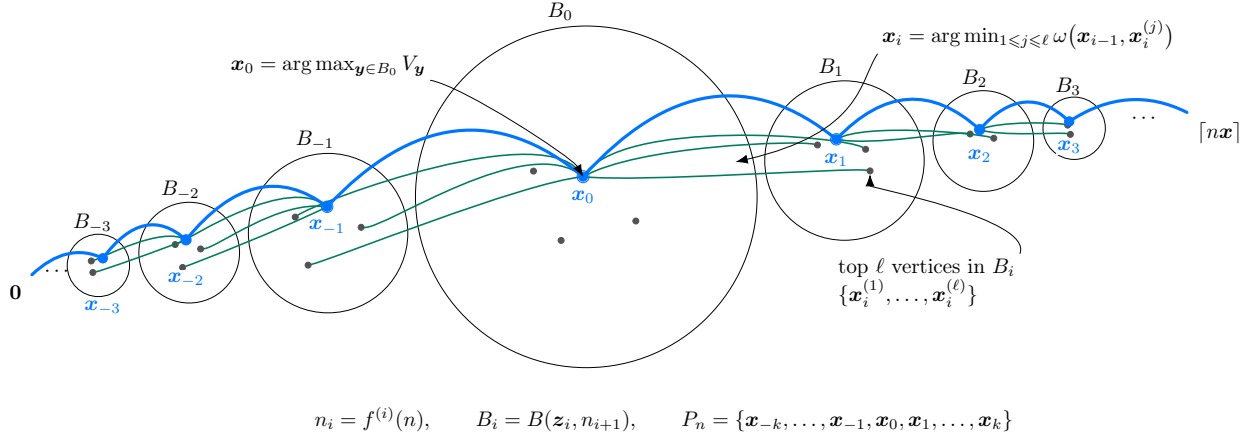

\subsection{The Binary Edge-Bridge multi-scale Ansatz for the edge-dominated regimes}\label{ssec:multi-scale-ansatz}

This is the construction used in Regimes III and $\rm IV_E$, where the upper bound comes from repeatedly minimizing edge noises over many candidate pairs. We build a recursive binary path from $\mvzero$ to $\mvm=\lceil n\mvx\rceil$.

Throughout this subsection we work with a scale map $f_\phi$ and a depth $k=k(n)$ such that
\begin{align}\label{eq:edge-scale-separation}
	n_0=n,\qquad n_i=f_\phi^{(i)}(n),\qquad n_{i+1}\le\tfrac{1}{16}n_i\quad(0\le i\le k-1),
\end{align}
and, if $\mvx\ne\mvzero$,
\begin{align}\label{eq:root-separation}
	n_1\le\tfrac{1}{16}\norm{\mvm}_\infty
\end{align}
for all large $n$. Both hold for the choices made in Sections~\ref{sec:proof-III} and~\ref{sec:proof-IVE}. The case $\mvx=\mvzero$ is trivial. Fix $a>0$ with $p_a=\pr(V>a)>0$, recall $\cV^a=\{\mvz:V_{\mvz}>a\}$, and put
\begin{align*}
	B^a(\mvy,r):=B(\mvy,r)\cap\cV^a,
	\qquad
	\wt B^a(\mvy,r):=\wt B(\mvy,r)\cap\cV^a,
	\qquad \wt B(\mvy,r)=B(\mvy,r)\setminus B(\mvy,r/2).
\end{align*}
All bridge endpoints are chosen from $\cV^a$, which guarantees a uniform lower bound $a$ on the vertex weights at the ends of every bridge and thereby removes the vertex environment from the analysis. This is why the construction is insensitive to $\gc$.

\subsubsection*{Percolated shells}

Fix $\eps_0\in(0,1)$. Since $\bigl|\wt B^a(\mvy,r)\bigr|\sim\mathrm{Bin}\big(\bigl|\wt B(\mvy,r)\bigr|,p_a\big)$ for deterministic $\mvy$, Hoeffding's inequality gives constants $c_1,c_2>0$ depending only on $a,\eps_0,d$ such that for every deterministic $\mvy$ and every $r\ge1$,
\begin{align}\label{eq:hoeffding-Va-ball}
	\pr\Big(\bigl|\wt B^a(\mvy,r)\bigr|\le(1-\eps_0)p_a\bigl|\wt B(\mvy,r)\bigr|\Big)\le c_1\exp\{-c_2r^d\}.
\end{align}
The centers arising in the construction are random, so we require a corresponding estimate for random centers.

\begin{lem}[Hoeffding for a uniformly sampled favorable center]\label{lem:uniform-center}
	Let $S\subset\dZ^d$ be finite and deterministic and put $S^a:=S\cap\cV^a$. On $\{\bigl|S^a\bigr|\ge m\}$ let $U$ be a random vertex which, conditionally on the vertex weights, is uniform on $S^a$. Then for every $r\ge1$,
	\begin{align*}
		\pr\Big(\bigl|\wt B^a(U,r)\bigr|\le(1-\eps_0)p_a\bigl|\wt B(\mvzero,r)\bigr|, \bigl|S^a\bigr|\ge m\Big)
		\le\frac{\bigl|S\bigr|}{m} c_1\exp\{-c_2r^d\}.
	\end{align*}
\end{lem}

\begin{proof}
	Write $X_{\mvz}:=\ind_{\{V_{\mvz}>a\}}$ and call $\mvu$ \emph{bad} if $\bigl|\wt B^a(\mvu,r)\bigr|\le(1-\eps_0)p_a\bigl|\wt B(\mvzero,r)\bigr|$. On $\{\bigl|S^a\bigr|\ge m\}$,
	\begin{align*}
		\pr\big(U\text{ is bad}\mid(V_{\mvz})_{\mvz}\big)
		=\frac{1}{\bigl|S^a\bigr|}\sum_{\mvu\in S}X_{\mvu}\ind_{\{\mvu\text{ bad}\}}
		\le\frac1m\sum_{\mvu\in S}X_{\mvu}\ind_{\{\mvu\text{ bad}\}}.
	\end{align*}
	Taking expectations and applying~\eqref{eq:hoeffding-Va-ball} to each of the $\bigl|S\bigr|$ deterministic centers gives the claim.
\end{proof}

\subsubsection*{The binary tree of segments and bridges}

For $i\ge1$ let $\Sigma_i:=\{0,1\}^i$ and $\Sigma_0:=\{\varnothing\}$. Suppressing the empty word at the root, set
\begin{align*}
	\mvu_0:=\mvzero,\qquad\mvu_1:=\mvm.
\end{align*}
Suppose that for some $0\le i\le k-1$ and every $\gs\in\Sigma_i$ the segment endpoints $\mvu_{\gs0},\mvu_{\gs1}$ have been defined. Put
\begin{align*}
	\wt B_{\gs j}:=\wt B\big(\mvu_{\gs j},n_{i+1}\big),
	\qquad
	\wt B_{\gs j}^a:=\wt B_{\gs j}\cap\cV^a,
	\qquad j\in\{0,1\}.
\end{align*}
To make the selection unambiguous even when the law of $\go$ has atoms, enlarge the probability space by i.i.d.~continuous tie-breaking marks $(U_e)_{e\in\sE_{\mathrm{com}}}$, independent of all model variables. If there is at least one edge joining $\wt B_{\gs0}^a$ to $\wt B_{\gs1}^a$, choose the edge that minimizes the pair $(\go_e,U_e)$ in lexicographic order and write its endpoints as $\mvu_{\gs01}\in\wt B_{\gs0}^a$ and $\mvu_{\gs10}\in\wt B_{\gs1}^a$, so that
\begin{align}\label{eq:bridge-min}
	\go_{\mvu_{\gs01}\mvu_{\gs10}}
	=\min\big\{\go_{\mvu\mvv}:\mvu\in\wt B_{\gs0}^a, \mvv\in\wt B_{\gs1}^a, \mvu\ne\mvv\big\}.
\end{align}
The continuous marks make the choice exchangeable and ensure that, conditionally on the candidate sets and on the past, each selected endpoint has the uniform marginal on its candidate shell. Otherwise define the two points arbitrarily. This fallback is not used on the good event below. The two child segments of $\gs$ are $(\mvu_{\gs00},\mvu_{\gs01})$ and $(\mvu_{\gs10},\mvu_{\gs11})$, where $\mvu_{\gs00}:=\mvu_{\gs0}$ and $\mvu_{\gs11}:=\mvu_{\gs1}$. Whenever the selected endpoints lie in their prescribed shells,
\begin{align}\label{eq:seg-scale}
	\tfrac12n_{i+1}\le\norm{\mvu_{\gs00}-\mvu_{\gs01}}_\infty\le n_{i+1},
	\qquad
	\tfrac12n_{i+1}\le\norm{\mvu_{\gs10}-\mvu_{\gs11}}_\infty\le n_{i+1},
\end{align}
so that every terminal segment at depth $k$ has length of order $n_k$, and inductively
\begin{align}\label{eq:seg-scale-i}
	\tfrac12 n_i\le\norm{\mvu_{\gs0}-\mvu_{\gs1}}_\infty\le n_i,
	\qquad \gs\in\Sigma_i, 1\le i\le k,
\end{align}
with $\norm{\mvu_0-\mvu_1}_\infty=\norm{\mvm}_\infty\ge16n_1$ at the root.

\subsubsection*{The percolation good event}

For $0\le i\le k-1$, $\gs\in\Sigma_i$ and $j\in\{0,1\}$ set
$\cA_{\gs j}:=\big\{\bigl|\wt B_{\gs j}^a\bigr|\ge(1-\eps_0)p_a\bigl|\wt B_{\gs j}\bigr|\big\}$, and
\begin{align}\label{eq:cA-k}
	\cA_k:=\bigcap_{i=0}^{k-1}\bigcap_{\gs\in\Sigma_i}\big(\cA_{\gs0}\cap\cA_{\gs1}\big).
\end{align}
There is $c_{\rm sh}>0$ depending only on $a,\eps_0,d$ such that on $\cA_k$
\begin{align}\label{eq:fav-shell-lower}
	\bigl|\wt B_{\gs j}^a\bigr|\ge c_{\rm sh}n_{i+1}^{d},
	\qquad 0\le i\le k-1, \gs\in\Sigma_i, j\in\{0,1\},
\end{align}
and, after decreasing $c_{\rm sh}$, the corresponding candidate edge sets contain at least $c_{\rm sh}n_{i+1}^{2d}$ edges.

\begin{lem}[Probability of the percolation good event]\label{lem:Ak-prob}
	There are constants $C\ge2$ and $c>0$ depending only on $a,\eps_0,d$ such that
	\begin{align}\label{eq:cA-k-prob}
		\pr(\cA_k^{\rm c})\le C^k\exp\{-cn_k^d\}.
	\end{align}
\end{lem}

\begin{proof}
	Write $\cA_{\le i}$ for the intersection in~\eqref{eq:cA-k} restricted to levels $0,\dots,i$. We claim that for every $0\le i\le k-1$, $\gs\in\Sigma_i$ and $j\in\{0,1\}$,
	\begin{align}\label{eq:one-shell-bad}
		\pr\big(\cA_{\gs j}^{\rm c}\cap\cA_{\le i-1}\big)\le C_0^{ i} c_1\exp\{-c_2n_{i+1}^d\},
		\qquad C_0:=\frac{\sup_r \bigl|\wt B(\mvzero,r)\bigr|}{c_{\rm sh}r^d}<\infty ,
	\end{align}
	with $\cA_{\le-1}$ the sure event. A union bound over the $2^{i+1}$ pairs $(\gs,j)$ at level $i$ and over $i$ then gives~\eqref{eq:cA-k-prob} with $C:=4C_0$, using $n_{i+1}\ge n_k$.

	The point requiring care is that the center $\mvu_{\gs j}$ of the shell is random, and so is the center of the shell in which it was itself selected. Every center is created exactly once, as the endpoint of a minimizing edge~\eqref{eq:bridge-min} at some level $i_1\le i-1$, inside the favorable shell $\wt B^a(\mvp_1,n_{i_1+1})$ of its \emph{parent center} $\mvp_1$. The parent was in turn created at a level $i_2<i_1$ inside the shell of its parent $\mvp_2$, and so on, until one reaches one of the two deterministic root centers $\mvzero,\mvm$ after at most $i$ steps. (If $\mvu_{\gs j}$ is itself a root center the claim is~\eqref{eq:hoeffding-Va-ball}.) Call a vertex $\mvw$ \emph{bad} if $\bigl|\wt B^a(\mvw,n_{i+1})\bigr|\le(1-\eps_0)p_a\bigl|\wt B(\mvw,n_{i+1})\bigr|$. This is an $\cF_V$-measurable property of $\mvw$, and $\cA_{\gs j}^{\rm c}=\{\mvu_{\gs j}\text{ is bad}\}$.

	Condition on $\cF_{i_1,\tau}$, the past of the bridge that created $\mvu_{\gs j}$ (this $\gs$-field is defined in the next subsection). By the exchangeable choice of the minimizing edge and the conditional i.i.d.~property of the candidate noises (Lemma~\ref{lem:edge-disjoint} and the paragraph following it, neither uses the present lemma), $\mvu_{\gs j}$ is uniform on $\wt B^a(\mvp_1,n_{i_1+1})$, a set of cardinality $\ge c_{\rm sh}n_{i_1+1}^d$ on $\cA_{\le i_1}$. Exactly as in the proof of Lemma~\ref{lem:uniform-center},
	\begin{align*}
		\pr\big(\mvu_{\gs j}\text{ bad}, \cA_{\le i-1}\big)
		\le\frac{1}{c_{\rm sh}n_{i_1+1}^d}
		\E\Biggl[\ind_{\cA_{\le i_1-1}}\sum_{\mvw\in\wt B(\mvp_1,n_{i_1+1})}\ind_{\{\mvw\text{ bad}\}}\Biggr].
	\end{align*}
	The summand is now a function of the random center $\mvp_1$ and of $\cF_V$, and we repeat the argument for $\mvp_1$: conditionally on the past of the bridge that created it, $\mvp_1$ is uniform on $\wt B^a(\mvp_2,n_{i_2+1})$, whence
	\begin{align*}
		\E\Biggl[\ind_{\cA_{\le i_1-1}}\sum_{\mvw\in\wt B(\mvp_1,n_{i_1+1})}\ind_{\{\mvw\text{ bad}\}}\Biggr]
		\le\frac{1}{c_{\rm sh}n_{i_2+1}^d}
		\E\Biggl[\ind_{\cA_{\le i_2-1}}\sum_{\mvp\in\wt B(\mvp_2,n_{i_2+1})} \sum_{\mvw\in\wt B(\mvp,n_{i_1+1})}\ind_{\{\mvw\text{ bad}\}}\Biggr].
	\end{align*}
	After at most $i$ iterations the outermost center is deterministic, all summation sets are deterministic, and each of the at most $\prod_r\bigl|\wt B(\cdot,n_{i_r+1})\bigr|$ terms is bounded by~\eqref{eq:hoeffding-Va-ball} with $r=n_{i+1}$. Each iteration costs a factor $\bigl|\wt B(\cdot,n_{i_r+1})\bigr|/(c_{\rm sh}n_{i_r+1}^d)\le C_0$, which proves~\eqref{eq:one-shell-bad}.
\end{proof}

The factor $C^k$ in~\eqref{eq:cA-k-prob} is harmless in both applications: it is $(\log n)^{O(1)}$ in Regime III and $n^{O(1)}$ in Regime $\rm IV_E$, against a stretched-exponential, respectively exponential, decay in $n_k^d$.

\subsubsection*{Ordering of the bridges, and disjointness}

Let $\cI_k:=\{(i,\gs):0\le i\le k-1, \gs\in\Sigma_i\}$, ordered by level and then lexicographically. We write $\prec$ for this order. For $(i,\gs)\in\cI_k$ let
\begin{align*}
	\sE_{i,\gs}:=\big\{\{\mvu,\mvv\}:\mvu\in\wt B_{\gs0}^a, \mvv\in\wt B_{\gs1}^a, \mvu\ne\mvv\big\},
	\qquad
	N_{i,\gs}:=\abs{\sE_{i,\gs}},
\end{align*}
and define the past
\begin{align*}
	\cF_{i,\gs}:=\gs\Big(\big(V_{\mvz}\big)_{\mvz\in\dZ^d},
	\{(\go_e,U_e):(j,\tau)\prec(i,\gs),\ e\in\sE_{j,\tau}\}\Big),
\end{align*}
so that the construction of $\sE_{i,\gs}$ is $\cF_{i,\gs}$-measurable.

\begin{lem}[Disjointness of candidate edge sets]\label{lem:edge-disjoint}
	On $\cA_k$, the sets $\sE_{i,\gs}$, $(i,\gs)\in\cI_k$, are pairwise disjoint. Moreover every edge of $\sE_{i,\gs}$ has $\ell_\infty$ length in $\big[\tfrac38n_i, 2n_i\big]$ for $1\le i\le k-1$, and every edge of the root set $\sE_{0,\varnothing}$ has $\ell_\infty$ length in $\big[\norm{\mvm}_\infty-2n_1, \norm{\mvm}_\infty+2n_1\big]\subseteq\big[14n_1, 2n_0\big]$.
\end{lem}

\begin{proof}
	Let $\mvu\in\wt B_{\gs0}^a$ and $\mvv\in\wt B_{\gs1}^a$ with $\gs\in\Sigma_i$. At the root, $\norm{\mvu-\mvv}_\infty$ differs from $\norm{\mvm}_\infty$ by at most $2n_1$, and $14n_1\le\norm{\mvm}_\infty-2n_1$, $\norm{\mvm}_\infty+2n_1\le n+1+n/8\le2n_0$ by~\eqref{eq:root-separation}. For $1\le i\le k-1$, by~\eqref{eq:seg-scale-i} and $n_{i+1}\le n_i/16$,
	\begin{align*}
		\norm{\mvu-\mvv}_\infty
		 & \ge\norm{\mvu_{\gs0}-\mvu_{\gs1}}_\infty-2n_{i+1}\ge\tfrac12n_i-\tfrac18n_i=\tfrac38n_i, \\
		\norm{\mvu-\mvv}_\infty
		 & \le\norm{\mvu_{\gs0}-\mvu_{\gs1}}_\infty+2n_{i+1}\le n_i+\tfrac18n_i\le2n_i,
	\end{align*}
	which is the length statement. For two indices at the same level $i$ with $\gs\ne\tau$, we claim that every endpoint of the segment $\gs$ is at $\ell_\infty$ distance at least $4n_i$ from every endpoint of the segment $\tau$. Since the shells have radius $n_{i+1}\le n_i/16$, they are then disjoint and $\sE_{i,\gs}\cap\sE_{i,\tau}=\varnothing$. To see the claim, let $\rho\in\Sigma_{i'}$, $i'<i$, be the longest common prefix of $\gs$ and $\tau$, so that (say) $\gs$ descends from the child segment $(\mvu_{\rho00},\mvu_{\rho01})$ and $\tau$ from $(\mvu_{\rho10},\mvu_{\rho11})$. By~\eqref{eq:seg-scale}, every endpoint of a descendant of a segment at level $i'+1$ lies within $\sum_{j\ge i'+2}n_j\le\tfrac{16}{15}n_{i'+2}$ of one of that segment's two endpoints. The four points $\mvu_{\rho0},\mvu_{\rho01}$ and $\mvu_{\rho10},\mvu_{\rho1}$ satisfy $\norm{\mvu_{\rho01}-\mvu_{\rho10}}_\infty\ge\tfrac38n_{i'}$ (bridge length) and $\norm{\mvu_{\rho0}-\mvu_{\rho1}}_\infty\ge\tfrac12n_{i'}$, with $\mvu_{\rho01}$ within $n_{i'+1}$ of $\mvu_{\rho0}$ and $\mvu_{\rho10}$ within $n_{i'+1}$ of $\mvu_{\rho1}$, so any point of the first pair is at distance at least $\tfrac38n_{i'}-n_{i'+1}\ge\tfrac{5}{16}n_{i'}$ from any point of the second. Hence the endpoints of $\gs$ and $\tau$ are at distance at least $\tfrac5{16}n_{i'}-\tfrac{32}{15}n_{i'+2}\ge\tfrac14n_{i'}\ge4n_{i'+1}\ge4n_i$. For two different levels $i<j$, every edge of $\sE_{j,\tau}$ has length at most $2n_j\le2n_{i+1}\le\tfrac18n_i<\tfrac38n_i$, while every edge of $\sE_{i,\gs}$ has length at least $\tfrac38n_i$ (at least $14n_1>2n_1$ if $i=0$). So the two sets are disjoint.
\end{proof}

Since the edge noises are i.i.d.~and independent of the vertex weights, Lemma~\ref{lem:edge-disjoint} implies that on $\cA_k$, conditionally on $\cF_{i,\gs}$, the family $\{\go_e:e\in\sE_{i,\gs}\}$ is i.i.d.~with the law of $\go$ and independent of the past. In particular the bridge minima at a given level are conditionally independent given the earlier levels.

\subsubsection*{Ansatz decomposition}

For $\gs\in\Sigma_i$ with $0\le i\le k-1$ set $\cW_\gs:=W_{\mvu_{\gs01}\mvu_{\gs10}}$, the weight of the bridge. At depth $k$ we stop and keep the terminal segments $\{\mvu_{\gs0},\mvu_{\gs1}\}$, $\gs\in\Sigma_k$. For each such $\gs$ let $\mvpi_\gs^{\loc}$ be a finite path from $\mvu_{\gs0}$ to $\mvu_{\gs1}$, and define recursively
\begin{align*}
	\mvpi_\gs:=
	\begin{cases}
		\mvpi_\gs^{\loc},                                                  & \abs{\gs}=k,          \\[2pt]
		\mvpi_{\gs0}\circ\la\mvu_{\gs01} \mvu_{\gs10}\ra\circ\mvpi_{\gs1}, & 0\le\abs{\gs}\le k-1.
	\end{cases}
\end{align*}
If the concatenated walk repeats a vertex, erase loops. This only decreases the passage time and yields an admissible path from $\mvzero$ to $\mvm$. With
\begin{align}\label{eq:Sloc-Sbr-def}
	S_n^{\br}:=\sum_{i=0}^{k-1}\sum_{\gs\in\Sigma_i}\cW_\gs,
	\qquad
	S_n^{\loc}:=\sum_{\gs\in\Sigma_k}T(\mvu_{\gs0},\mvu_{\gs1}),
\end{align}
taking the infimum over the terminal local paths gives, on $\cA_k$,
\begin{align}\label{eq:Tn-Sn}
	T_n\le S_n^{\br}+S_n^{\loc}.
\end{align}

\subsubsection*{Bridge truncation}

For $0\le i\le k-1$ and $\gs\in\Sigma_i$ put $Y_{i,\gs}:=\min\{\go_e:e\in\sE_{i,\gs}\}$. On $\cA_k$ we have $N_{i,\gs}\ge c_{\rm sh}n_{i+1}^{2d}$, and since the bridge endpoints lie in $\cV^a$ and have separation at most $2n_i$ by Lemma~\ref{lem:edge-disjoint},
\begin{align}\label{eq:bridge-M}
	\cW_\gs\le Ca^{-2}n_i^{\ga} Y_{i,\gs}.
\end{align}

\begin{lem}[Simultaneous bridge truncation]\label{lem:bridge-trunc}
	Fix $\eps'>0$ and $R>0$, and for $0\le i\le k-1$ set $N_i:=\lfloor c_{\rm sh}n_{i+1}^{2d}\rfloor$. Then there are $C<\infty$ and $n_0<\infty$, depending on $a,\eps_0,d,\ga,\theta,\eps',R$ and the lower tail of $\go$, such that for all $n\ge n_0$ there is an event $\cB_k$ with
	\begin{align}\label{eq:bridge-trunc-prob}
		\pr\big(\cB_k^{\rm c}\cap\cA_k\big)\le\sum_{i=0}^{k-1}2^iN_i^{-R}\le C\sum_{i=0}^{k-1}2^in_{i+1}^{-2dR},
	\end{align}
	and, on $\cA_k\cap\cB_k$,
	\begin{align}\label{eq:bridge-trunc-bound}
		\cW_\gs\le\gl_i,\qquad 0\le i\le k-1, \gs\in\Sigma_i,
		\text{ where }
		\gl_i:=C n_i^{\ga} n_{i+1}^{-2d/\theta+\eps'}.
	\end{align}
\end{lem}

\begin{proof}
	Choose $s\in(0,1)$ so small that $\tfrac{2d}{\theta}-\tfrac{2d}{\theta+s}<\tfrac{\eps'}{2}$. By~\eqref{def:go} there is $c_s>0$ with $\pr(\go\le t)\ge c_st^{\theta+s}$ for all small $t>0$. Set
	\begin{align*}
		t_i:=\Big(\frac{R\log N_i}{c_sN_i}\Big)^{1/(\theta+s)},
		\qquad
		\cB_k:=\bigcap_{i=0}^{k-1}\bigcap_{\gs\in\Sigma_i}\{Y_{i,\gs}\le t_i\}.
	\end{align*}
	Since $n_k\to\infty$, all $t_i$ eventually lie in the range where the lower-tail bound applies. Fix $(i,\gs)$. On $\cA_k$ we have $N_{i,\gs}\ge N_i$, and the set $\sE_{i,\gs}$ is $\cF_{i,\gs}$-measurable with conditionally i.i.d.~noises, so
	\begin{align*}
		\pr\big(Y_{i,\gs}>t_i, \cA_k\big)
		 & \le\E\Big[\ind_{\{N_{i,\gs}\ge N_i\}}\pr(\go>t_i)^{N_{i,\gs}}\Big] \\
		 & \le\exp\big\{-N_i\pr(\go\le t_i)\big\}
		\le\exp\{-R\log N_i\}=N_i^{-R},
	\end{align*}
	using $(1-p)^N\le e^{-Np}$. A union bound over the $2^i$ bridges at level $i$ gives~\eqref{eq:bridge-trunc-prob}. On $\cA_k\cap\cB_k$,~\eqref{eq:bridge-M} gives $\cW_\gs\le Cn_i^{\ga}t_i$, and since $N_i\asymp n_{i+1}^{2d}$,
	\begin{align*}
		t_i\le C\big(\log n_{i+1}\big)^{1/(\theta+s)}n_{i+1}^{-2d/(\theta+s)}
		\le Cn_{i+1}^{-2d/\theta+\eps'},
	\end{align*}
	uniformly in $i$ for all large $n$: the first factor is at most $n_{i+1}^{\eps'/2}$ and $-2d/(\theta+s)\le-2d/\theta+\eps'/2$ by the choice of $s$.
\end{proof}

Set
\begin{align}\label{eq:gL-fk-q-def}
	\Upsilon_{\phi,k}^{\rm br}:=\sum_{i=0}^{k-1}2^i\gl_i,
	\text{ so that }
	S_n^{\br}\le\Upsilon_{\phi,k}^{\rm br} \text{ on }\cA_k\cap\cB_k.
\end{align}
The sum in~\eqref{eq:gL-fk-q-def} contains $2^i$ bridge terms at level $i$. While deeper bridges are shorter, their drastically shrinking candidate pools can cause the cost coefficients $\gl_i$ to strictly increase with $i$, as is the case in Regime III. The regime-specific proofs balance these competing effects directly via the chosen scale map and terminal depth.

\subsubsection*{The local contribution}

For $i\ge0$ write $\gs_i^-:=0^i$ and $\gs_i^+:=1^i$ for the all-zero and all-one words in $\Sigma_i$, and put $\Sigma_k^{\rm int}:=\Sigma_k\setminus\{\gs_k^-,\gs_k^+\}$. On $\cA_k$ the two extreme terminal segments are the only ones containing $\mvzero$ or $\mvm$. Every interior terminal segment has both endpoints in $\cV^a$.

Fix a desired local tail exponent $q_0>0$. In Proposition~\ref{prop:lin-ub-input}\textup{(i)} choose the corresponding fixed-scale construction. The terminal $\ell_1$ lengths lie between $n_k/2$ and $dn_k$. Set, once and for all,
\begin{align}\label{eq:local-eps}
	c_{\rm seg}:=\frac38,\qquad C_{\rm seg}:=2d,\qquad
	\eps_1:=\frac{c_{\rm seg}}{8C_{\rm seg}}.
\end{align}
The slightly wider constants allow us to replace a prescribed endpoint by a nearby favorable anchor.

For $\mvu\in\{\mvzero,\mvm\}$ let $\mva(\mvu)$ be a nearest vertex of $\cV^a$ to $\mvu$ in $\ell_\infty$, with ties broken deterministically, and write
\begin{align}\label{eq:boundary-anchors}
	\mva_0:=\mva(\mvzero),\qquad \mva_1:=\mva(\mvm),\qquad
	R_\partial:=\norm{\mva_0}_\infty\vee\norm{\mva_1-\mvm}_\infty.
\end{align}
The anchors are $\cF_V$-measurable and finite almost surely. Indeed, for constants $c,C>0$,
\begin{align}\label{eq:anchor-tail}
	\sup_{\mvu\in\dZ^d}\pr\big(\norm{\mva(\mvu)-\mvu}_\infty>r\big)
	\le C\exp\{-cr^d\},\qquad r\ge1.
\end{align}
Using anchors is important: it lets every fixed-scale segment have favorable endpoints and avoids conditioning an unconditional endpoint estimate on the bridge history.

The terminal segments are crossed by the fixed-scale construction, whose efficiency depends on the vertex environment in a corridor around the segment. Since the bridge endpoints are selected recursively, we build the corridor requirement into the last bridge selection: only candidate endpoints with a favorable corridor are admitted.

\begin{defn}[Corridor-good pairs]\label{def:corridor}
	Let $\Lambda_0<\infty$ be the constant of Proposition~\ref{prop:lin-ub-input}\textup{(i)} for $(\eps_1,q_0)$. For $\mvu,\mvv\in\dZ^d$ with $m:=\norm{\mvu-\mvv}$, let $\cC(\mvu,\mvv)$ be the $\cF_V$-measurable event in~\eqref{eq:corridor-event}. We call the pair $(\mvu,\mvv)$ \emph{corridor-good} if $\cC(\mvu,\mvv)$ occurs.
\end{defn}

At the last level $i=k-1$, for $\gs\in\Sigma_{k-1}$ define the favorable reference endpoints
\begin{align*}
	\mvr_{\gs0}:=
	\begin{cases}\mva_0,&\gs=\gs_{k-1}^-,\\ \mvu_{\gs0},&\text{otherwise},\end{cases}
	\qquad
	\mvr_{\gs1}:=
	\begin{cases}\mva_1,&\gs=\gs_{k-1}^+,\\ \mvu_{\gs1},&\text{otherwise}.\end{cases}
\end{align*}
We replace the candidate shells by their corridor-good sub-shells
\begin{align}\label{eq:corridor-shell}
	\wt B_{\gs0}^{a,\rm gd}
	:=\big\{\mvu\in\wt B_{\gs0}^{a}:(\mvr_{\gs0},\mvu)\text{ is corridor-good}\big\},
	\qquad
	\wt B_{\gs1}^{a,\rm gd}
	:=\big\{\mvv\in\wt B_{\gs1}^{a}:(\mvv,\mvr_{\gs1})\text{ is corridor-good}\big\},
\end{align}
and take the minimizing edge in~\eqref{eq:bridge-min} over $\wt B_{\gs0}^{a,\rm gd}\times\wt B_{\gs1}^{a,\rm gd}$. All these sets are $\cF_V$-measurable once the parent centers are known.
For terminal words $\tau\in\Sigma_k$, use the notation
\begin{align*}
	\mvr_{\tau0}:=
	\begin{cases}\mva_0,&\tau=\gs_k^-,\\ \mvu_{\tau0},&\text{otherwise},\end{cases}
	\qquad
	\mvr_{\tau1}:=
	\begin{cases}\mva_1,&\tau=\gs_k^+,\\ \mvu_{\tau1},&\text{otherwise}.\end{cases}
\end{align*}
Accordingly, enlarge the good event to
\begin{align}\label{eq:cA-gd}
	\cA_k^{\rm gd}
	:=\cA_k\cap\bigcap_{\gs\in\Sigma_{k-1}}\Big\{
	\big|\wt B_{\gs0}^{a,\rm gd}\big|\ge\tfrac12\big|\wt B_{\gs0}^{a}\big|,\
	\big|\wt B_{\gs1}^{a,\rm gd}\big|\ge\tfrac12\big|\wt B_{\gs1}^{a}\big|
	\Big\}.
\end{align}
On $\cA_k^{\rm gd}$ the candidate edge sets at the last level still contain at least $\tfrac14c_{\rm sh}n_k^{2d}$ edges. Lemma~\ref{lem:edge-disjoint} therefore applies unchanged, and Lemma~\ref{lem:bridge-trunc} applies with $\cA_k$ replaced by $\cA_k^{\rm gd}$ after decreasing $c_{\rm sh}$ by a factor $4$. By construction, every interior terminal segment is corridor-good. When $R_\partial\le n_k/8$, the two extreme segments are corridor-good after replacing their prescribed endpoints by $\mva_0$ and $\mva_1$. From now on $\cB_k$ denotes the truncation event for this modified construction, so that~\eqref{eq:bridge-trunc-prob} reads $\pr(\cB_k^{\rm c}\cap\cA_k^{\rm gd})\le C\sum_{i<k}2^in_{i+1}^{-2dR}$. The uniform one-point marginal used in Lemma~\ref{lem:Ak-prob} is unaffected at levels $i\le k-2$, which is all that Lemma~\ref{lem:Agd-prob} below requires.

\begin{lem}[Probability of the corridor-good event]\label{lem:Agd-prob}
	For every $q>0$ there is $C_q<\infty$ such that, with $C$ the constant of Lemma~\ref{lem:Ak-prob},
	\begin{align}\label{eq:Agd-prob}
		\pr\big((\cA_k^{\rm gd})^{\rm c}\big)
		\le\pr(\cA_k^{\rm c})+C_q C^k n_k^{-q}
	\end{align}
	for all sufficiently large $n$.
\end{lem}

\begin{proof}
	Fix $\gs\in\Sigma_{k-1}$ and consider the first constraint in~\eqref{eq:cA-gd}. The second is symmetric. Suppose first that the reference center $\mvw:=\mvr_{\gs0}$ is deterministic and favorable, and lies within $n_k/8$ of the center of $\wt B_{\gs0}$. For every favorable candidate $\mvu$ the distance $\norm{\mvw-\mvu}$ is between constant multiples of $n_k$. Hence~\eqref{eq:lin-int-a} gives, for every $q'>0$,
	\begin{align*}
		\E\Big[\#\big\{\mvu\in\wt B_{\gs0}^a:(\mvw,\mvu)\text{ not corridor-good}\big\}\Big]
		\le\big|\wt B_{\gs0}\big|\cdot C_{q'} n_k^{-q'}
		\le C n_k^{d-q'}.
	\end{align*}
	On $\cA_k$ we have $\bigl|\wt B_{\gs0}^a\bigr|\ge c_{\rm sh}n_k^d$, so by Markov's inequality
	\begin{align*}
		\pr\Big(\big|\wt B_{\gs0}^{a,\rm gd}\big|<\tfrac12\big|\wt B_{\gs0}^{a}\big|, \cA_k\Big)
		\le\frac{C n_k^{d-q'}}{\tfrac12c_{\rm sh}n_k^{d}}
		=C'n_k^{-q'} ,
	\end{align*}
	which is the required bound with $q'=q$ after renaming constants. We bound the expected \emph{number} of bad candidates rather than the probability that some candidate is bad. No union bound over candidate positions is needed.

	For a non-extreme random reference center, iterate along its chain of parents exactly as in the proof of Lemma~\ref{lem:Ak-prob}. Each step costs a factor at most $C_0$, so the preceding deterministic-center estimate becomes $C_0^{ k}C'n_k^{-q'}$.

	It remains to justify the anchor references on the two extreme sides. By~\eqref{eq:anchor-tail}, the probability that an anchor lies farther than $n_k/8$ from its prescribed endpoint is at most $Ce^{-cn_k^d}$. On the complementary event, sum the deterministic-center estimate over the $O(n_k^d)$ possible anchor positions. Choosing $q'=q+d+1$ shows that either extreme constraint fails with probability at most $Cn_k^{-q}$. Finally, summing over the $2^{k-1}$ words and both sides gives~\eqref{eq:Agd-prob}, after enlarging $C$ if necessary.
\end{proof}

\begin{lem}[Local contribution]\label{lem:Sloc-control}
	Assume $n_k\to\infty$ and use the corridor construction for the fixed exponent $q_0>0$. There are constants $c_1,C_1<\infty$ and a function $\rho(r)\downarrow0$ such that, with
	\begin{align*}
		Z_\partial:=\ind_{\{\mva_0\ne\mvzero\}}W_{\mvzero\mva_0}
		+\ind_{\{\mva_1\ne\mvm\}}W_{\mva_1\mvm},
	\end{align*}
	there is an event $\cL_k$ on which
	\begin{align}\label{eq:Sloc-scale-short}
		S_n^{\loc}\le2^kc_1n_k \text{ on }\cL_k,
	\end{align}
	and, for all large $n$,
	\begin{align}\label{eq:Lk-prob}
		\pr\big(\cL_k^{\rm c}\cap\cA_k^{\rm gd}\big)
		\le C_1 2^kn_k^{-q_0}+\rho(n_k).
	\end{align}
\end{lem}

\begin{proof}
	On $\{R_\partial\le n_k/8\}$, every reference pair created at the last level has $\ell_1$ distance between $c_{\rm seg}n_k$ and $C_{\rm seg}n_k$. Let $\cH$ be generated by $\cF_V$, all tie-breaking marks, and all bridge noises revealed by the binary construction. Lemma~\ref{lem:edge-disjoint} shows that every revealed bridge edge has length at least $\tfrac38n_{k-1}\ge6n_k$, whereas the fixed-scale construction uses only edges of length at most $\eps_1C_{\rm seg}n_k< n_k$. Thus the relevant short-edge noises are fresh conditionally on $\cH$. The conditional estimate~\eqref{eq:lin-int}, with the independent tie-breaking field adjoined, gives for each of the $2^k$ reference pairs
	\begin{align*}
		\pr\Big(T(\mvr_{\gs0},\mvr_{\gs1})>c_{\rm int}C_{\rm seg}n_k \,\Big|\, \cH\Big)
		\le C n_k^{-q_0}.
	\end{align*}
	Here, for the two extreme words, $\mvr_{\gs0}$ or $\mvr_{\gs1}$ denotes the corresponding anchor. For all other words it denotes the original terminal endpoint. A union bound therefore costs at most $C2^kn_k^{-q_0}$.

	The two anchor radii have the tail~\eqref{eq:anchor-tail}. Moreover $Z_\partial$ is tight uniformly in $n$: each of its two summands has, by translation invariance, the same a.s. finite marginal distribution as the cost of joining a fixed vertex to its nearest favorable anchor. Hence there is a deterministic function $\rho(r)\downarrow0$ such that
	\begin{align*}
		\pr(R_\partial>r/8)+\pr(Z_\partial>r)\le\rho(r).
	\end{align*}
	Let $\cL_k$ be the intersection of $\{R_\partial\le n_k/8\}$, $\{Z_\partial\le n_k\}$, and the events that all reference-pair passage times are at most $c_{\rm int}C_{\rm seg}n_k$. This proves~\eqref{eq:Lk-prob}. Finally, the triangle inequality for $T$ gives
	\begin{align*}
		S_n^{\loc}
		\le Z_\partial+\sum_{\gs\in\Sigma_k}T(\mvr_{\gs0},\mvr_{\gs1})
		\le \big(1+c_{\rm int}C_{\rm seg}\big)2^kn_k
	\end{align*}
	on $\cL_k$, proving~\eqref{eq:Sloc-scale-short} with $c_1:=1+c_{\rm int}C_{\rm seg}$.
\end{proof}

\begin{prop}[Multi-scale ansatz master estimate]\label{prop:Ansatz-bound}
	Let $\cA_k^{\rm gd}$, $\cB_k$, $\cL_k$ be as in~\eqref{eq:cA-gd}, Lemma~\ref{lem:bridge-trunc} and Lemma~\ref{lem:Sloc-control}. Then for all large $n$, and for every $q>0$,
	\begin{align}\label{eq:Tn-ub-tail}
		\pr\Big(T_n>\Upsilon_{\phi,k}^{\rm br}+2^kc_1n_k\Big)
		\le\pr(\cA_k^{\rm c})+C_qC^kn_k^{-q}
		+\pr\big(\cB_k^{\rm c}\cap\cA_k^{\rm gd}\big)+\pr\big(\cL_k^{\rm c}\cap\cA_k^{\rm gd}\big),
	\end{align}
	where $c_1$ is the constant of Lemma~\ref{lem:Sloc-control}, $C$ that of Lemma~\ref{lem:Ak-prob} and $C_q$ that of Lemma~\ref{lem:Agd-prob}.
\end{prop}

\begin{proof}
	On $\cA_k^{\rm gd}\cap\cB_k\cap\cL_k$ the decomposition~\eqref{eq:Tn-Sn}, the bridge bound~\eqref{eq:gL-fk-q-def} and the local bound~\eqref{eq:Sloc-scale-short} give $T_n\le\Upsilon_{\phi,k}^{\rm br}+2^kc_1n_k$. Taking complements and using~\eqref{eq:Agd-prob} to bound $\pr((\cA_k^{\rm gd})^{\rm c})$ yields~\eqref{eq:Tn-ub-tail}.
\end{proof}

\begin{figure}[htbp]
	\centering
	\resizebox{0.80\linewidth}{!}{%
		\begin{tikzpicture}[
				x=1.20cm,y=1.20cm,
				line cap=round,
				line join=round,
				font=\small,
				binarybridge/.style={
						draw=blue!80!black,
						line width=1.65pt
					},
				binarylocal/.style={
						draw=red!78!black,
						line width=1.05pt
					},
				binaryvertex/.style={
						circle,
						fill=black,
						inner sep=1.45pt
					},
				binarylabel/.style={
						fill=white,
						fill opacity=0.95,
						text opacity=1,
						inner sep=0.7pt
					},
				binaryshelllabel/.style={
						font=\scriptsize,
						text=black!68,
						inner sep=1pt
					},
				binarybridgelabel/.style={
						font=\small,
						text=blue!80!black,
						fill=white,
						fill opacity=0.95,
						text opacity=1,
						inner sep=0.8pt
					}
			]

			\colorlet{binarylevelone}{gray!10}
			\colorlet{binaryleveltwo}{blue!10}

			\def\RrootOuter{2.40}
			\def\RrootInner{1.20}
			\def\RchildOuter{0.50}
			\def\RchildInner{0.25}

			\coordinate (u000) at (0.00,0.00);
			\coordinate (u001) at (0.35,0.08);
			\coordinate (u010) at (1.46,0.73);
			\coordinate (u011) at (1.84,0.81);

			\coordinate (u100) at (4.67,0.00);
			\coordinate (u101) at (5.02,0.08);
			\coordinate (u110) at (6.13,0.73);
			\coordinate (u111) at (6.51,0.81);

			\path[
				draw=black!55,
				line width=0.42pt,
				fill=binarylevelone,
				fill opacity=0.95,
				even odd rule
			]
			($(u000)+(-\RrootOuter,-\RrootOuter)$)
			rectangle
			($(u000)+(\RrootOuter,\RrootOuter)$)
			($(u000)+(-\RrootInner,-\RrootInner)$)
			rectangle
			($(u000)+(\RrootInner,\RrootInner)$);

			\path[
				draw=black!55,
				line width=0.42pt,
				fill=binarylevelone,
				fill opacity=0.95,
				even odd rule
			]
			($(u111)+(-\RrootOuter,-\RrootOuter)$)
			rectangle
			($(u111)+(\RrootOuter,\RrootOuter)$)
			($(u111)+(-\RrootInner,-\RrootInner)$)
			rectangle
			($(u111)+(\RrootInner,\RrootInner)$);

			\path[
				draw=black!55,
				line width=0.42pt,
				fill=binaryleveltwo,
				fill opacity=0.95,
				even odd rule
			]
			($(u000)+(-\RchildOuter,-\RchildOuter)$)
			rectangle
			($(u000)+(\RchildOuter,\RchildOuter)$)
			($(u000)+(-\RchildInner,-\RchildInner)$)
			rectangle
			($(u000)+(\RchildInner,\RchildInner)$);

			\path[
				draw=black!55,
				line width=0.42pt,
				fill=binaryleveltwo,
				fill opacity=0.95,
				even odd rule
			]
			($(u011)+(-\RchildOuter,-\RchildOuter)$)
			rectangle
			($(u011)+(\RchildOuter,\RchildOuter)$)
			($(u011)+(-\RchildInner,-\RchildInner)$)
			rectangle
			($(u011)+(\RchildInner,\RchildInner)$);

			\path[
				draw=black!55,
				line width=0.42pt,
				fill=binaryleveltwo,
				fill opacity=0.95,
				even odd rule
			]
			($(u100)+(-\RchildOuter,-\RchildOuter)$)
			rectangle
			($(u100)+(\RchildOuter,\RchildOuter)$)
			($(u100)+(-\RchildInner,-\RchildInner)$)
			rectangle
			($(u100)+(\RchildInner,\RchildInner)$);

			\path[
				draw=black!55,
				line width=0.42pt,
				fill=binaryleveltwo,
				fill opacity=0.95,
				even odd rule
			]
			($(u111)+(-\RchildOuter,-\RchildOuter)$)
			rectangle
			($(u111)+(\RchildOuter,\RchildOuter)$)
			($(u111)+(-\RchildInner,-\RchildInner)$)
			rectangle
			($(u111)+(\RchildInner,\RchildInner)$);

			\node[binaryshelllabel,anchor=north west]
			at ($(u000)+(-2.12,2.08)$) {$\wt B_0$};

			\node[binaryshelllabel,anchor=north west]
			at ($(u111)+(-2.12,2.08)$) {$\wt B_1$};


			\draw[binarybridge]
			(u001)
			.. controls ($(u001)!0.33!(u010)+(0,0.62)$)
			and ($(u001)!0.67!(u010)+(0,0.62)$)
			.. (u010);

			\draw[binarybridge]
			(u011)
			.. controls ($(u011)!0.33!(u100)+(0,1.35)$)
			and ($(u011)!0.67!(u100)+(0,1.35)$)
			.. (u100);

			\draw[binarybridge]
			(u101)
			.. controls ($(u101)!0.33!(u110)+(0,0.62)$)
			and ($(u101)!0.67!(u110)+(0,0.62)$)
			.. (u110);

			\node[binarybridgelabel]
			at ($(u001)!0.52!(u010)+(0,0.92)$) {$\cW_0$};

			\node[binarybridgelabel]
			at ($(u011)!0.50!(u100)+(0,1.58)$) {$\cW_{\varnothing}$};

			\node[binarybridgelabel]
			at ($(u101)!0.48!(u110)+(0,0.92)$) {$\cW_1$};


			\draw[binarylocal]
			(u000)
			.. controls ($(u000)!0.18!(u001)+(0,0.05)$)
			and ($(u000)!0.34!(u001)+(0,-0.05)$)
			.. ($(u000)!0.52!(u001)$)
			.. controls ($(u000)!0.68!(u001)+(0,0.05)$)
			and ($(u000)!0.84!(u001)+(0,-0.04)$)
			.. (u001);

			\draw[binarylocal]
			(u010)
			.. controls ($(u010)!0.18!(u011)+(0,0.05)$)
			and ($(u010)!0.34!(u011)+(0,-0.05)$)
			.. ($(u010)!0.52!(u011)$)
			.. controls ($(u010)!0.68!(u011)+(0,0.05)$)
			and ($(u010)!0.84!(u011)+(0,-0.04)$)
			.. (u011);

			\draw[binarylocal]
			(u100)
			.. controls ($(u100)!0.18!(u101)+(0,0.05)$)
			and ($(u100)!0.34!(u101)+(0,-0.05)$)
			.. ($(u100)!0.52!(u101)$)
			.. controls ($(u100)!0.68!(u101)+(0,0.05)$)
			and ($(u100)!0.84!(u101)+(0,-0.04)$)
			.. (u101);

			\draw[binarylocal]
			(u110)
			.. controls ($(u110)!0.18!(u111)+(0,0.05)$)
			and ($(u110)!0.34!(u111)+(0,-0.05)$)
			.. ($(u110)!0.52!(u111)$)
			.. controls ($(u110)!0.68!(u111)+(0,0.05)$)
			and ($(u110)!0.84!(u111)+(0,-0.04)$)
			.. (u111);

			\node[binaryvertex] at (u000) {};
			\node[binaryvertex] at (u001) {};
			\node[binaryvertex] at (u010) {};
			\node[binaryvertex] at (u011) {};
			\node[binaryvertex] at (u100) {};
			\node[binaryvertex] at (u101) {};
			\node[binaryvertex] at (u110) {};
			\node[binaryvertex] at (u111) {};


			\node[binarylabel,anchor=north east]
			at ($(u000)+(-0.02,-0.15)$) {$\mvu_{000}$};

			\node[binarylabel,anchor=south east]
			at ($(u001)+(-0.03,0.14)$) {$\mvu_{001}$};

			\node[binarylabel,anchor=north east]
			at ($(u010)+(-0.04,-0.11)$) {$\mvu_{010}$};

			\node[binarylabel,anchor=south west]
			at ($(u011)+(0.06,0.14)$) {$\mvu_{011}$};

			\node[binarylabel,anchor=north east]
			at ($(u100)+(-0.03,-0.15)$) {$\mvu_{100}$};

			\node[binarylabel,anchor=north west]
			at ($(u101)+(0.03,-0.15)$) {$\mvu_{101}$};

			\node[binarylabel,anchor=north west]
			at ($(u110)+(0.04,-0.11)$) {$\mvu_{110}$};

			\node[binarylabel,anchor=south west]
			at ($(u111)+(0.06,0.14)$) {$\mvu_{111}$};

		\end{tikzpicture}%
	}%
	\caption{Two refinements of the binary edge-bridge ansatz ($k=2$), drawn schematically and not to scale. The two gray annuli are the root shells and the four blue annuli are the child shells. In each annulus, the white square is the removed inner core. The blue curves represent the three bridge edges $\cW_0$, $\cW_{\varnothing}$, and $\cW_1$, while the red curves represent the four terminal local first-passage pieces contributing to $S_n^{\mathrm{loc}}$. The endpoints are $\mvu_{000}=\mvzero$ and $\mvu_{111}=\lceil n\mvx\rceil$.}
	\label{fig:binary-multiscale-ansatz}
\end{figure}
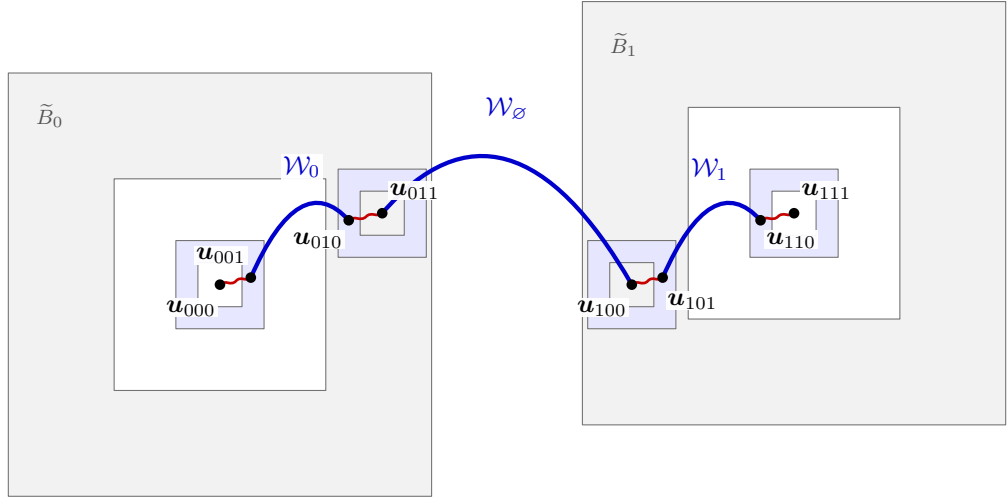

Proposition~\ref{prop:Ansatz-bound} isolates all of the geometric content. The regime-specific proofs of Sections~\ref{sec:proof-III} and~\ref{sec:proof-IVE} consist of choosing $f_\phi$, $k(n)$ and $\eps'$ so that the two deterministic terms have the required order while the three bad probabilities vanish.

\section{\texorpdfstring{Proofs for Regime $\rm I_E$}{Proofs for Regime I-E}}\label{sec:proof-IE}

The mechanism is a second-moment argument run at infinitely many well-separated scales. At each scale we produce a large family of $k$-hop paths from $\mvzero$ to $\mvx$ whose intermediate vertices are constrained to lie in $k-1$ concentric annuli. The number of such paths grows like $\ell_j^{d(k-1)}$, while the probability that a given one is cheap decays like $\ell_j^{-k\ga(\theta+\eps)}$. The condition $\ga<d/\theta$ is exactly what makes the first factor win, and the separation between scales makes the resulting successes independent.

We first record the combinatorial input.

\begin{lem}[Overlap counting]\label{lem:overlap}
	Let $k\ge2$, let $A_1,\dots,A_{k-1}\subset\dZ^d$ be pairwise disjoint sets not containing $\mvzero$ or $\mvx$, and let
	\begin{align*}
		\cP:=\big\{(\mvz_0,\dots,\mvz_k):\mvz_0=\mvzero, \mvz_k=\mvx, \mvz_i\in A_i (1\le i\le k-1)\big\}.
	\end{align*}
	For $\mvpi,\mvpi'\in\cP$ let $m(\mvpi,\mvpi')$ be the number of edges common to both. Then $m(\mvpi,\mvpi')=k$ forces $\mvpi'=\mvpi$, and for every $\mvpi\in\cP$ and every $1\le m\le k-1$,
	\begin{align*}
		\#\big\{\mvpi'\in\cP:m(\mvpi,\mvpi')=m\big\}
		\le\binom{k}{m} \max_{1\le i\le k-1}\abs{A_i}^{ k-1-m}.
	\end{align*}
\end{lem}

\begin{proof}
	Write $\mvpi=(\mvz_0,\dots,\mvz_k)$ and $\mvpi'=(\mvz_0',\dots,\mvz_k')$. Suppose the edge $\la\mvz_{r-1}'\mvz_r'\ra$ of $\mvpi'$ equals the edge $\la\mvz_{s-1}\mvz_s\ra$ of $\mvpi$ as an unordered pair. Each vertex of either path lies in exactly one of the $k+1$ pairwise disjoint sets $\{\mvzero\},A_1,\dots,A_{k-1},\{\mvx\}$, and its position along the path is determined by which of these it lies in. Comparing the two endpoints therefore forces $r=s$: shared edges occupy matching positions. Let $M\subseteq[k]$ be the set of shared positions, $\abs{M}=m$. If $\mvz_{i-1}'=\mvz_{i-1}$ and $\mvz_i'=\mvz_i$ for every $i\in M$, then the indices in
	$P(M):=\bigcup_{i\in M}\{i-1,i\}$
	are pinned, and the remaining internal coordinates of $\mvpi'$ are free, each ranging over at most $\max_i\abs{A_i}$ values. The number of free internal coordinates is $(k-1)-\abs{P(M)\cap[1,k-1]}$, so it suffices to show
	\begin{align}\label{eq:pinned-count}
		\abs{P(M)\cap[1,k-1]}\ge m \text{ whenever }1\le m\le k-1.
	\end{align}
	Since the map $i\mapsto\{i-1,i\}$ has the property that $\abs{P(M)}\ge m+1$ for any nonempty $M$, with equality if and only if $M$ is an interval, we distinguish three cases. If $M$ contains neither $1$ nor $k$, then $P(M)\subseteq[1,k-1]$ and $\abs{P(M)\cap[1,k-1]}=\abs{P(M)}\ge m+1$. If $M$ contains exactly one of $1,k$, then exactly one element of $P(M)$ (namely $0$ or $k$) is discarded, so $\abs{P(M)\cap[1,k-1]}\ge m+1-1=m$. If $M$ contains both $1$ and $k$ and $m\le k-1$, then $M$ is not an interval, so $\abs{P(M)}\ge m+2$ and discarding $0$ and $k$ leaves at least $m$. This proves~\eqref{eq:pinned-count}. Finally, if $m=k$ then $M=[k]$ pins every coordinate and $\mvpi'=\mvpi$. Summing over the $\binom km$ choices of $M$ gives the bound.
\end{proof}

\begin{proof}[Proof of Theorem~\ref{thm:IE}]
	We may assume $\mvx\ne\mvzero$. Set $\ell_0:=\norm{\mvx}$. Since $\ga\theta<d$, we may fix an integer $k\ge2$ and $\eps\in(0,\theta)$ such that
	\begin{align}\label{eq:IE-k-eps}
		k\big(d-\ga(\theta+\eps)\big)>d
		\text{ and }
		\ga\eps\le\frac{d-\ga\theta}{4k-1}.
	\end{align}
	Indeed, first choose $k>d/(d-\ga\theta)$, which makes $k(d-\ga\theta)>d$, and then take $\eps$ small enough for both displays. Fix also $0<a_0<A_0<\infty$ with $p_0:=\pr(a_0\le V\le A_0)>0$, which is possible since $\pr(V>0)=1$.

	\medskip
	\noindent\textbf{Step 1: a family of paths at each scale.}
	Fix $M_{\rm sc}>4k+3$ and set $\ell_j:=M_{\rm sc}^{ j}\ell_0$ for $j\ge1$. Let $\mvz$ be a lattice point nearest to $\mvx/2$, so $\norm{\mvzero-\mvz}\vee\norm{\mvx-\mvz}\le\ell_0$. For $j\ge1$ and $1\le i\le k-1$ define the \emph{separated} annuli
	\begin{align}\label{eq:IE-annuli}
		A_i^{(j)}:=\big\{\mvy\in\dZ^d:(4i-1)\ell_j\le\norm{\mvy-\mvz}<(4i+1)\ell_j\big\},
	\end{align}
	and
	\begin{align*}
		\cP_j:=\big\{(\mvz_0,\dots,\mvz_k):\mvz_0=\mvzero, \mvz_k=\mvx, \mvz_i\in A_i^{(j)} (1\le i\le k-1)\big\}.
	\end{align*}
	The annuli $A_1^{(j)},\dots,A_{k-1}^{(j)}$ are pairwise disjoint, being separated by gaps of width $2\ell_j$, and they avoid $\mvzero$ and $\mvx$ since $\norm{\mvzero-\mvz}\vee\norm{\mvx-\mvz}\le\ell_0\le\ell_j<3\ell_j$. Moreover $\bigl|A_i^{(j)}\bigr|\asymp\ell_j^{ d}$, so
	\begin{align}\label{eq:IE-path-count}
		\bigl|\cP_j\bigr|
		=\prod_{i=1}^{k-1}\bigl|A_i^{(j)}\bigr|
		\asymp\ell_j^{ d(k-1)} ,
	\end{align}
	with constants depending only on $d$ and $k$. The separation of the annuli also gives two-sided control of the edge lengths: for every $\mvpi=(\mvz_0,\dots,\mvz_k)\in\cP_j$ and $1\le i\le k$,
	\begin{align}\label{eq:IE-edge-lengths}
		2\ell_j\le\norm{\mvz_i-\mvz_{i-1}}\le(8k+6)\ell_j.
	\end{align}
	For $2\le i\le k-1$ this is because $\norm{\mvz_i-\mvz}\ge(4i-1)\ell_j$ while $\norm{\mvz_{i-1}-\mvz}<(4i-3)\ell_j$. For $i=1$ and $i=k$ it follows from $\norm{\mvzero-\mvz}\vee\norm{\mvx-\mvz}\le\ell_0\le\ell_j$ together with $k\ge2$. Finally, since $M_{\rm sc}>4k+3$, every edge used at scale $j$ has length at most $(8k+6)\ell_j<2\ell_{j+1}$, whereas every edge used at scale $j+1$ has length at least $2\ell_{j+1}$. Hence
	\begin{align}\label{eq:IE-scale-disjoint}
		\text{the edge sets used by }\cP_j\text{ and }\cP_{j'}\text{ are disjoint for }j\ne j'.
	\end{align}

	\medskip
	\noindent\textbf{Step 2: good vertices and good paths.}
	For $j\ge1$ and $1\le i\le k-1$ set
	$A_{i,[a_0,A_0]}^{(j)}:=\{\mvy\in A_i^{(j)}:a_0\le V_{\mvy}\le A_0\}$.
	Since $\bigl|A_{i,[a_0,A_0]}^{(j)}\bigr|\sim\mathrm{Bin}\bigl(\bigl|A_i^{(j)}\bigr|,p_0\bigr)$ with $\bigl|A_i^{(j)}\bigr|\asymp\ell_j^d$, a Chernoff bound gives
	$\pr\bigl(\bigl|A_{i,[a_0,A_0]}^{(j)}\bigr|<\tfrac{p_0}{2}\bigl|A_i^{(j)}\bigr|\bigr)\le\exp(-c\ell_j^{ d})$,
	which is summable in $j$ because $\ell_j$ grows geometrically. By Borel--Cantelli there is an event $\gO_0$ of full probability and an a.s. finite random $J_0$ such that on $\gO_0$
	\begin{align}\label{eq:IE-good-vtx}
		\bigl|A_{i,[a_0,A_0]}^{(j)}\bigr|\ge\tfrac{p_0}{2}\bigl|A_i^{(j)}\bigr|\asymp\ell_j^{ d}
		\text{ for all }j\ge J_0, 1\le i\le k-1.
	\end{align}
	Note that $\gO_0$ and $J_0$ are $\cF_V$-measurable. Define the good path family
	$\cP_j^{a_0,A_0}:=\{\mvpi\in\cP_j:\mvz_i\in A_{i,[a_0,A_0]}^{(j)} \forall 1\le i\le k-1\}$,
	so that on $\gO_0$, for all $j\ge J_0$,
	\begin{align}\label{eq:IE-good-paths-count}
		\bigl|\cP_j^{a_0,A_0}\bigr|
		=\prod_{i=1}^{k-1}\bigl|A_{i,[a_0,A_0]}^{(j)}\bigr|
		\ge c_1\ell_j^{ d(k-1)}.
	\end{align}

	\medskip
	\noindent\textbf{Step 3: first and second moments.}
	For $t>0$ and $j\ge1$ put
	\begin{align*}
		N_j(t):=\sum_{\mvpi\in\cP_j^{a_0,A_0}}\ind\{W_{\mvpi}\le t\},
		\qquad
		T_j:=\inf\big\{W_{\mvpi}:\mvpi\in\cP_j^{a_0,A_0}\big\}.
	\end{align*}
	Fix $\mvpi=(\mvz_0,\dots,\mvz_k)\in\cP_j^{a_0,A_0}$. Conditionally on $\cF_V$ we may write
	\begin{align*}
		W_{\mvpi}=\sum_{i=1}^k\frac{\go_i}{\gl_i},
		\qquad
		\gl_i:=\frac{V_{\mvz_{i-1}}V_{\mvz_i}}{\norm{\mvz_i-\mvz_{i-1}}^{\ga}},
	\end{align*}
	where the $\go_i$ are i.i.d.~copies of $\go$ independent of $\cF_V$. By~\eqref{eq:IE-edge-lengths} and $a_0\le V\le A_0$ on internal vertices,
	\begin{align}\label{eq:IE-gl-bounds}
		c_2 \big(V_{\mvzero}\wedge1\big)\big(V_{\mvx}\wedge1\big)\ell_j^{-\ga}
		\le\gl_i\le
		C_2 \big(V_{\mvzero}\vee1\big)\big(V_{\mvx}\vee1\big)\ell_j^{-\ga},
		\qquad 1\le i\le k,
	\end{align}
	for constants $c_2,C_2$ depending on $a_0,A_0,k,\ga,d$. Only the two edges incident to $\mvzero$ and $\mvx$ involve the unbounded weights $V_{\mvzero},V_{\mvx}$. In particular $\gl_i\to0$ as $j\to\infty$, so for any fixed $t>0$ the hypothesis $t\gl_i/k\le a(\eps)$ of Lemma~\ref{lem:go-tail-bound} holds for all $j$ large enough, almost surely. Applying the lower bound of Lemma~\ref{lem:go-tail-bound} conditionally on $\cF_V$ and absorbing the $k$-dependent factors into constants, we obtain, for every $t\in(0,a]$ and all large $j$,
	\begin{align}\label{eq:IE-path-lb}
		\pr\big(W_{\mvpi}\le t\mid\cF_V\big)
		\ge c_\eps t^{k(\theta+\eps)}\prod_{i=1}^k\gl_i^{\theta+\eps}
		\ge c_3 \Psi_- t^{k(\theta+\eps)} \ell_j^{-k\ga(\theta+\eps)},
	\end{align}
	where $\Psi_-:=\big((V_{\mvzero}\wedge1)(V_{\mvx}\wedge1)\big)^{k(\theta+\eps)}>0$ is $\cF_V$-measurable and a.s. positive. Combining with~\eqref{eq:IE-good-paths-count}, on $\gO_0$ and for $j\ge J_0$ large,
	\begin{align}\label{eq:IE-ENj}
		\E\big[N_j(t)\mid\cF_V\big]
		\ge c_4 \Psi_- t^{k(\theta+\eps)} \ell_j^{ d(k-1)-k\ga(\theta+\eps)}.
	\end{align}
	By the first inequality in~\eqref{eq:IE-k-eps} the exponent $d(k-1)-k\ga(\theta+\eps)=k(d-\ga(\theta+\eps))-d$ is strictly positive, so
	\begin{align}\label{eq:IE-ENj-infty}
		\E\big[N_j(t)\mid\cF_V\big]\longrightarrow\infty \text{ on }\gO_0, \text{as }j\to\infty.
	\end{align}

	For the second moment write, conditionally on $\cF_V$,
	\begin{align*}
		\E\big[N_j(t)^2\mid\cF_V\big]
		=\sum_{\mvpi,\mvpi'}\pr\big(W_{\mvpi}\le t, W_{\mvpi'}\le t\mid\cF_V\big)
		=\sum_{m=0}^{k}S_m ,
	\end{align*}
	where $S_m$ collects the pairs with $m(\mvpi,\mvpi')=m$. If $m=0$ the two paths use disjoint edge sets, so the corresponding noises are conditionally independent and $S_0\le\E[N_j(t)\mid\cF_V]^2$. If $m=k$ then $\mvpi'=\mvpi$ by Lemma~\ref{lem:overlap}, so $S_k=\E[N_j(t)\mid\cF_V]$.

	Fix $1\le m\le k-1$. Since $\{W_{\mvpi}\le t,W_{\mvpi'}\le t\}\subseteq\{\sum_{e\in E(\mvpi)\cup E(\mvpi')}W_e\le2t\}$ and $\abs{E(\mvpi)\cup E(\mvpi')}=2k-m$, the upper bound of Lemma~\ref{lem:go-tail-bound}, which requires no constraint on $t$, together with the upper bound in~\eqref{eq:IE-gl-bounds}, gives
	\begin{align}\label{eq:IE-pair-ub}
		\pr\big(W_{\mvpi}\le t, W_{\mvpi'}\le t\mid\cF_V\big)
		\le C_5 \Psi_+ t^{(2k-m)(\theta-\eps)} \ell_j^{-(2k-m)\ga(\theta-\eps)},
	\end{align}
	with $\Psi_+:=\big((V_{\mvzero}\vee1)(V_{\mvx}\vee1)\big)^{2k(\theta-\eps)}$, an $\cF_V$-measurable a.s. finite quantity. By Lemma~\ref{lem:overlap} and~\eqref{eq:IE-path-count} the number of ordered pairs with $m(\mvpi,\mvpi')=m$ is at most $C\ell_j^{d(k-1)}\cdot\ell_j^{d(k-1-m)}$. Hence, dividing by the square of~\eqref{eq:IE-ENj},
	\begin{align}\label{eq:IE-ratio}
		\frac{S_m}{\E[N_j(t)\mid\cF_V]^2}
		\le C_6(t,V_{\mvzero},V_{\mvx}) \ell_j^{\gk_m(\eps)},
		\qquad
		\gk_m(\eps):=(4k-m)\ga\eps-m\big(d-\ga\theta\big),
	\end{align}
	Here $C_6(t,V_{\mvzero},V_{\mvx})<\infty$ is $\cF_V$-measurable. It also
	absorbs the finite, $t$-dependent factor
	\begin{align*}
		t^{(2k-m)(\theta-\eps)-2k(\theta+\eps)}.
	\end{align*}
	The exponent of $\ell_j$ in~\eqref{eq:IE-ratio} is
	\begin{align*}
		 & d(k-1)+d(k-1-m)-(2k-m)\ga(\theta-\eps)            \\
		 & \hspace{4em}-\big[2d(k-1)-2k\ga(\theta+\eps)\big]
		=-dm+m\ga\theta+(4k-m)\ga\eps,
	\end{align*}
	which is $\gk_m(\eps)$. By the second inequality in~\eqref{eq:IE-k-eps},
	\begin{align*}
		\gk_m(\eps)\le(4k-1)\ga\eps-m(d-\ga\theta)\le(1-m)(d-\ga\theta)\le0
		\text{ for }1\le m\le k-1 ,
	\end{align*}
	with $\gk_m(\eps)\le-(d-\ga\theta)<0$ as soon as $m\ge2$. Together with $S_0\le\E[N_j\mid\cF_V]^2$, $S_k=\E[N_j\mid\cF_V]$ and~\eqref{eq:IE-ENj-infty}, we conclude that on $\gO_0$ there is an a.s. finite $\cF_V$-measurable $C_7=C_7(t,V_{\mvzero},V_{\mvx})$ and an a.s.~finite random $J_1\ge J_0$ with
	\begin{align*}
		\E\big[N_j(t)^2\mid\cF_V\big]\le C_7 \E\big[N_j(t)\mid\cF_V\big]^2
		\text{ for all }j\ge J_1.
	\end{align*}
	The conditional Paley--Zygmund inequality then gives, on $\gO_0$,
	\begin{align}\label{eq:IE-PNj}
		\pr\big(N_j(t)\ge1\mid\cF_V\big)
		\ge\frac{\E[N_j(t)\mid\cF_V]^2}{\E[N_j(t)^2\mid\cF_V]}
		\ge\frac{1}{C_7}=:A(t,V)>0,
		\qquad j\ge J_1,
	\end{align}
	a bound which is uniform in $j\ge J_1$.

	\medskip
	\noindent\textbf{Step 4: conclusion by independence across scales.}
	By~\eqref{eq:IE-scale-disjoint} the families $\cP_j^{a_0,A_0}$, $j\ge1$, use pairwise disjoint edge sets, so conditionally on $\cF_V$ the variables $\{T_j\}_{j\ge1}$ are independent, and $\{N_j(t)\ge1\}=\{T_j\le t\}$. Since $T(\mvzero,\mvx)\le T_j$ for every $j$, and since $J_1$ is $\cF_V$-measurable, for any fixed $t\in(0,a]$ and every $N\ge1$,
	\begin{align*}
		\pr\big(T(\mvzero,\mvx)>t\big)
		 & \le\E\Big[\ind_{\gO_0} \pr\Big(\bigcap_{j=J_1}^{J_1+N-1}\{T_j>t\} \Big| \cF_V\Big)\Big]      \\
		 & =\E\Big[\ind_{\gO_0}\prod_{j=J_1}^{J_1+N-1}\Big(1-\pr\big(N_j(t)\ge1\mid\cF_V\big)\Big)\Big] \\
		 & \le\E\Big[\ind_{\gO_0}\big(1-A(t,V)\big)^{N}\Big]
		\xrightarrow[N\to\infty]{}0
	\end{align*}
	by dominated convergence, since $A(t,V)>0$ on $\gO_0$ and $\pr(\gO_0)=1$. Hence $\pr(T(\mvzero,\mvx)>t)=0$ for every $t\in(0,a]$, and letting $t\downarrow0$ along a sequence gives $T(\mvzero,\mvx)=0$ almost surely.
\end{proof}

\begin{rem}[One-sided noise assumption in Regime $\rm I_E$]
	\label{rem:IE-one-sided}
	Although the preceding proof uses both bounds of Lemma~\ref{lem:go-tail-bound}, Regime $\rm I_E$ requires only the lower tail of $\go$.
	Suppose $0<V<\infty$ a.s., $\E\go^{\gz}<\infty$, and $\pr(\go\le t)\ge ct^{\theta+\eta}$ for small $t>0$, where $\ga(\theta+\eta)<d$.
	A standard quantile coupling yields $\widetilde\go \ge \go$ with $\pr(\widetilde\go\le t) = \min\{\pr(\go\le t), ct^{\theta+\eta}\}$.
	By construction, $\widetilde\go$ inherits the upper tail of $\go$ (hence $\E\widetilde\go^{\gz}<\infty$) and satisfies the exact polynomial lower-tail.
	The resulting passage metrics satisfy $T \le \widetilde T$.
	Applying the proof to $\widetilde T$ gives $\widetilde T(\mvzero,\mvx)=0$ a.s., which implies $T(\mvzero,\mvx)=0$ a.s.
\end{rem}

\section{\texorpdfstring{Proofs for Regime $\rm I_H$}{Proofs for Regime I-H}}\label{sec:proof-IH}

Here no counting is needed: a \emph{single} sufficiently heavy vertex, found in each dyadic annulus, already makes a two-hop path arbitrarily cheap.

\begin{proof}[Proof of Theorem~\ref{thm:IH}]
	Fix $\mvx\ne\mvzero$, set $n:=\norm{\mvx}\ge1$, and recall the annuli~\eqref{eq:def-Ahub}. Since $\ga<d/\gc$ we may fix
	\begin{align}\label{eq:IH-eps}
		\eps\in\Big(0,\tfrac{d}{\gc}-\ga\Big)
		\text{ and then }
		\eps'\in\Big(0,\frac{\eps\gc^2}{2d}\Big).
	\end{align}
	For $k\ge1$ let $\mvy_k\in A_{2^kn}^{\Hub}$ attain $V_{\mvy_k}=\max_{\mvz\in A_{2^kn}^{\Hub}}V_{\mvz}$, ties broken by a fixed deterministic rule, so that $\mvy_k$ is $\cF_V$-measurable.

	\medskip
	\noindent\textbf{Step 1: a heavy hub in every annulus.}
	Write $m_k:=2^kn$, so $\abs{A_{m_k}^{\Hub}}\ge c m_k^{ d}$. By the lower half of~\eqref{def:V_x} there is $x_0<\infty$ with $\pr(V\ge x)\ge x^{-\gc-\eps'}$ for all $x\ge x_0$. Applying this with $x=m_k^{d/\gc-\eps}$, which exceeds $x_0$ for all large $k$ because $d/\gc-\eps>\ga>0$,
	\begin{align*}
		\pr\big(V\ge m_k^{d/\gc-\eps}\big)
		\ge m_k^{-(d/\gc-\eps)(\gc+\eps')}
		=m_k^{-d+\gd},
		\qquad
		\gd:=\eps\gc+\eps\eps'-{d\eps'}/{\gc}.
	\end{align*}
	The choice of $\eps'$ in~\eqref{eq:IH-eps} gives $d\eps'/\gc<\eps\gc/2$, hence $\gd\ge\eps\gc/2>0$. Consequently, by independence,
	\begin{align*}
		\pr\Big(\max_{\mvz\in A_{m_k}^{\Hub}}V_{\mvz}<m_k^{d/\gc-\eps}\Big)
		=\pr\big(V<m_k^{d/\gc-\eps}\big)^{\abs{A_{m_k}^{\Hub}}}
		\le\big(1-m_k^{-d+\gd}\big)^{c m_k^{d}}
		\le\exp\big(-c m_k^{\gd}\big),
	\end{align*}
	which is summable in $k$ since $m_k=2^kn$ grows geometrically. By Borel--Cantelli there is an a.s. finite random $k_1$ with
	\begin{align}\label{eq:IH-hub-lower}
		V_{\mvy_k}\ge(2^kn)^{d/\gc-\eps} \text{ for all }k\ge k_1.
	\end{align}

	\medskip
	\noindent\textbf{Step 2: the two-hop cost.}
	Since $\mvy_k\in A_{2^kn}^{\Hub}$ we have $2^kn\le\norm{\mvy_k}<2^{k+1}n$, and by the triangle inequality and $\norm{\mvx}=n\le2^{k-1}n$ for $k\ge1$,
	\begin{align*}
		2^{k-1}n\le\norm{\mvy_k}-n\le\norm{\mvx-\mvy_k}\le\norm{\mvy_k}+n\le3\cdot2^{k}n.
	\end{align*}
	Hence there is a deterministic $c_0\ge1$ with $c_0^{-1}2^kn\le\norm{\mvy_k}\wedge\norm{\mvx-\mvy_k}$ and $\norm{\mvy_k}\vee\norm{\mvx-\mvy_k}\le c_02^kn$ for all $k\ge1$. Writing
	\begin{align*}
		T_k:=W_{\mvzero\mvy_k}+W_{\mvy_k\mvx}
		=\frac{\norm{\mvy_k}^{\ga}\go_{\mvzero\mvy_k}}{V_{\mvzero}V_{\mvy_k}}
		+\frac{\norm{\mvx-\mvy_k}^{\ga}\go_{\mvx\mvy_k}}{V_{\mvy_k}V_{\mvx}},
	\end{align*}
	we obtain from~\eqref{eq:IH-hub-lower}, for all $k\ge k_1$,
	\begin{align}\label{eq:IH-Tk}
		T_k\le c_0^{\ga} (2^kn)^{\ga-d/\gc+\eps}
		\Big(\frac{\go_{\mvzero\mvy_k}}{V_{\mvzero}}+\frac{\go_{\mvx\mvy_k}}{V_{\mvx}}\Big).
	\end{align}

	\medskip
	\noindent\textbf{Step 3: conditional Borel--Cantelli.}
	Conditionally on $\cF_V$ the sequence $(\mvy_k)_{k\ge1}$ is deterministic, the annuli $A_{2^kn}^{\Hub}$ are pairwise disjoint so the $\mvy_k$ are distinct, and the noises $\go_{\mvzero\mvy_k},\go_{\mvx\mvy_k}$, $k\ge1$, are i.i.d.~copies of $\go$ independent of $\cF_V$. Fix $\gd_0>0$. By a union bound and Markov's inequality applied to the $\gz$-th moment,
	\begin{align*}
		\pr\big(T_k>\gd_0\mid\cF_V\big)
		 & \le\pr\Big(\go_{\mvzero\mvy_k}>\tfrac{\gd_0V_{\mvzero}}{2c_0^{\ga}}(2^kn)^{d/\gc-\ga-\eps} \Big| \cF_V\Big)
		+\pr\Big(\go_{\mvx\mvy_k}>\tfrac{\gd_0V_{\mvx}}{2c_0^{\ga}}(2^kn)^{d/\gc-\ga-\eps} \Big| \cF_V\Big)                                   \\
		 & \le\E(\go^{\gz})\Big(\frac{2c_0^{\ga}}{\gd_0}\Big)^{\gz}\big(V_{\mvzero}^{-\gz}+V_{\mvx}^{-\gz}\big) (2^kn)^{-\gz(d/\gc-\ga-\eps)}
	\end{align*}
	for $k\ge k_1$. The prefactor is a.s. finite because $\E\go^{\gz}<\infty$ and $V>0$ a.s., and $d/\gc-\ga-\eps>0$ by~\eqref{eq:IH-eps}, so the bound is summable in $k$. Hence $\sum_k\pr(T_k>\gd_0\mid\cF_V)<\infty$ almost surely. By the conditional Borel--Cantelli lemma, $\pr(T_k>\gd_0 \text{i.o.}\mid\cF_V)=0$ a.s., and intersecting over rational $\gd_0>0$ gives $T_k\to0$ almost surely. Since $T(\mvzero,\mvx)\le T_k$ for every $k$, we conclude $T(\mvzero,\mvx)=0$ a.s. Intersecting over the countably many $\mvx\in\dZ^d$ gives a single almost-sure event.
\end{proof}

\begin{proof}[Proof of Theorem~\ref{thm:IH-hop}]
	By~\eqref{def:go}, $\pr(\go\le x)\to0$ as $x\downarrow0$, so $\pr(\go>0)=1$. Also, note that $\pr(V<\infty)=1$. Hence every edge weight is strictly positive almost surely. Since there are only countably many finite self-avoiding paths from $\mvzero$ to $\mvx$ and countably many $\mvx$, almost surely $W_{\mvpi}>0$ for every such path. By Theorem~\ref{thm:IH}, $T(\mvzero,\mvx)=0$. So the infimum in~\eqref{eq:def-T} is not attained by any finite path, that is, $H_{\mvx}(0)=\infty$.

	Now fix $t\in(0,W_{\mvzero\mvx})$. The unique one-edge path has weight $W_{\mvzero\mvx}>t$, so $H_{\mvx}(t)\ge2$. On the other hand Theorem~\ref{thm:IH} produces vertices $\mvy_k$ with $W_{\mvzero\mvy_k}+W_{\mvy_k\mvx}\to0$, so for all large $k$ the two-hop path $\la\mvzero \mvy_k \mvx\ra$ has weight at most $t$. It is self-avoiding since $\mvy_k\notin\{\mvzero,\mvx\}$. As the annuli $A_{2^kn}^{\Hub}$ are disjoint, the $\mvy_k$ are distinct, giving infinitely many such $2$-hop paths. Hence $H_{\mvx}(t)=2$ on $(0,W_{\mvzero\mvx})$. Finally, for $t\ge W_{\mvzero\mvx}$ the single edge $\la\mvzero\mvx\ra$ is admissible, so $H_{\mvx}(t)=1$.
\end{proof}

\section{Proofs for Regime II}\label{sec:proof-II}

Recall the range $d/\gc<\ga<2d/\gc$. The scale map is $f_\phi(t)=t^{\phi}$ with $\phi\in(0,1)$, so that the scales $n_i=n^{\phi^i}$ decrease \emph{doubly exponentially} down to a bounded terminal scale. The chain therefore has $k(n)\asymp\log\log n$ hops on each side, matching the lower bound of Theorem~\ref{thm:II-hop}, and -- crucially -- the costs $\gl_i(n)$ decay doubly exponentially as one moves from either endpoint toward the center of the chain, which is what makes the total cost summable uniformly in $n$.

\begin{proof}[Proof of Theorem~\ref{thm:II-ub}]
	Fix $\eps\in(0,1)$. Since $\ga\gc<2d$, choose once and for all $\gc<\tilde\gc<2d/\ga.$
	Use the normalized hub weights and Lemma~\ref{lem:hub-ub} with this
	fixed $\tilde\gc$. Since $\ga\tilde\gc<2d$ and
	$d(\phi+\phi^2)\uparrow2d$ as $\phi\uparrow1$, we may fix
	$\phi\in(0,1)$ with
	\begin{align}\label{eq:II-phi}
		\ga\tilde\gc<d(\phi+\phi^2),
		\qquad
		s:=\frac{d(\phi+\phi^2)}{\tilde\gc}-\ga>0.
	\end{align}
	Fix $p:=2$, choose $\ell\in\dN$ with $\ell>2/\gz$, and choose $r>1/\tilde\gc$. For a parameter $v_\ast\in(0,\tfrac12)$ to be fixed below, set
	\begin{align}\label{eq:II-v-seq}
		v_j:=v_\ast(j+2)^{-r},\quad j\ge0.
	\end{align}

	\medskip
	\noindent\textbf{Step 0: reduction to large targets.}
	For $\mvu\in\dZ^d$ with $0<\norm{\mvu}\le N$ we may use the single edge $\la\mvzero\mvu\ra$, giving
	$T(\mvzero,\mvu)\le N^{\ga}\go_{\mvzero\mvu}/(V_{\mvzero}V_{\mvu})$,
	whose law does not depend on $\mvu$. Hence for every $N$ there is $C_N(\eps)<\infty$ with
	$\sup_{0<\norm{\mvu}\le N}\pr(T(\mvzero,\mvu)>C_N(\eps))<\eps$. It therefore suffices to prove~\eqref{eq:II-uniform} for $\norm{\mvu}\ge N$ with $N$ as large as we please, and by translation invariance and isotropy of the construction it suffices to bound $T(\mvzero,\mvm)$ for $\mvm=\mvu$, uniformly in the direction. Below, $n:=\norm{\mvu}$.

	\medskip
	\noindent\textbf{Step 1: scales and the hub-good event.}
	Take $n\ge c^{1/\phi}$ and apply the Hub-Chain Ansatz of Section~\ref{ssec:Ansatz_II_IVH} with
	\begin{align*}
		n_i:=f_\phi^{(i)}(n)=n^{\phi^i},
		\qquad
		k=k(n):=\max\{m\ge0:n_{m+1}\ge c\},
	\end{align*}
	where $c\ge1$ is a constant fixed below. Then $c\le n_{k+1}<c^{1/\phi}$ and $k\asymp\log\log n$. We shall choose $c$ large enough that
	\begin{align}\label{eq:II-c-cond}
		c\ge8^{1/(1-\phi)},
		\qquad
		c\ge(16d)^{\phi/(1-\phi)},
		\qquad
		\big|B(\mvzero,c)\big|\ge\ell.
	\end{align}
	The first condition guarantees $n_{i+1}=n_i^{\phi}\le n_i/8$ for all $i\le k$. The second guarantees $n^{1-\phi}\ge16d$ for all $n\ge c^{1/\phi}$, hence $n_1=n^{\phi}\le n/(16d)\le\tfrac1{16}\norm{\mvu}_\infty$ (as $\norm{\mvu}_\infty\ge n/d$), uniformly in the direction of $\mvu$. So the separation hypothesis~\eqref{eq:hub-scale-separation} holds and Lemma~\ref{lem:hub-geom} applies. The third guarantees that the smallest ball $B_{\pm k}$, of radius $n_{k+1}\ge c$, contains at least $\ell$ vertices.

	Define the hub-good event with the \emph{reversed} truncation sequence $\widetilde v_i:=v_{k-i}$:
	\begin{align*}
		\cG_n^{\rm II}
		:=\cG_n^{\Hub}(\vec{\widetilde v})
		=\bigcap_{i=-k}^{k}\Big\{V_i^{(\ell)}\ge v_{k-\abs{i}} n_{\abs{i}+1}^{d/\tilde\gc}\Big\}.
	\end{align*}
	The reversal is essential: the thresholds must be \emph{weakest} where the balls are largest, at the center of the chain, and \emph{strongest} at the ends where the balls have bounded radius. To verify the hypothesis~\eqref{eq:hub-xV-cond} of Lemma~\ref{lem:hub-ub}, write $j:=k-i$ and note $n_{i+1}=n_{k+1}^{\phi^{i-k}}\ge c^{\phi^{-j}}$, so
	\begin{align*}
		\widetilde v_i n_{i+1}^{d/\tilde\gc}
		\ge v_\ast(j+2)^{-r}c^{(d/\tilde\gc)\phi^{-j}}
		=:v_\ast h_c(j).
	\end{align*}
	Once $v_\ast$ is fixed, the bound $\sup_{j\ge0}\phi^j \big(\log(x_V/v_\ast)+r\log(j+2)\big)<\infty$ allows us to enlarge $c$ so that $v_\ast(j+2)^{-r}c^{(d/\tilde\gc)\phi^{-j}}\ge x_V$ for every $j\ge0$, as well as~\eqref{eq:II-c-cond}. Lemma~\ref{lem:hub-ub}~(i) then gives
	\begin{align*}
		\sup_{n\ge c^{1/\phi}}\pr\big((\cG_n^{\rm II})^{\rm c}\big)
		\le2\sum_{j=0}^{\infty}C_1\exp\big(-c_2v_j^{-\tilde\gc}\big)
		=2\sum_{j=0}^{\infty}C_1\exp\big(-c_2v_\ast^{-\tilde\gc}(j+2)^{r\tilde\gc}\big).
	\end{align*}
	Because $r\tilde\gc>1$ the series converges, and since every term decreases as $v_\ast\downarrow0$, we may fix $v_\ast$ so small that
	\begin{align}\label{eq:II-P-Gn-ub}
		\sup_{n\ge c^{1/\phi}}\pr\big((\cG_n^{\rm II})^{\rm c}\big)<\frac{\eps}{4}.
	\end{align}

	\medskip
	\noindent\textbf{Step 2: uniform control of the chain sum.}
	By Lemma~\ref{lem:hub-ub}~(ii), on $\cG_n^{\rm II}$,
	\begin{align*}
		S_k \ind_{\cG_n^{\rm II}}
		\le\sum_{i=0}^{k-1}\frac{\gl_i(n)}{v_{k-i} v_{k-i-1}} \go_{i,[\ell]},
		\qquad
		\gl_i(n)=c_0^{\ga}\frac{n_i^{\ga}}{n_{i+1}^{d/\tilde\gc}n_{i+2}^{d/\tilde\gc}}.
	\end{align*}
	Since $n_i=n^{\phi^i}$ we have $n_{i+1}=n_i^{\phi}$ and $n_{i+2}=n_i^{\phi^2}$, so
	\begin{align}\label{eq:II-lambda}
		\gl_i(n)=c_0^{\ga} n_i^{ \ga-(d/\tilde\gc)(\phi+\phi^2)}=c_0^{\ga} n_i^{-s},
	\end{align}
	with $s>0$ by~\eqref{eq:II-phi}. Writing again $j:=k-i\in\{1,\dots,k\}$ and using $n_i\ge c^{\phi^{-(j+1)}}$,
	\begin{align*}
		\gl_i(n)\le c_0^{\ga} c^{-s\phi^{-(j+1)}}.
	\end{align*}
	Set
	\begin{align*}
		b_j:=\frac{c_0^{\ga} c^{-s\phi^{-(j+1)}}}{v_jv_{j-1}},
		\qquad
		Y_j:=\go_{k-j,[\ell]},
		\qquad 1\le j\le k,
	\end{align*}
	so that $S_k\ind_{\cG_n^{\rm II}}\le\sum_{j=1}^kb_jY_j$, where by Lemma~\ref{lem:hub-indep} the $Y_j$ are i.i.d.~with the law of $\go_{[\ell]}$. Since $v_jv_{j-1}\ge v_\ast^2(j+2)^{-2r}$ decays only polynomially while $c^{-s\phi^{-(j+1)}}$ decays doubly exponentially,
	\begin{align}\label{eq:II-B12}
		B_1:=\sum_{j=1}^{\infty}b_j<\infty,
		\qquad
		B_2:=\sum_{j=1}^{\infty}b_j^2<\infty,
	\end{align}
	and both are finite \emph{independently of $n$ and $k$}. This is the analytic heart of Regime II: the cost of the chain is a convergent series whose terms do not depend on $n$.

	Put $\mu_\ell:=\E\go_{[\ell]}$ and $M_{p,\ell}:=\norm{\go_{[\ell]}-\mu_\ell}_p$, both finite by Lemma~\ref{lem:min-moments} since $\ell>p/\gz$. Applying Lemma~\ref{lem:p-conS_k} to the centered i.i.d.~family $\{Y_j-\mu_\ell\}$ with coefficients $b_j$, and setting
	\begin{align*}
		A_1:=\mu_\ell B_1,
		\qquad
		u(\eps):=C_pM_{p,\ell}B_2^{1/2}(4/\eps)^{1/p},
	\end{align*}
	we obtain
	\begin{align}\label{eq:II-Sk-bound}
		\pr\Big(S_k^+>A_1+u(\eps),\cG_n^{\rm II}\Big)
		 & \le\pr\Big(\sum_{j=1}^k b_j(Y_j-\mu_\ell)>u(\eps)\Big)
		\le\frac{\eps}{4},
		\text{ for } n\ge c^{1/\phi}.
	\end{align}
	The same bound holds for $S_k^-$. In particular, both bounds are uniform in $n$.

	\medskip
	\noindent\textbf{Step 3: the two endpoint edges.}
	Write $\go_{-k}^{\rm end}:=\go(\mvzero,\mvx_{-k})$ and $\go_{k}^{\rm end}:=\go(\mvx_k,\mvm)$. By~\eqref{eq:hub-terminal}, $\norm{\mvzero-\mvx_{-k}}\vee\norm{\mvm-\mvx_k}\le c_0n_k$, and $n_k=n_{k+1}^{1/\phi}<c^{1/\phi^2}$ is bounded. On $\cG_n^{\rm II}$ the outermost hubs satisfy $V_{\mvx_{\pm k}}\ge\widetilde v_kn_{k+1}^{d/\tilde\gc}=v_0n_{k+1}^{d/\tilde\gc}\ge v_0c^{d/\tilde\gc}$ with $v_0=v_\ast2^{-r}$. Hence
	\begin{align}\label{eq:II-end-bound}
		\big(W_{\mvzero\mvx_{-k}}+W_{\mvx_k\mvm}\big)\ind_{\cG_n^{\rm II}}
		\le b_{\rm end} v_0^{-1} Z_{\rm end},
		\qquad
		Z_{\rm end}:=\frac{\go_{-k}^{\rm end}}{V_{\mvzero}}+\frac{\go_{k}^{\rm end}}{V_{\mvm}},
	\end{align}
	for a constant $b_{\rm end}=b_{\rm end}(\ga,d,\phi,c)<\infty$. The vertices $\mvx_{\pm k}$ are measurable with respect to $\cF_k^{\Hub}$ (and its left-half analogue), the $\gs$-field generated by $\cF_V$ and the selection edges, and $\mvzero\notin B_{-k}$, $\mvm\notin B_k$ by Lemma~\ref{lem:hub-geom}. So the two edges $\la\mvzero\mvx_{-k}\ra$ and $\la\mvx_k\mvm\ra$ were never examined by the construction. Conditionally on that $\gs$-field their noises are therefore independent copies of $\go$, and since $V_{\mvzero},V_{\mvm}$ are $\cF_V$-measurable, the law of $Z_{\rm end}$ is that of $\go_1/V_{\mvzero}+\go_2/V_{\mvm}$ with $\go_1,\go_2$ i.i.d.~copies of $\go$ independent of $(V_{\mvzero},V_{\mvm})$. This law does not depend on $n$ and is a.s. finite, so there is $u_{\rm end}(\eps)<\infty$ with $\pr(Z_{\rm end}>u_{\rm end}(\eps))<\eps/4$, whence
	\begin{align}\label{eq:II-end-cost}
		\sup_{n\ge c^{1/\phi}}
		\pr\Big(W_{\mvzero\mvx_{-k}}+W_{\mvx_k\mvm}>b_{\rm end}v_0^{-1}u_{\rm end}(\eps), \cG_n^{\rm II}\Big)<\frac{\eps}{4}.
	\end{align}

	\medskip
	\noindent\textbf{Step 4: conclusion.}
	Set $C_1(\eps):=2(A_1+u(\eps))+b_{\rm end}v_0^{-1}u_{\rm end}(\eps)$. By the decomposition~\eqref{eq:Tn-Sk},
	\begin{align*}
		\{T(\mvzero,\mvm)>C_1(\eps)\}
		\subseteq(\cG_n^{\rm II})^{\rm c}
		 & \cup\big\{S_k^+>A_1+u(\eps);\cG_n^{\rm II}\big\}
		\cup\big\{S_k^->A_1+u(\eps);\cG_n^{\rm II}\big\}                                                            \\
		 & \cup\big\{W_{\mvzero\mvx_{-k}}+W_{\mvx_k\mvm}>b_{\rm end}v_0^{-1}u_{\rm end}(\eps);\cG_n^{\rm II}\big\},
	\end{align*}
	so~\eqref{eq:II-P-Gn-ub},~\eqref{eq:II-Sk-bound} and~\eqref{eq:II-end-cost} give
	$\sup_{\norm{\mvu}\ge c^{1/\phi}}\pr(T(\mvzero,\mvu)>C_1(\eps))<\eps$.
	Combining with Step 0 applied with $N=c^{1/\phi}$ and taking $C(\eps):=\max\{C_1(\eps),C_N(\eps)\}$ proves~\eqref{eq:II-uniform}.
\end{proof}

\begin{proof}[Proof of Proposition~\ref{prop:II-ball}]
	(i) Fix $\eps\in(0,1)$ and let $t_\eps:=C(\eps)$ from Theorem~\ref{thm:II-ub}, so that with $A_n:=\{T(\mvzero,\lceil n\mvx\rceil)\le t_\eps\}$ we have $\inf_{n\ge1}\pr(A_n)\ge1-\eps$. Since $\lceil n\mvx\rceil\to\infty$ in norm, on $\limsup_nA_n$ infinitely many distinct vertices lie in $\vB_{t_\eps}(\mvzero)$, so
	$\limsup_nA_n\subseteq\{\abs{\vB_{t_\eps}(\mvzero)}=\infty\}$.
	By reverse Fatou, $\pr(\limsup_nA_n)\ge\limsup_n\pr(A_n)\ge1-\eps$. An infinite subset of $\dZ^d$ is unbounded, so $\abs{\vB_t(\mvzero)}=\infty$ if and only if $\vD_t(\mvzero)=\infty$, giving (i).

	(ii) Let $\eps_m:=2^{-m}$, let $t_m:=t_{\eps_m}$ be given by (i), and set $s_m:=\max\{t_1,\dots,t_m\}$ and $E_m:=\{\abs{\vB_{s_m}(\mvzero)}=\infty\}$. Both $s_m$ and $E_m$ are nondecreasing in $m$, and since $\vB_t(\mvzero)$ increases in $t$,
	$\pr(E_m)\ge\pr(\abs{\vB_{t_m}(\mvzero)}=\infty)\ge1-2^{-m}$.
	As $\bigcup_mE_m\subseteq\{\tau_\infty<\infty\}$ and $E_m$ increases,
	$\pr(\tau_\infty<\infty)\ge\lim_m\pr(E_m)=1$.
\end{proof}

The next two proofs are first-moment arguments and use neither ansatz.

\begin{proof}[Proof of Theorem~\ref{thm:II-lb}]
	Since $\ga>d/\theta$ and $\ga>d/\gc$ we may choose $\eta$ with
	\begin{align*}
		\frac{d}{\ga}<\eta<\theta\wedge\gc.
	\end{align*}
	By~\eqref{def:go} and $\eta<\theta$ there is $C_1<\infty$ with $\pr(\go\le t)\le C_1t^{\eta}$ for all $t\in(0,1]$, and trivially for $t\ge1$ after enlarging $C_1$. By~\eqref{def:V_x} and $\eta<\gc$ we have $\E V^{\eta}<\infty$. For $\mvy\ne\mvzero$, conditioning on the two endpoint weights,
	\begin{align*}
		\pr\big(W_{\mvzero\mvy}\le t\mid V_{\mvzero},V_{\mvy}\big)
		=\pr\Big(\go_{\mvzero\mvy}\le t\frac{V_{\mvzero}V_{\mvy}}{\norm{\mvy}^{\ga}} \,\Big|\, V_{\mvzero},V_{\mvy}\Big)
		\le C_1t^{\eta}\frac{V_{\mvzero}^{\eta}V_{\mvy}^{\eta}}{\norm{\mvy}^{\ga\eta}}.
	\end{align*}
	Taking expectations and using independence,
	\begin{align}\label{eq:single-edge-tail}
		\pr\big(W_{\mvzero\mvy}\le t\big)
		\le C_1\big(\E V^{\eta}\big)^2 t^{\eta} \norm{\mvy}^{-\ga\eta}
		=:C_2 t^{\eta}\norm{\mvy}^{-\ga\eta},
		\qquad t\ge0,
	\end{align}
	where the bound for $t>1$ holds because $C_1\ge1$ may be assumed.
	Every path from $\mvzero$ to any $\mvu\ne\mvzero$ contains an edge incident to $\mvzero$, so $T(\mvzero,\mvu)\ge T_{\mvzero}:=\inf_{\mvy\ne\mvzero}W_{\mvzero\mvy}$.
	For every $h>0$, \eqref{eq:single-edge-tail} and a union bound give
	\begin{align*}
		\pr(T_{\mvzero}\le t)
		\le\sum_{\mvy\ne\mvzero}\pr(W_{\mvzero\mvy}\le t+h)
		\le C_2(t+h)^{\eta}
		\sum_{\mvy\ne\mvzero}\norm{\mvy}^{-\ga\eta}
		=C_3(t+h)^{\eta}.
	\end{align*}
	The series converges since $\ga\eta>d$. Letting $h\downarrow0$ completes the proof.
\end{proof}

\begin{proof}[Proof of Theorem~\ref{thm:II-hop}]
	Fix $K>0$. Since $\ga\gc>d$ and $\ga\theta>d$, we may choose $\eta$ with
	\begin{align*}
		\frac{d}{\ga}<\eta<\gc\wedge\theta.
	\end{align*}
	Then $\ga\eta-d>0$, so we may fix
	\begin{align*}
		a_0>\max\Big\{1,\frac{d}{\ga\eta-d}\Big\},
		\qquad
		\gd_\ast:=a_0(\ga\eta-d)-d>0.
	\end{align*}
	(In Regime II one has $\ga\eta<2d$, so $d/(\ga\eta-d)>1$ and the maximum is the second term.)
	Let $R_0\ge2$ be a large integer to be chosen, and put $R_i:=\lfloor R_0^{(a_0)^i}\rfloor$ for $i\ge0$. For $R_0$ large, $R_i\ge2R_{i-1}$ for all $i\ge1$. Define the ``long cheap edge'' events
	\begin{align*}
		\fB_i:=\Big\{\exists \mvu\in B(\mvzero,R_{i-1}), \mvv\notin B(\mvzero,R_i):W_{\mvu\mvv}\le K\Big\},
		\qquad i\ge1.
	\end{align*}
	By~\eqref{eq:single-edge-tail}, translation invariance and a union bound,
	\begin{align*}
		\pr(\fB_i)\le\sum_{\mvu\in B(\mvzero,R_{i-1})}\sum_{\mvv\notin B(\mvzero,R_i)}C_2K^{\eta}\norm{\mvu-\mvv}^{-\ga\eta}.
	\end{align*}
	For such $\mvu,\mvv$ we have $\norm{\mvu-\mvv}\ge R_i-R_{i-1}\ge R_i/2$, so, using $\ga\eta>d$,
	$\sum_{\mvv\notin B(\mvzero,R_i)}\norm{\mvu-\mvv}^{-\ga\eta}\lesssim R_i^{-(\ga\eta-d)}$
	uniformly in $\mvu$, whence
	\begin{align*}
		\pr(\fB_i)\lesssim R_{i-1}^{d}R_i^{-(\ga\eta-d)}
		\le C R_0^{\{d-a_0(\ga\eta-d)\}(a_0)^{i-1}}
		=C\exp\big\{-\gd_\ast(a_0)^{i-1}\log R_0\big\}.
	\end{align*}
	Here $C$ absorbs the harmless factors caused by the integer parts in
	$R_i=\lfloor R_0^{(a_0)^i}\rfloor$.
	The right-hand side is summable in $i$, and given $\eps>0$ we may take $R_0=R_0(\eps)$ so large that $\sum_{i\ge1}\pr(\fB_i)<\eps$.

	Work on $\gO_{R_0}:=\bigcap_{i\ge1}\fB_i^{\rm c}$ and let $\mvpi=(\mvy_0,\dots,\mvy_m)\in\cP_n(K)$. All edge weights are nonnegative, so every edge of $\mvpi$ has weight at most $K$. Hence on $\gO_{R_0}$ no edge of $\mvpi$ can leave $B(\mvzero,R_i)$ from inside $B(\mvzero,R_{i-1})$, and induction on $j$ gives $\mvy_j\in B(\mvzero,R_j)$ for $0\le j\le m$. Since $\mvy_m=\lceil n\mvx\rceil$ has norm $\asymp n$, we need $R_m\gtrsim n$, and $R_m\le R_0^{(a_0)^m}$ forces
	$m\ge\log\log n/\log a_0-C_{K,\eps}$
	for a constant $C_{K,\eps}$ depending only on $R_0(\eps)$ and $\mvx$. Therefore
	\begin{align*}
		\pr\Big(\exists \mvpi\in\cP_n(K)\text{ with }\abs{\mvpi}\le\frac{\log\log n}{\log a_0}-C_{K,\eps}\Big)
		\le\pr\big(\gO_{R_0}^{\rm c}\big)\le\sum_{i\ge1}\pr(\fB_i)<\eps
	\end{align*}
	for all large $n$, which is the first claim. For the second, if $c_K<1/\log a_0$ then $c_K\log\log n\le\log\log n/\log a_0-C_{K,\eps}$ for all large $n$, whatever $\eps$. So the probability in question is eventually at most $\eps$ for every $\eps>0$, and hence tends to $0$.
\end{proof}

\section{Proofs for Regime III}\label{sec:proof-III}

Here $d/\theta<\ga<2d/\theta$, and by Remark~\ref{rem:III-gc} the hypothesis $\ga>2d/\gc$ plays no role in the argument. The scale map is again $f_\phi(r)=r^{\phi}$ with $\phi\in(0,1)$, but now it drives the \emph{binary} ansatz, so the depth $k$ produces $2^k$ terminal segments. The optimization of $\phi$ against the entropy factor $2^k$ is what produces $\Delta_{\rm III}$.

\begin{proof}[Proof of Theorem~\ref{thm:III-ub}]
	If $\mvx=\mvzero$ then $T_n=0$ for all $n$. So, we assume that $\mvx\ne\mvzero$. Fix $\eps>0$ and recall $\Delta_{\rm III}=\log2/\log(2d/(\ga\theta))$.

	\medskip
	\noindent\textbf{Choice of parameters.}
	By continuity of
	$\eps'\mapsto\log 2/\log((2d/\theta-\eps')/\ga)$ at $0$, choose $\eps'>0$
	small enough that
	\begin{align}\label{eq:eps-prime-III-small}
		0<\eps'<\tfrac14\Big(\frac{2d}{\theta}-\ga\Big),
		\qquad
		\frac{\log2}{\log((2d/\theta-\eps')/\ga)}
		<\Delta_{\rm III}+\frac{\eps}{8},
	\end{align}
	and then choose $\phi\in(0,1)$, sufficiently close to
	$\ga/(2d/\theta-\eps')$ from above, so that
	\begin{align}\label{eq:III-phi-cond-final}
		\frac{\ga}{2d/\theta-\eps'}<\phi<1,
		\qquad
		\gb_\phi+\frac{\eps}{4}\le\Delta_{\rm III}+\frac{\eps}{2},
		\qquad
		\gb_\phi:=\frac{\log2}{\log(1/\phi)}.
	\end{align}
	The first inequality is nonempty by~\eqref{eq:eps-prime-III-small}, while the
	second follows from the displayed strict bound and the continuity and monotonicity
	of $\gb_\phi$ in $\phi$. Set
	\begin{align*}
		n_i:=f_\phi^{(i)}(n)=n^{\phi^i},
		\qquad
		k=k(n):=\Big\lceil\frac{\log\big(\log n/((\eps/4)\log\log n)\big)}{\log(1/\phi)}\Big\rceil.
	\end{align*}
	Then $\phi^k\log n\asymp(\eps/4)\log\log n$, so $k\asymp\log\log n$ and, for all large $n$,
	\begin{align}\label{eq:ord-nk-2k-III-final}
		(\log n)^{\phi\eps/4}\le n_k\le(\log n)^{\eps/4},
		\qquad
		2^k\le C(\log n)^{\gb_\phi}.
	\end{align}
	In particular $n_k\to\infty$, so the separation conditions~\eqref{eq:edge-scale-separation} and~\eqref{eq:root-separation} hold for all large $n$: indeed $n_{i+1}=n_i^{\phi}\le n_i/16$ as soon as $n_i\ge n_k\ge16^{1/(1-\phi)}$, and $n_1=n^{\phi}\le\tfrac{1}{16}\norm{\mvm}_\infty$ since $\norm{\mvm}_\infty\asymp n$. Combining the two displays of~\eqref{eq:ord-nk-2k-III-final} with~\eqref{eq:III-phi-cond-final},
	\begin{align}\label{eq:local-scale-III-final}
		2^kn_k\le C(\log n)^{\gb_\phi+\eps/4}\le C(\log n)^{\Delta_{\rm III}+\eps/2}.
	\end{align}

	\medskip
	\noindent\textbf{The bridge sum.}
	Apply Lemma~\ref{lem:bridge-trunc} with this $\eps'$ and with $R>0$ to be chosen. Since $n_i=n^{\phi^i}$ and $n_{i+1}=n^{\phi^{i+1}}$,
	\begin{align*}
		\gl_i=Cn_i^{\ga}n_{i+1}^{-2d/\theta+\eps'}
		=Cn^{\phi^i\ga-\phi^{i+1}(2d/\theta-\eps')}
		=Cn^{-c_\phi\phi^i},
		\qquad
		c_\phi:=\phi\Big(\frac{2d}{\theta}-\eps'\Big)-\ga>0
	\end{align*}
	by the first constraint in~\eqref{eq:III-phi-cond-final}. In particular $\gl_i\le C$ for every $i$, so
	\begin{align}\label{eq:Upsilon-br-III-final}
		\Upsilon_{\phi,k}^{\rm br}=\sum_{i=0}^{k-1}2^i\gl_i\le C2^k\le C(\log n)^{\gb_\phi}.
	\end{align}
	The bound $\gl_i\le C$ is crude -- the bridges are in fact much cheaper than $O(1)$ -- but it is enough, because in this regime the cost is dominated not by the individual bridges but by their sheer number.

	\medskip
	\noindent\textbf{The bad events.}
	Since $k\asymp\log\log n$, the factor $C^k$ in Lemmas~\ref{lem:Ak-prob}
	and~\ref{lem:Agd-prob} is at most $(\log n)^{c_3}$ for a constant
	$c_3=c_3(\phi,\eps)$. By Lemma~\ref{lem:Ak-prob} and
	\eqref{eq:ord-nk-2k-III-final},
	\begin{align*}
		\pr(\cA_k^{\rm c})
		\le C^k\exp\{-cn_k^d\}
		\le(\log n)^{c_3}\exp\{-c(\log n)^{d\phi\eps/4}\}.
	\end{align*}
	Consequently, for every $q>0$,
	\begin{align}\label{eq:Ak-small-III-final}
		\pr(\cA_k^{\rm c})\le(\log n)^{-q}
	\end{align}
	for all large $n$, a stretched-exponential in $\log n$ beating every power. By~\eqref{eq:bridge-trunc-prob} and $n_{i+1}\ge n_k$,
	\begin{align}\label{eq:Bk-small-III-final}
		\pr\big(\cB_k^{\rm c}\cap\cA_k^{\rm gd}\big)\le C\sum_{i=0}^{k-1}2^in_{i+1}^{-2dR}\le C2^kn_k^{-2dR}
		\le C(\log n)^{\gb_\phi-dR\phi\eps/2},
	\end{align}
	which tends to $0$ once $R$ is chosen with $\gb_\phi-dR\phi\eps/2<0$. Choose $q_{\rm int}$ so large that
	\begin{align*}
		c_3+\gb_\phi-\frac{\phi\eps}{4} q_{\rm int}<0.
	\end{align*}
	Use the corridor construction with $q_0=q_{\rm int}$ and apply Lemmas~\ref{lem:Sloc-control} and~\ref{lem:Agd-prob}. Then
	\begin{align*}
		C_{q_{\rm int}}C^kn_k^{-q_{\rm int}}+\pr\big(\cL_k^{\rm c}\cap\cA_k^{\rm gd}\big)
		\le C(\log n)^{c_3}n_k^{-q_{\rm int}}+C2^kn_k^{-q_{\rm int}}+\rho(n_k),
	\end{align*}
	whose terms vanish by~\eqref{eq:ord-nk-2k-III-final}, the choice of $q_{\rm int}$, and $n_k\to\infty$. Hence
	\begin{align}\label{eq:Lk-small-III-final}
		C_{q_{\rm int}}C^kn_k^{-q_{\rm int}}+\pr\big(\cL_k^{\rm c}\cap\cA_k^{\rm gd}\big)\longrightarrow0.
	\end{align}

	\medskip
	\noindent\textbf{Conclusion.}
	By~\eqref{eq:Upsilon-br-III-final} and~\eqref{eq:local-scale-III-final},
	\begin{align*}
		\Upsilon_{\phi,k}^{\rm br}+2^kc_1n_k
		\le C(\log n)^{\gb_\phi}+C(\log n)^{\Delta_{\rm III}+\eps/2}
		\le C(\log n)^{\Delta_{\rm III}+\eps}
	\end{align*}
	for all large $n$, using $\gb_\phi\le\Delta_{\rm III}+\eps/4$. Proposition~\ref{prop:Ansatz-bound}, together with~\eqref{eq:Ak-small-III-final},~\eqref{eq:Bk-small-III-final} and~\eqref{eq:Lk-small-III-final}, gives
	$\pr\big(T_n>C(\log n)^{\Delta_{\rm III}+\eps}\big)\to0$.
\end{proof}

\section{\texorpdfstring{Proofs for Regime $\rm IV_E$}{Proofs for Regime IV-E}}\label{sec:proof-IVE}

Now $2d/\theta<\ga<1+2d/\theta$, so by~\eqref{eq:heur-bridge} longer bridges are more expensive and the recursion must be truncated early rather than late: the scale map becomes \emph{geometric}, $f_\phi(r)=r/\phi$, the depth is $k\asymp\log n$, and the bridge sum is dominated by its first term, the single macroscopic bridge, as the concavity heuristic (E3) predicts.

\begin{proof}[Proof of Theorem~\ref{thm:IVE-ub}]
	Assume $\mvx\ne\mvzero$, the case $\mvx=\mvzero$ being trivial, and fix $\eps>0$. Recall $\Delta_{\rm IV_E}=\ga-2d/\theta\in(0,1)$ and choose
	\begin{align*}
		0<\eps'<\min\Big\{\frac{\eps}{4}, \frac12\Big(\ga-\frac{2d}{\theta}\Big), 1\Big\},
		\qquad
		\Delta_+:=\Delta_{\rm IV_E}+\eps'.
	\end{align*}
	Choose $\phi>16$ so large that, for all large $n$,
	\begin{align}\label{eq:IVE-root-sep}
		\frac{n}{\phi}\le\frac{1}{16}\norm{\mvm}_\infty,
	\end{align}
	which is possible because $\norm{\mvm}_\infty\asymp n$ (with constants depending on $\mvx$), and moreover, using $2d/\theta-\eps'-\ga<0$,
	\begin{align}\label{eq:phi-IVE-cond-final}
		2 \phi^{2d/\theta-\eps'-\ga}<\frac12,
		\qquad
		\log_\phi2\le\frac{\Delta_+}{4}.
	\end{align}
	Choose then $s_\star\in(0,1)$ so small that
	\begin{align}\label{eq:s-star-IVE}
		s_\star+(1-s_\star)\log_\phi2<\frac{\Delta_+}{2},
	\end{align}
	which is possible by the second part of~\eqref{eq:phi-IVE-cond-final}: it suffices that $s_\star<\Delta_+/4$. Set
	\begin{align*}
		f_\phi(r):=\frac r\phi,
		\qquad
		n_i=\frac{n}{\phi^i},
		\qquad
		k:=\big\lfloor(1-s_\star)\log_\phi n\big\rfloor.
	\end{align*}
	Then~\eqref{eq:edge-scale-separation} holds since $\phi>16$,~\eqref{eq:root-separation} is~\eqref{eq:IVE-root-sep}, and
	\begin{align}\label{eq:nk-2k-IVE-final}
		n_k=n^{s_\star+o(1)},
		\qquad
		2^k=n^{(1-s_\star)\log_\phi2+o(1)} ,
	\end{align}
	whence, by~\eqref{eq:s-star-IVE},
	\begin{align}\label{eq:local-IVE-scale-final}
		2^kn_k=n^{s_\star+(1-s_\star)\log_\phi2+o(1)}\le n^{\Delta_+/2}
	\end{align}
	for all large $n$.

	\medskip
	\noindent\textbf{The bridge sum.}
	Apply Lemma~\ref{lem:bridge-trunc} with this $\eps'$ and with $R$ to be chosen. Then
	\begin{align*}
		\gl_i=Cn_i^{\ga}n_{i+1}^{-2d/\theta+\eps'}
		=C\Big(\frac{n}{\phi^i}\Big)^{\ga}\Big(\frac{n}{\phi^{i+1}}\Big)^{-2d/\theta+\eps'}
		=Cn^{\Delta_+}\phi^{2d/\theta-\eps'} r_\phi^{ i},
		\qquad
		r_\phi:=\phi^{2d/\theta-\eps'-\ga}<1.
	\end{align*}
	Absorbing $\phi^{2d/\theta-\eps'}$ into $C$ and using $2r_\phi<\tfrac12$ from~\eqref{eq:phi-IVE-cond-final},
	\begin{align}\label{eq:Upsilon-br-IVE-final}
		\Upsilon_{\phi,k}^{\rm br}=\sum_{i=0}^{k-1}2^i\gl_i
		\le Cn^{\Delta_+}\sum_{i\ge0}(2r_\phi)^i\le Cn^{\Delta_+}.
	\end{align}
	In contrast with Regime III, the geometric series is now dominated by its $i=0$ term: the cost is that of one macroscopic bridge, and the entire subtree beneath it contributes only a bounded multiple.

	\medskip
	\noindent\textbf{The bad events.}
	Since $k\le\log_\phi n$, the factor $C^k$ of Lemmas~\ref{lem:Ak-prob} and~\ref{lem:Agd-prob} is at most $n^{c_3}$ with $c_3:=\log_\phi C$. By Lemma~\ref{lem:Ak-prob} and~\eqref{eq:nk-2k-IVE-final}, $\pr(\cA_k^{\rm c})\le n^{c_3}\exp\{-cn^{ds_\star+o(1)}\}\le n^{-q}$ for every $q>0$ and all large $n$. By~\eqref{eq:bridge-trunc-prob} and $n_{i+1}\ge n_k$,
	\begin{align}\label{eq:Bk-small-IVE-final}
		\pr\big(\cB_k^{\rm c}\cap\cA_k^{\rm gd}\big)\le C2^kn_k^{-2dR}
		=Cn^{(1-s_\star)\log_\phi2-2ds_\star R+o(1)}\longrightarrow0
	\end{align}
	once $R$ is chosen with $(1-s_\star)\log_\phi2-2ds_\star R<0$. Choose $q_{\rm int}$ so large that $c_3+(1-s_\star)\log_\phi2-s_\star q_{\rm int}<0$, use the corridor construction with $q_0=q_{\rm int}$, and apply Lemmas~\ref{lem:Sloc-control} and~\ref{lem:Agd-prob}. Then
	\begin{align}\label{eq:Lk-small-IVE-final}
		C_{q_{\rm int}}C^kn_k^{-q_{\rm int}}+\pr\big(\cL_k^{\rm c}\cap\cA_k^{\rm gd}\big)
		\le Cn^{c_3}n_k^{-q_{\rm int}}+C2^kn_k^{-q_{\rm int}}+\rho(n_k)\longrightarrow0,
	\end{align}
	because $n_k=n^{s_\star+o(1)}\to\infty$.

	\medskip
	\noindent\textbf{Conclusion.}
	By~\eqref{eq:Upsilon-br-IVE-final} and~\eqref{eq:local-IVE-scale-final},
	$\Upsilon_{\phi,k}^{\rm br}+2^kc_1n_k\le Cn^{\Delta_+}$,
	and $\Delta_+=\Delta_{\rm IV_E}+\eps'\le\Delta_{\rm IV_E}+\eps/4$, so $n^{\Delta_+}\le n^{\Delta_{\rm IV_E}+\eps}$ for large $n$. Proposition~\ref{prop:Ansatz-bound} together with the three bad-event estimates gives
	$\pr\big(T_n>Cn^{\Delta_{\rm IV_E}+\eps}\big)\to0$.
\end{proof}

\section{\texorpdfstring{Proofs for Regime $\rm IV_H$}{Proofs for Regime IV-H}}\label{sec:proof-IVH}

Here $2d/\gc<\ga<1+2d/\gc$, and we return to the hub chain with a geometric scale map. With the auxiliary exponent $\tilde\gc>\gc$ chosen below, the chain has $k\asymp\log n$ links and the costs $\gl_i(n)$ form a geometric sequence with ratio $\phi^{-(\ga-2d/\tilde\gc)}<1$. The total is therefore dominated by the first, macroscopic, link.

\begin{proof}[Proof of Theorem~\ref{thm:IVH-ub}]
	Fix $\eps>0$ and $\gd\in(0,1)$. Choose $\tilde\gc>\gc$, with $\Delta:=\ga-2d/\tilde\gc$, such that
	\begin{align*}
		0<\Delta-\Delta_{\rm IV_H}
		=\frac{2d}{\gc}-\frac{2d}{\tilde\gc}
		<\min\{\eps,1-\Delta_{\rm IV_H}\}.
	\end{align*}
	In particular, $0<\Delta<1$ and $\Delta<\Delta_{\rm IV_H}+\eps$.
	Use the normalized hub weights and Lemma~\ref{lem:hub-ub} with this fixed $\tilde\gc$. All constants below may depend on $\eps$ through this choice; this dependence is suppressed.
	Choose $\phi\ge\max\{8, 32/\norm{\mvx}_\infty\}$, $p:=2$, $\ell\in\dN$ with $\ell>2/\gz$, then $r>1/\tilde\gc$, and finally $A>r\tilde\gc/d$. Apply the Hub-Chain Ansatz of Section~\ref{ssec:Ansatz_II_IVH} with
	\begin{align*}
		f_\phi(t):=\frac t\phi,
		\qquad
		n_i=\frac{n}{\phi^i},
		\qquad
		k=k(n):=\max\big\{m\ge0:n_{m+1}\ge(\log n)^A\big\},
	\end{align*}
	so that $\phi\ge8$ gives the first part of~\eqref{eq:hub-scale-separation}, the root separation $n_1=n/\phi\le\tfrac1{16}\norm{\mvm}_\infty$ holds for all large $n$ because $\norm{\mvm}_\infty\ge n\norm{\mvx}_\infty-1$, and
	\begin{align}\label{eq:IVH-scales}
		(\log n)^A\le n_{k+1}<\phi(\log n)^A,
		\qquad
		n_k=\phi n_{k+1}<\phi^2(\log n)^A,
		\qquad
		k\asymp\log n.
	\end{align}
	By~\eqref{eq:Tn-Sk} it suffices to control $S_k^{\pm}$ and the two endpoint edges.

	\medskip
	\noindent\textbf{Step 1: the hub-good event and the chain sum.}
	For $v_\ast\in(0,\tfrac12)$ to be fixed, put $v_i:=v_\ast(i+2)^{-r}$ for $0\le i\le k$ -- note that, unlike in Regime II, the sequence is used in its natural order -- and define
	\begin{align*}
		\cG_n^{\rm IV_H}:=\cG_n^{\Hub}(\vec v)
		=\bigcap_{i=-k}^{k}\Big\{V_i^{(\ell)}\ge v_{\abs{i}}n_{\abs{i}+1}^{d/\tilde\gc}\Big\}.
	\end{align*}
	Since $v_i\ge v_k$ and $n_{i+1}\ge n_{k+1}$ for $0\le i\le k$, we have $v_i n_{i+1}^{d/\tilde\gc}\ge v_k n_{k+1}^{d/\tilde\gc} \asymp v_\ast(k+2)^{-r}(\log n)^{Ad/\tilde\gc},$
	and since $k\asymp\log n$ and $A>r\tilde\gc/d$ the right-hand side tends to infinity. So there is $N_0<\infty$ with $v_i n_{i+1}^{d/\tilde\gc}\ge x_V$ for all $0\le i\le k$ whenever $n\ge N_0$, which is the hypothesis~\eqref{eq:hub-xV-cond}. Lemma~\ref{lem:hub-ub}~(i) then gives, for $n\ge N_0$,
	\begin{align}\label{eq:IVH-P-Gn-ub}
		\pr\big((\cG_n^{\rm IV_H})^{\rm c}\big)
		\le2\sum_{i=0}^{\infty}C_1\exp\big(-c_2v_\ast^{-\tilde\gc}(i+2)^{r\tilde\gc}\big)<\frac{\gd}{4}
	\end{align}
	after fixing $v_\ast$ small, the series converging because $r\tilde\gc>1$.

	On $\cG_n^{\rm IV_H}$, Lemma~\ref{lem:hub-ub}(ii) gives $S_k\ind_{\cG_n^{\rm IV_H}}\le\sum_{i=0}^{k-1}\gl_i(n)(v_iv_{i+1})^{-1}\go_{i,[\ell]}$. Substituting $n_i=n/\phi^i$ into~\eqref{eq:sum-gl},
	\begin{align*}
		\gl_i(n)
		=c_0^{\ga}\frac{(n/\phi^i)^{\ga}}{(n/\phi^{i+1})^{d/\tilde\gc}(n/\phi^{i+2})^{d/\tilde\gc}}
		=c_0^{\ga} n^{\ga-2d/\tilde\gc} \phi^{-i\ga+(2i+3)d/\tilde\gc}
		=n^{\Delta} b_i,
		\qquad
		b_i:=c_0^{\ga}\phi^{3d/\tilde\gc}\phi^{-i\Delta},
	\end{align*}
	so that
	\begin{align*}
		S_k\ind_{\cG_n^{\rm IV_H}}\le n^{\Delta}\sum_{i=0}^{k-1}a_i\go_{i,[\ell]},
		\qquad
		a_i:=\frac{b_i}{v_iv_{i+1}}
		=c_0^{\ga}\phi^{3d/\tilde\gc}v_\ast^{-2} \phi^{-i\Delta}(i+2)^{r}(i+3)^{r}.
	\end{align*}
	Because $\Delta>0$ and $\phi>1$, the geometric factor $\phi^{-i\Delta}$ dominates the polynomial one, so
	$A_1:=\sum_{i\ge0}a_i<\infty$ and $A_2:=\sum_{i\ge0}a_i^2<\infty$,
	both independent of $n$ and $k$. With $\mu_\ell:=\E\go_{[\ell]}$ and $M_{p,\ell}:=\norm{\go_{[\ell]}-\mu_\ell}_p<\infty$ (Lemma~\ref{lem:min-moments}, $\ell>p/\gz$), Lemma~\ref{lem:p-conS_k} applied to the centered i.i.d.~family of Lemma~\ref{lem:hub-indep} gives, with
	$C_{\rm br}(\gd):=\mu_\ell A_1+C_pM_{p,\ell}A_2^{1/2}(4/\gd)^{1/p}$,
	\begin{align}\label{eq:IVH-bridge-ub}
		\pr\Big(S_k^{\pm}>C_{\rm br}(\gd) n^{\Delta} ; \cG_n^{\rm IV_H}\Big)\le\frac{\gd}{4},
		\quad n\ge N_0.
	\end{align}

	\medskip
	\noindent\textbf{Step 2: the endpoint edges.}
	Unlike in Regime II the terminal scale $n_k\asymp(\log n)^A$ diverges, so the endpoint edges are not $O_{\pr}(1)$. But they are polylogarithmic, hence negligible against $n^{\Delta}$. By~\eqref{eq:hub-terminal} and~\eqref{eq:IVH-scales}, $\norm{\mvzero-\mvx_{-k}}\le c_0n_k\le c_0\phi^2(\log n)^A$, and on $\cG_n^{\rm IV_H}$ the outermost hubs obey
	$V_{\mvx_{\pm k}}\ge v_kn_{k+1}^{d/\tilde\gc}\ge v_\ast(k+2)^{-r}(\log n)^{Ad/\tilde\gc}$.
	Therefore
	\begin{align}\label{eq:IVH-end-bound}
		\big(W_{\mvzero\mvx_{-k}}+W_{\mvx_k\mvm}\big)\ind_{\cG_n^{\rm IV_H}}
		\le C v_\ast^{-1}(k+2)^{r} (\log n)^{A(\ga-d/\tilde\gc)}
		\Big(\frac{\go_{-k}^{\rm end}}{V_{\mvzero}}+\frac{\go_{k}^{\rm end}}{V_{\mvm}}\Big)
		\le C(\log n)^{\gk} Z_{\rm end},
	\end{align}
	where $\gk:=A(\ga-d/\tilde\gc)+r<\infty$, using $k\asymp\log n$, and where, exactly as in Step 3 of the proof of Theorem~\ref{thm:II-ub}, $Z_{\rm end}=\go_1/V_{\mvzero}+\go_2/V_{\mvm}$ in law, with $\go_1,\go_2$ i.i.d.~copies of $\go$ independent of $(V_{\mvzero},V_{\mvm})$. In particular the law of $Z_{\rm end}$ does not depend on $n$ and $Z_{\rm end}<\infty$ a.s. Here we used $\ga>2d/\tilde\gc>d/\tilde\gc$, so the exponent $\gk$ is positive but finite. Since $\Delta>0$ we have $(\log n)^{\gk}/n^{\Delta}\to0$, hence
	\begin{align}\label{eq:IVH-end-ub}
		\pr\Big(W_{\mvzero\mvx_{-k}}+W_{\mvx_k\mvm}>n^{\Delta} ; \cG_n^{\rm IV_H}\Big)
		\le\pr\Big(Z_{\rm end}>\frac{n^{\Delta}}{C(\log n)^{\gk}}\Big)<\frac{\gd}{4}
	\end{align}
	for all $n\ge N_1$, some $N_1\ge N_0$.

	This single-edge estimate replaces, and is much simpler than, a comparison with Regime II at a reduced exponent $\widehat\ga\in(d/\gc,2d/\gc)$. The latter would also require transferring Theorem~\ref{thm:II-ub} to the random endpoint $\mvx_{-k}$, whereas here the terminal edge is a single unexamined edge whose noise is conditionally a fresh copy of $\go$.

	\medskip
	\noindent\textbf{Step 3: conclusion.}
	Set $C_1(\gd):=2C_{\rm br}(\gd)+1$. By~\eqref{eq:Tn-Sk}, ~\eqref{eq:IVH-P-Gn-ub},~\eqref{eq:IVH-bridge-ub} for both signs, and~\eqref{eq:IVH-end-ub}, $\sup_{n\ge N_1} \pr\big(T_n>C_1(\gd)n^{\Delta}\big)<\gd$. Since $\gd\in(0,1)$ was arbitrary, $T_n=O_{\pr}(n^\Delta)$.
	Since $\Delta<\Delta_{\rm IV_H}+\eps$, this completes the proof.
\end{proof}
\section{\texorpdfstring{Proofs for Regimes $\rm V_E$ and $\rm V_H$}{Proofs for Regimes V-E and V-H}}\label{sec:proof-V}

\begin{proof}[Proof of Proposition~\ref{prop:lin-ub-input}]
	By translation invariance we may take $\mvu=\mvzero$ and write $\mvv=:\mvw$, $m:=\norm{\mvw}$.

	\medskip
	\noindent\textbf{The corridor.}
	Fix $a>0$ with $p_a=\pr(V>a)>0$. For part~\textup{(i)}, choose
	$p_0>\max\{2,2q_0\}$ and then $\ell\in\dN$ so that $\ell\gz>p_0$. For
	part~\textup{(ii)}, choose instead $\ell\gz>p$. Fix $R<\infty$ so large that every
	$\ell_\infty$ ball of radius $R$ contains at least $2\ell+2$ lattice points, and
	choose an integer $L>2dR$.

	Let $(\mveta_t)_{t\ge0}$ be a deterministic nearest-neighbor ray that starts at
	$\mvzero$, reaches $\mvw$ at time $m$, and moves monotonically in every coordinate
	(after time $m$, continue in any coordinate without reversing direction). Thus
	$\norm{\mveta_s-\mveta_t}=\abs{s-t}$. Put $\mvz_i:=\mveta_{Li}$ and
	$D_i:=B(\mvz_i,R)$ for $i\ge1$. The $D_i$ are pairwise disjoint, because
	$\norm{\mvz_i-\mvz_j}_\infty\ge L\abs{i-j}/d>2R$ for $i\ne j$. For each $i$ choose
	a deterministic set $Q_i\subseteq D_i\setminus\{\mvzero,\mvw\}$ of exactly
	$2\ell$ vertices. The sets $Q_i$ are disjoint and have the same cardinality.

	Call $i$ \emph{good} if $\abs{Q_i\cap\cV^a}\ge\ell$. The indicators of goodness
	are therefore i.i.d.~Bernoulli with parameter
	$p_\square:=\pr(\mathrm{Bin}(2\ell,p_a)\ge\ell)>0$. Let
	$I_1<I_2<\cdots$ be the good indices, set $I_0:=0$, and put
	$Y_j:=I_j-I_{j-1}$ for $j\ge1$. Then the $Y_j$ are i.i.d.~
	$\mathrm{Geom}(p_\square)$. Finally, let $M:=\lceil m/L\rceil$ and put
	\begin{align*}
		N:=\min\{j\ge1:I_j\ge M\},
		\qquad
		Y_{N+1}:=I_N-M+1,
	\end{align*}
	so that $N-1\sim\mathrm{Bin}(M-1,p_\square)$ and $Y_{N+1}$ has the
	$\mathrm{Geom}(p_\square)$ law (it is the waiting time from index $M$ to the next
	good index). All these quantities are $\cF_V$-measurable. For all sufficiently
	large $m$, the deterministic inequalities $N\le M\le m$ and
	$Y_{N+1}\le Y_N$ hold: the first uses $I_j\ge j$, and the second follows from
	$I_{N-1}<M$. This defines the quantities in~\eqref{eq:corridor-event}. The
	geometric constant $C_0$ is specified in~\eqref{eq:hop-lengths} below.

	\medskip
	\noindent\textbf{Proof of~\eqref{eq:lin-int-a}.}
	The variables $Y_j$ are geometric, hence have moments of all orders, and so does $Y_j^{2\ga}$. Fix $\Lambda_0:=\max\{1, 8\E Y_1^{2\ga}\}$. The constraint $N\le\Lambda_0m$ holds deterministically. Next, by the two facts just noted, $\sum_{j\le N+1}Y_j^{2\ga}\le2\sum_{j\le N}Y_j^{2\ga}\le2\sum_{j\le m}Y_j^{2\ga}$, a sum of $m$ i.i.d.~nonnegative variables with mean $\E Y_1^{2\ga}\le\Lambda_0/8$ each. By Rosenthal's inequality applied with an exponent $p''$, we have
	\begin{align*}
		\pr\Big(\sum_{j\le N+1}Y_j^{2\ga}>\Lambda_0m\Big)
		\le\pr\Big(\sum_{j\le m}\big(Y_j^{2\ga}-\E Y_1^{2\ga}\big)>\tfrac{\Lambda_0}{4}m\Big)
		\le C_{p''}m^{-p''/2}.
	\end{align*}
	Finally, $\pr(\max_{j\le N+1}Y_j>\eps_1m/C_0)\le\pr(\max_{j\le m}Y_j>\eps_1m/C_0)\le m(1-p_\square)^{\eps_1m/C_0}\le c_1e^{-c_2m}$. As $p''$ is arbitrary,~\eqref{eq:lin-int-a} follows.

	\medskip
	\noindent\textbf{The path, and proof of~\eqref{eq:lin-int}.}
	Assume $V_{\mvzero}\wedge V_{\mvw}>a$ and let
	$\cD_j^a:=Q_{I_j}\cap\cV^a$, so $\bigl|\cD_j^a\bigr|\ge\ell$ on the corridor event.
	Set $\mvy_0:=\mvzero$ and define recursively, for $1\le j\le N-1$,
	\begin{align*}
		\mvy_j:=\argmin\big\{\go(\mvy_{j-1},\mvz):\mvz\in\cD_j^a\big\},
	\end{align*}
	and at the last step, in order to optimize the final edge as well,
	\begin{align}\label{eq:last-step}
		\mvy_N:=\argmin\big\{\go(\mvy_{N-1},\mvz)+\go(\mvz,\mvw):\mvz\in\cD_N^a\big\}.
	\end{align}
	The resulting walk $\la\mvzero \mvy_1\cdots\mvy_N \mvw\ra$ is self-avoiding.
	Since $\mvw=\mveta_m$, $M=\lceil m/L\rceil$, and the $Q_i$ are contained in
	$B(\mveta_{Li},R)$, there is $C_0=C_0(L,R,d)$ such that
	\begin{align}\label{eq:hop-lengths}
		\norm{\mvy_{j-1}-\mvz}\le C_0Y_j\quad(\mvz\in Q_{I_j}),
		\qquad
		\norm{\mvw-\mvy_N}\le C_0Y_{N+1}.
	\end{align}
	Hence
	\begin{align}\label{eq:linear-path-decomp}
		T(\mvzero,\mvw)
		\le\sum_{j=1}^{N-1}W_{\mvy_{j-1}\mvy_j}
		+\Big(W_{\mvy_{N-1}\mvy_N}+W_{\mvy_N\mvw}\Big).
	\end{align}
	All the vertices $\mvy_j$ have weight $>a$, and so do $\mvzero$ and $\mvw$ by assumption, so using~\eqref{eq:hop-lengths},
	\begin{align}\label{eq:hop-cost}
		W_{\mvy_{j-1}\mvy_j}\le\frac{(C_0Y_j)^{\ga}}{a^2}\min_{\mvz\in\cD_j^a}\go(\mvy_{j-1},\mvz),
		\quad 1\le j\le N-1,
	\end{align}
	and, by~\eqref{eq:last-step} and $\norm{\mvy_{N-1}-\mvy_N}\vee\norm{\mvy_N-\mvw}\le C_0(Y_N\vee Y_{N+1})$,
	\begin{align}\label{eq:last-cost}
		W_{\mvy_{N-1}\mvy_N}+W_{\mvy_N\mvw}
		\le\frac{\big(C_0(Y_N\vee Y_{N+1})\big)^{\ga}}{a^2}
		\min_{\mvz\in\cD_N^a}\Big(\go(\mvy_{N-1},\mvz)+\go(\mvz,\mvw)\Big).
	\end{align}
	Restrict now to paths with edges of length at most $\eps_1m$: by~\eqref{eq:hop-lengths} it suffices that $\max_{j\le N+1}Y_j\le\eps_1m/C_0$, which is the last constraint in the corridor event~\eqref{eq:corridor-event}.

	Condition now on a $\gs$-field $\cH$ with $\cF_V\subseteq\cH\subseteq\gs(\cF_V,\{\go_e:\norm{e}_\infty>\eps_1m\})$. All the edges appearing in~\eqref{eq:hop-cost}--\eqref{eq:last-cost} have length at most $\eps_1m$, they are pairwise distinct -- those in the $j$-th minimum join $\mvy_{j-1}$ to $\cD_j^a$, and the boxes $D_{I_j}$ are disjoint, while the edges in~\eqref{eq:last-step} join $\cD_N^a$ to the two distinct vertices $\mvy_{N-1}\notin D_{I_N}$ and $\mvw\notin\cD_N^a$ -- and none of them is $\cH$-measurable. Hence, conditionally on $\cH$, and using $\bigl|\cD_j^a\bigr|\ge\ell$,
	\begin{align*}
		\min_{\mvz\in\cD_j^a}\go(\mvy_{j-1},\mvz) \preceq \go_{j,[\ell]},
		\qquad
		\min_{\mvz\in\cD_N^a}\big(\go(\mvy_{N-1},\mvz)+\go(\mvz,\mvw)\big) \preceq \widehat\go_{[\ell]},
	\end{align*}
	where $\{\go_{j,[\ell]}\}_j$ are i.i.d.~copies of $\go_{[\ell]}$ and $\widehat\go_{[\ell]}$ is the minimum of $\ell$ i.i.d.~copies of $\go+\go'$, all independent of $\cH$. The domination and the independence follow, exactly as in Lemma~\ref{lem:hub-indep}, by revealing the minima one step at a time. By Lemma~\ref{lem:min-moments} both have finite moments of every order below $\ell\gz$.

	Combining~\eqref{eq:linear-path-decomp}--\eqref{eq:last-cost}, on the corridor event and conditionally on $\cH$,
	\begin{align*}
		T^{\eps_1}(\mvzero,\mvw) \preceq \frac{C_0^{\ga}}{a^2}\sum_{j=1}^{N+1}Y_j^{\ga} \xi_j,
	\end{align*}
	where $\xi_1,\dots,\xi_{N-1}$ are i.i.d.~copies of $\go_{[\ell]}$, $\xi_N=\xi_{N+1}=\widehat\go_{[\ell]}$ is independent of them (the last bracket in~\eqref{eq:linear-path-decomp} contributes the two terms $Y_N^{\ga}$ and $Y_{N+1}^{\ga}$ with the \emph{same} minimum, since $(Y_N\vee Y_{N+1})^{\ga}\le Y_N^{\ga}+Y_{N+1}^{\ga}$), so that $\{\xi_j\}_{j\le N+1}$ is a $1$-dependent family with finite moments of every order below $\ell\gz$, independent of $\cH$. The $Y_j$ are $\cH$-measurable. Writing $\mu:=\E\go_{[\ell]}\vee\E\widehat\go_{[\ell]}$ and $M_{p_0}:=\max_j\norm{\xi_j-\E\xi_j}_{p_0}$, and applying Lemma~\ref{lem:p-conS_k} with coefficients $a_j:=Y_j^{\ga}$,
		\begin{align*}
			\pr\Big(\sum_{j\le N+1}Y_j^{\ga}\xi_j>2\mu\Lambda_0m \,\Big|\, \cH\Big)
			 & \le\pr\Big(\sum_{j\le N+1}Y_j^{\ga}(\xi_j-\E\xi_j)>\mu\Lambda_0 m \,\Big|\, \cH\Big)                          \\
			 & \le C_{p_0}^{ p_0}\left(\frac{M_{p_0}\big(\sum_{j\le N+1}Y_j^{2\ga}\big)^{1/2}}{\mu\Lambda_0m}\right)^{\!p_0}
			\le C m^{-p_0/2}\le C m^{-q_0},
		\end{align*}
		where we used $\sum_{j\le N+1}Y_j^{\ga}\E\xi_j\le\mu\sum_{j\le N+1}Y_j^{2\ga}\le\mu\Lambda_0m$ on the corridor event for the centering (recall $Y_j\ge1$), and $\sum_{j\le N+1}Y_j^{2\ga}\le\Lambda_0m$ for the variance proxy. Setting $c_{\rm int}:=2\mu\Lambda_0C_0^{\ga}a^{-2}$ gives~\eqref{eq:lin-int}.

		\medskip
		\noindent\textbf{Proof of~\eqref{eq:lin-bdry}.}
		Now the endpoints are $\mvzero$ and $\mvw$ with possibly small weights, and we work on
	$\cE:=\{V_{\mvzero}\wedge V_{\mvw}\ge c_\ast m^{-\gb}\}$.
		Use the same construction with $\ell\gz>p$. The interior hops $j=2,\dots,N-1$ have both endpoints in $\cV^a$, so $W_{\mvy_{j-1}\mvy_j}\preceq Z_j:=(C_0Y_j)^{\ga}\go_{j,[\ell]}/a^2$, an i.i.d.~family with finite $p$-th moment. Writing $\mu_\ast:=\E Z_2$, $\gs_\ast^2:=\var(Z_2)$, $\nu_p:=\E\abs{Z_2-\mu_\ast}^p$ and $y:=2+\lceil2p_\square M\rceil$, a Chernoff bound gives $\pr(N>y)\le c_1e^{-c_2m}$, and Rosenthal's inequality gives, for $t>(y-1)\mu_\ast$,
		\begin{align*}
			\pr\Big(\sum_{j=2}^{N-1}W_{\mvy_{j-1}\mvy_j}>t\Big)
			\le c_1e^{-c_2m}+\frac{C_p\big(y\nu_p+(y\gs_\ast^2)^{p/2}\big)}{\big(t-(y-1)\mu_\ast\big)^p}.
		\end{align*}
		Choosing $b_0$ so large that $\tfrac{b_0}{3}m-(y-1)\mu_\ast\ge c'm$ and taking $t=\tfrac{b_0}{3}m$, and using $y\asymp m$, the second term is at most $C(m+m^{p/2})/m^p\le Cm^{-p/2}$ for $p\ge2$.

		When $N\ge2$, the first edge was selected by the one-sided minimization, and on
	$\cE$ we have $V_{\mvzero}^{-1}\le c_\ast^{-1}m^{\gb}$. Hence
	$W_{\mvzero\mvy_1}\ind_{\cE\cap\{N\ge2\}}\preceq Cm^{\gb}Y_1^{\ga}\go_{1,[\ell]}$
		and, since $\E Y_1^{\ga\gz}<\infty$ and $\E\go_{[\ell]}^{\gz}\le\E\go^{\gz}<\infty$,
		\begin{align*}
			\E\Big[\big(W_{\mvzero\mvy_1}\ind_{\cE\cap\{N\ge2\}}\big)^{\gz}\Big]\le Cm^{\gb\gz}.
		\end{align*}
		For the final two-hop bracket, neither noise is a fresh copy of $\go$, because both entered the minimization~\eqref{eq:last-step}. Nevertheless, each is bounded by the minimized sum, and, conditionally on $\cF_V$ and on the noises used to construct $\mvy_{N-1}$,
		\begin{align*}
			\go(\mvy_{N-1},\mvy_N)+\go(\mvy_N,\mvw)
			\preceq\widehat\go_{[\ell]},
		\end{align*}
		which has a finite $\gz$-th moment by Lemma~\ref{lem:min-moments}. On $\cE$, the two denominators are bounded below by $a^2$ and $ac_\ast m^{-\gb}$, respectively. Since $Y_{N+1}\le Y_N$,
		\begin{align*}
			\E\Big[\big((W_{\mvy_{N-1}\mvy_N}+W_{\mvy_N\mvw})\ind_{\cE}\big)^{\gz}\Big]
			\le C m^{\gb\gz}.
		\end{align*}
		If $N=1$, this final bracket already contains the edge from $\mvzero$ and no
		separate first-edge term is needed. By Markov's inequality with
	$t=\tfrac{b_0}{3}m$, the first-edge term and the final two-hop bracket each exceed
	$\tfrac{b_0}{3}m$ on $\cE$ with probability at most
	$Cm^{\gb\gz}m^{-\gz}=Cm^{-\gz(1-\gb)}$. Summing the three contributions,
		\begin{align*}
			\pr\Big(T(\mvzero,\mvw)\ge b_0m, \cE\Big)
			\le Cm^{-\gz(1-\gb)}+c_1e^{-c_2m}+Cm^{-p/2}
			\le Cm^{-\gd},
			\quad
			\gd=\min\{p/2,\gz(1-\gb)\},
		\end{align*}
		for all large $m$, which is~\eqref{eq:lin-bdry}.
\end{proof}

\begin{proof}[Proof of Theorem~\ref{thm:linear-ub}]
	We may assume $\mvx\ne\mvzero$. Set $\mvu_n:=\lceil n\mvx\rceil$, so $\norm{\mvu_n}\to\infty$ and $\norm{\mvu_n}\le C_{\mvx}n$ for large $n$. Apply Proposition~\ref{prop:lin-ub-input}(ii) with $p=2$ and $\gb=\tfrac12$, obtaining $b_0,c_\ast$, and put $\cE_n:=\{V_{\mvzero}\wedge V_{\mvu_n}\ge c_\ast\norm{\mvu_n}^{-1/2}\}$ and $C:=b_0C_{\mvx}$. Then for large $n$,
	\begin{align*}
		\pr(T_n>Cn)
		\le\pr\big(T(\mvzero,\mvu_n)\ge b_0\norm{\mvu_n}, \cE_n\big)+\pr(\cE_n^{\rm c}).
	\end{align*}
	The first term is at most $C\norm{\mvu_n}^{-\gd}\to0$ with $\gd=\min\{1,\gz/2\}>0$. For the second,
	$\pr(\cE_n^{\rm c})\le2\pr\big(V<c_\ast\norm{\mvu_n}^{-1/2}\big)\to0$
	since $\norm{\mvu_n}\to\infty$ and $\pr(V>0)=1$. Hence $\pr(T_n>Cn)\to0$.
\end{proof}

\section{Discussion and Open Problems}\label{sec:discussion}

The principal results proved here are upper bounds. A companion manuscript~\cite{CDK26} develops complementary lower bounds using renormalization and a differential inequality for volume growth. Theorems~\ref{thm:II-lb} and~\ref{thm:II-hop} are the cases in which a first-moment argument already provides the needed lower information here. We next discuss several open problems related to the model.

\subsection{Boundaries}
We have said nothing about the six phase boundaries
\begin{align*}
	\ga\theta=d,\quad\ga\theta=2d,\quad\ga\theta=2d+\theta,
	\qquad
	\ga\gc=d,\quad\ga\gc=2d,\quad\ga\gc=2d+\gc,
\end{align*}
nor about the transition curves along which the two mechanisms exchange dominance. At a critical value, slowly varying corrections in the cost function or in either tail can affect the answer and must be retained. This is analogous to the delicate critical case $s=2d$ in long-range percolation resolved by Ding and Sly~\cite{DingSly2013Critical}, and by B\"aumler~\cite{Baeumler2022CriticalAllD}.

Two phase boundaries are of particular interest due to the competition between distinct routing mechanisms.
First, for $\theta<\gc<2\theta$, the line $\ga\theta=d$ separates Regimes $\rm I_E$ and II. The lack of a strictly dominant mechanism here suggests a delicate structural transition.
Second, the curve $\ga=2d/\gc$ separates Regime II from Regimes III and $\rm IV_H$. On this curve, the hub one-step exponent is zero, while the edge one-step exponent is $2d(1/\gc-1/\theta)$. Because the costs of hub chains and binary trees balance at this transition, optimal paths are expected to exhibit a mixed geometry.


\subsection{Sharp asymptotics}
Beyond the estimates proved here, precise normalizations, constants, and limiting laws remain open. Three natural questions arise.
\begin{enumeratei}
	\item In Regime II, does $T_n$ converge in distribution? Proposition~\ref{prop:II-ball} shows that the metric explodes at an a.s.\ finite time $\tau_\infty$. The hub-chain representation heuristically suggests a connection to the passage time of a bi-infinite hub chain. However, a limiting description must incorporate not only the point process of large vertex weights but also the edge noises and endpoint environments. Identifying this limit law and the law of $\tau_\infty$ would be a natural next step.

	\item In Regime III, is $(\log n)^{-\Delta_{\rm III}}T_n$ tight, and does it oscillate? In long-range percolation, arithmetic oscillations arise because the depth $k$ of the relevant multiscale structure is an integer~\cite{BiskupKrieger2021Oscillations}. The binary tree construction suggests a similar mechanism, but establishing $\log\log n$-periodic oscillations for the passage metric likely requires finer tail assumptions than \eqref{A3}--\eqref{A4}.

	\item In Regimes $\rm V_E$ and $\rm V_H$, does an almost-sure shape theorem hold, \ie\ is there a norm $\nu$ on $\dR^d$ with $n^{-1}T(\mvzero,\lceil n\mvx\rceil)\to\nu(\mvx)$ a.s.? Subadditivity guarantees a deterministic finite limit along fixed rays under stronger moment conditions (e.g., on $V^{-1}$ and $\go$). However, establishing a full shape theorem requires further control to ensure uniform convergence. If a limit shape exists, the simulations of Figure~\ref{fig:sim-V} suggest it is close to a Euclidean ball in $\rm V_E$ but visibly distorted in $\rm V_H$; it would be interesting to determine whether it can fail to be strictly convex.
\end{enumeratei}

\subsection{Hop count and geodesics}
Theorem~\ref{thm:IH-hop} determines the hop-count profile in Regime $\rm I_H$. In Regime II, Theorem~\ref{thm:II-hop} gives a $\log\log n$ lower bound for bounded-cost paths, while the hub-chain construction gives paths with $O(\log\log n)$ edges and tight passage times. These bounds do not determine the hop count of near-minimizing paths.

For fixed $\eps>0$, determine the order and fluctuations of $H_{\lceil n\mvx\rceil}\big((1+\eps)T_n\big)$ in Regime II. In Regimes III--V, the corresponding question concerns existence of geodesics and $H_{\lceil n\mvx\rceil}(T_n)$, the minimum hop count among geodesics.
Further questions concern the length of the longest edge and, for $d\ge2$, the transverse fluctuations of geodesics. The paths constructed for our upper bounds need not describe the geometry of minimizing paths.

\subsection{Degenerate cases and generalizations}
A natural degenerate case arises when the edge noise is bounded away from zero: $\pr(\go\ge c_\go)=1$ for some $c_\go>0$. This class includes the pure vertex model $\go\equiv1$. Since arbitrarily small edge noises are unavailable, the small-noise edge mechanism vanishes. Formally setting $\theta=\infty$ yields $q_{\rm edge}=0$, and the phase diagram reduces to the four hub regimes $\rm I_H$, II, $\rm IV_H$, and $\rm V_H$.
In this setting, the one-edge tail behavior is entirely governed by the vertex weights. Since $W_{\mvu\mvv} \ge c_\go \norm{\mvu-\mvv}^\ga / (V_{\mvu} V_{\mvv})$, choosing $\eta<\gc$ with $\E V^{\eta}<\infty$ yields
$$ \pr\big(W_{\mvu\mvv}\le t\big) \le \pr\Big(tV_{\mvu}V_{\mvv}\ge c_\go{\norm{\mvu-\mvv}^{\ga}}\Big) \le c_\go^{-\eta}{(\E V^{\eta})^2 t^{\eta}}{\norm{\mvu-\mvv}^{-\ga\eta}}, $$
which isolates the polynomial tail estimate without relying on the edge noise. Symmetrically, $V\equiv1$ recovers long-range first-passage percolation. Several probabilistic estimates use only polynomial volume growth of exponent $d$, but the present path constructions also rely on the Euclidean placement and separation of balls, shells, and corridors. Extending the theorems to vertex-transitive graphs of polynomial growth therefore requires a geometric replacement for those constructions rather than a verbatim transfer. A random underlying point process, for instance a Poisson process in $\dR^d$, would additionally require concentration estimates for the number of points in the relevant regimes.

\subsection{Finite $d$-dimensional tori and other cost functions}
\label{ssec:finite-tori-kernels}

On the discrete torus $\dT_n^d:=(\dZ/n\dZ)^d$, with $n^d$ vertices and periodic $\ell_1$ distance $\dist_n$, define
\begin{align*}
	W^{(n)}_{\mvu\mvv}
	:=\cost_n\bigl(\dist_n(\mvu,\mvv)\bigr)
	\frac{\go_{\mvu\mvv}}{V_{\mvu}V_{\mvv}},
	\quad \mvu\ne\mvv.
\end{align*}
The choice $\cost_n(r)=r^\ga$ gives the finite-torus version of the present model. A finite-range version is obtained by taking
\begin{align*}
	\cost_n(r)=
	\begin{cases}
		1,      & 1\le r\le\ell_n, \\
		\infty, & r>\ell_n,
	\end{cases}
\end{align*}
where $\ell_n\ge1$ may grow with $n$; replacing $1$ by $r^\ga$ gives a truncated power-law cost. For $V\equiv1$, exponential edge noises and $0\le\ga<d$, the power-law case is studied in~\cite{vdHL24}.
For $d=1$, $V\equiv1$ and Weibull edge noises, the finite-range cycle model is studied in~\cite{DK26b}. The fixed-range line model is treated separately in~\cite{DK26a}.

Determine the normalizations and limiting laws of the passage time between two uniformly sampled vertices, the flooding time, and the diameter as $n\to\infty$.
How do the vertex weights affect the transition between spatial and mean-field behavior as $\ell_n$ grows?
Which additional tail assumptions are needed for the flooding time and diameter, as opposed to a typical passage time?
We plan to study finite-torus ILRFPP and its spread-out extensions in subsequent work.

\subsection{Two mechanisms, one metric}
We close by emphasizing what we regard as the conceptual point. The two exceptional mechanisms in this model -- a heavy vertex and a cheap edge -- produce thresholds by identical extreme-value computations, and one might reasonably expect them to produce identical phase diagrams with $\gc$ and $\theta$ interchanged. They do not, because a discount attached to a vertex can be used twice while a discount attached to an edge can be used once. That asymmetry is invisible at the level of the one-bridge cost and appears only when one constructs complete paths: the hub chain proves $\gTh_{\pr}(1)$ in Regime II, whereas the binary bridge construction yields the upper bound $O_{\pr}((\log n)^{\Delta_{\rm III}+\eps})$ for every $\eps>0$ in Regime III.

\medskip\noindent\textbf{Acknowledgements.}
This work originated in the Fall semester of 2023 as part of an \href{https://iml.math.illinois.edu/research/undergraduate-research-projects/}{Illinois Mathematics Lab (IML) project on \emph{Scale Free First Passage Percolation}}.
The project was recognized as a runner-up for the 2024 IML Research Award.
We thank the undergraduate scholars Haotian Ju, Hansheng Liu, Linzhe Teng, and Shuzhen Zhang for conducting simulations during the early stages of the project and the Department of Mathematics at the University of Illinois Urbana--Champaign for providing the opportunity. 

\medskip\noindent\textbf{Declaration on the use of AI.}
In the spirit of the transparency recommendations in the Leiden Declaration on Artificial Intelligence and Mathematics~\cite{LeidenDeclaration2026}, large language models were used as assistive tools for grammar and language editing, consistency checks of notation, literature searches and the identification of potentially relevant references, and assistance with the development and checking of simulation code. All mathematical statements, proofs, citations, references, and final verification were written and reviewed by the authors, who take full responsibility for the content of the manuscript.

\bibliographystyle{alphaurl}
\bibliography{sflrfpp}
\end{document}

%% file: style.tex
\usepackage{systeme}
\usepackage{mathrsfs,xfrac} 
\usepackage[colorlinks=true,linkcolor=blue,backref,citecolor=blue,urlcolor=blue]{hyperref} 
\usepackage{amsmath,amssymb,amsthm,amsfonts,amsbsy,latexsym,dsfont,color} 
\usepackage[foot]{amsaddr}
\usepackage{enumitem}

\usepackage[utf8]{inputenc}
\usepackage[english]{babel}
\usepackage{comment}

\usepackage[textsize=tiny]{todonotes}
\usepackage{regexpatch}
\makeatletter
\xpatchcmd{\@todo}{\setkeys{todonotes}{#1}}{\setkeys{todonotes}{inline,#1}}{}{}
\makeatother

\newenvironment{enumeratei}{\begin{enumerate}[label=\textup{\roman*.},leftmargin=2em,itemsep=1ex]}{\end{enumerate}}
\newenvironment{enumeratea}{\begin{enumerate}[label=\textup{\alph*)},leftmargin=2em,itemsep=1ex]}{\end{enumerate}}

\newtheorem{thm}{Theorem}[section]
\newtheorem{lem}[thm]{Lemma}

\newtheorem{prop}[thm]{Proposition}
\newtheorem{defn}[thm]{Definition}
\newtheorem{rem}[thm]{Remark}

\renewcommand{\le}{\leqslant} 
\renewcommand{\ge}{\geqslant} 
 
\renewcommand{\geq}{\geqslant} 

\newcommand{\ra}{\rangle}
\newcommand{\la}{\langle} 
\newcommand{\wt}{\widetilde}

\newcommand{\ind}{\mathds{1}}
\newcommand{\eps}{\varepsilon}

\newcommand{\norm}[1]{\left\Vert#1\right\Vert}
\newcommand{\abs}[1]{\left\vert#1\right\vert}

\newcommand{\ie}{\emph{i.e.,}}

\let\ga=\alpha \let\gb=\beta \let\gc=\gamma \let\gd=\delta 
    \let\gk=\kappa \let\gl=\lambda        \let\go=\omega   \let\gs=\sigma  
  \let\gz=\zeta
\let\gC=\Gamma \let\gD=\Delta   \let\gTh=\Theta
\let\gO=\Omega           
                                
\newcommand{\cA}{\mathcal{A}}\newcommand{\cB}{\mathcal{B}}\newcommand{\cC}{\mathcal{C}}
\newcommand{\cD}{\mathcal{D}}\newcommand{\cE}{\mathcal{E}}\newcommand{\cF}{\mathcal{F}}
\newcommand{\cG}{\mathcal{G}}\newcommand{\cH}{\mathcal{H}}\newcommand{\cI}{\mathcal{I}}
\newcommand{\cL}{\mathcal{L}}

\newcommand{\cP}{\mathcal{P}}

\newcommand{\cV}{\mathcal{V}}\newcommand{\cW}{\mathcal{W}}
  
\newcommand{\vB}{\mathbf{B}}
\newcommand{\vD}{\mathbf{D}}

\newcommand{\mvzero}{\boldsymbol{0}}

\newcommand{\mva}{\boldsymbol{a}}

\newcommand{\mvm}{\boldsymbol{m}}
\newcommand{\mvp}{\boldsymbol{p}}
\newcommand{\mvr}{\boldsymbol{r}}
\newcommand{\mvu}{\boldsymbol{u}}\newcommand{\mvv}{\boldsymbol{v}}
\newcommand{\mvw}{\boldsymbol{w}}\newcommand{\mvx}{\boldsymbol{x}}\newcommand{\mvy}{\boldsymbol{y}}
\newcommand{\mvz}{\boldsymbol{z}}

\newcommand{\mveta}{\boldsymbol{\eta}}

\newcommand{\mvpi}{\boldsymbol{\pi}}

\newcommand{\fB}{\mathfrak{B}}

\newcommand{\dA}{\mathds{A}}\newcommand{\dB}{\mathds{B}}
\newcommand{\dC}{\mathds{C}}

\newcommand{\dN}{\mathds{N}}

\newcommand{\dR}{\mathds{R}}
\newcommand{\dT}{\mathds{T}}

\newcommand{\dZ}{\mathds{Z}} 
\newcommand{\rA}{\mathrm{A}}\newcommand{\rB}{\mathrm{B}}\newcommand{\rC}{\mathrm{C}}

\newcommand{\rO}{\mathrm{O}}

\newcommand{\sE}{\mathscr{E}}

\DeclareMathOperator{\E}{\mathds{E}}
\DeclareMathOperator{\pr}{\mathds{P}}

\DeclareMathOperator{\var}{Var}

\DeclareMathOperator{\argmin}{argmin}

\newcommand{\hV}{\widehat{V}}

\newcommand{\cost}{\mathrm{cost}}